\documentclass[11pt, a4paper, abstract=on, bibliography=totoc]{scrartcl}

\usepackage[T1]{fontenc}
\usepackage[utf8]{inputenc}
\usepackage[english]{babel}

\usepackage{amsmath}
\usepackage{amssymb}
\usepackage{amsthm}

\usepackage{dsfont}
\usepackage{booktabs}

\usepackage{tikz}
\usetikzlibrary{arrows.meta, calc}
\usepackage{tikz-cd}

\usepackage{xcolor}
\definecolor{linknavy}{rgb}{0.0,0.0,0.55}

\usepackage{hyperref}
\hypersetup{
    colorlinks = true,
    linkcolor  = linknavy,
    citecolor  = linknavy,
    urlcolor   = linknavy,
    pdfpagemode = UseOutlines,
    pdfencoding = unicode,
    pdftitle   = {Sample-Smooth Spaces: A Convenient Category for
                  Differentiable Probabilistic Programming},
    pdfauthor  = {Patrick Forre},
    pdfsubject = {Category theory, probability monads, differentiable
                  probabilistic programming},
    pdfkeywords = {sample-smooth spaces, quasi-universal spaces, concrete site,
                  quasitopos, cohesion, probability monad, Markov category,
                  reparametrisation gradient, differentiable programming,
                  probabilistic programming}}

\usepackage{aliascnt}
\usepackage{microtype}
\usepackage{cleveref}

\theoremstyle{plain}

\newaliascnt{Thm}{mycounter}
\newtheorem{Thm}[Thm]{Theorem}
\aliascntresetthe{Thm}
\newaliascnt{Lem}{mycounter}
\newtheorem{Lem}[Lem]{Lemma}
\aliascntresetthe{Lem}
\newaliascnt{Prp}{mycounter}
\newtheorem{Prp}[Prp]{Proposition}
\aliascntresetthe{Prp}
\newaliascnt{Cor}{mycounter}
\newtheorem{Cor}[Cor]{Corollary}
\aliascntresetthe{Cor}

\theoremstyle{definition}
\newaliascnt{Def}{mycounter}
\newtheorem{Def}[Def]{Definition}
\aliascntresetthe{Def}
\newaliascnt{Not}{mycounter}
\newtheorem{Not}[Not]{Notation}
\aliascntresetthe{Not}

\theoremstyle{remark}
\newaliascnt{Rem}{mycounter}
\newtheorem{Rem}[Rem]{Remark}
\aliascntresetthe{Rem}
\newaliascnt{Eg}{mycounter}
\newtheorem{Eg}[Eg]{Example}
\aliascntresetthe{Eg}

\crefname{Thm}{Theorem}{Theorems}
\Crefname{Thm}{Theorem}{Theorems}
\crefname{Lem}{Lemma}{Lemmas}
\Crefname{Lem}{Lemma}{Lemmas}
\crefname{Prp}{Proposition}{Propositions}
\Crefname{Prp}{Proposition}{Propositions}
\crefname{Cor}{Corollary}{Corollaries}
\Crefname{Cor}{Corollary}{Corollaries}
\crefname{Def}{Definition}{Definitions}
\Crefname{Def}{Definition}{Definitions}
\crefname{Not}{Notation}{Notations}
\Crefname{Not}{Notation}{Notations}
\crefname{Rem}{Remark}{Remarks}
\Crefname{Rem}{Remark}{Remarks}
\crefname{Eg}{Example}{Examples}
\Crefname{Eg}{Example}{Examples}

\newcommand{\N}{\mathbb{N}}
\newcommand{\R}{\mathbb{R}}
\newcommand{\Q}{\mathbb{Q}}
\newcommand{\id}{\mathrm{id}}
\newcommand{\Acal}{\mathcal{A}}
\newcommand{\Bcal}{\mathcal{B}}
\newcommand{\Ccal}{\mathcal{C}}
\newcommand{\Ecal}{\mathcal{E}}
\newcommand{\Ncal}{\mathcal{N}}
\newcommand{\Nrm}{\mathrm{N}}
\newcommand{\Fnorm}{F_{\mathrm{N}}}
\newcommand{\Pcal}{\mathcal{P}}
\newcommand{\lA}{\left\langle}
\newcommand{\rA}{\right\rangle}
\newcommand{\lV}{\left\| }
\newcommand{\rV}{\right\| }
\newcommand{\lI}{\left| }
\newcommand{\rI}{\right| }
\newcommand{\Scal}{\mathcal{S}}
\newcommand{\Tcal}{\mathcal{T}}
\newcommand{\Wcal}{\mathcal{W}}
\newcommand{\Xcal}{\mathcal{X}}
\newcommand{\Ycal}{\mathcal{Y}}
\newcommand{\Zcal}{\mathcal{Z}}

\newcommand{\PrPf}{\mathsf{P}}
\newcommand{\PrPfO}{\PrPf_\Omega}
\newcommand{\uc}{\mathsf{u}}
\newcommand{\Mbb}{\mathbb{M}}

\newcommand{\ihom}{{\mathcal{SSS}}}

\newcommand{\cur}{\mathrm{cur}}
\newcommand{\unc}{\mathrm{unc}}

\newcommand{\ubar}[1]{\text{\b{$#1$}}}
\newcommand{\twob}{{\ubar{\mathbf{2}}}}
\newcommand{\twoc}{\breve{\mathbf{2}}}
\newcommand{\twop}{{\bar{\mathbf{2}}}}
\newcommand{\one}{\mathbf{1}}
\newcommand{\zero}{\mathbf{0}}

\newcommand{\sm}{\setminus}
\newcommand{\ins}{\subseteq}
\newcommand{\sni}{\supseteq}

\newcommand{\pr}{\mathrm{pr}}
\newcommand{\ev}{\mathrm{ev}}
\newcommand{\Kleisli}{\mathrm{Kl}}

\newcommand{\inj}{\hookrightarrow}
\newcommand{\bij}{\stackrel{\sim}{\longrightarrow}}

\newcommand{\E}{\mathbb{E}}
\newcommand{\I}{\mathds{1}}

\newcommand{\QUS}{\mathbf{QUS}}
\newcommand{\Und}{\natural}
\newcommand{\SSS}{\mathbf{SSS}}
\newcommand{\Diff}{\mathbf{Diff}}
\newcommand{\Meas}{\mathbf{Meas}}
\newcommand{\Sets}{\mathbf{Sets}}

\newcommand{\lp}{\left ( }
\newcommand{\rp}{\right ) }
\newcommand{\lb}{\left \langle }
\newcommand{\rb}{\right \rangle }
\newcommand{\lB}{\left [ }
\newcommand{\rB}{\right ] }
\newcommand{\lC}{\left \{ }
\newcommand{\rC}{\right \} }
\newcommand{\st}{\,\middle|\,}

\providecommand{\doi}[1]{\href{https://doi.org/#1}{\textsf{doi:#1}}}
\providecommand{\eprintlink}[1]{\href{https://arxiv.org/abs/#1}{\textsf{arXiv:#1}}}

\newcommand{\Site}{\mathbb{S}}
\newcommand{\Cmix}{\Ccal^\infty_{\mathrm{mix}}}
\newcommand{\pt}{\ast}
\newcommand{\CPsh}{\mathbf{cPsh}}
\newcommand{\Law}{\mathsf{L}}

\begin{document}

\begin{titlepage}

\title{Sample-Smooth Spaces}
\subtitle{\Large A Convenient Category for\\Differentiable Probabilistic Programming}
\author{Patrick Forr\'e\\[6pt]
    {\small AI4Science Lab}\\
    {\small Korteweg-de Vries Institute for Mathematics}\\
    {\small University of Amsterdam}\\
    {\small \texttt{p.d.forre@uva.nl}}}
\date{}
\maketitle
\thispagestyle{empty}

\begin{abstract}
We introduce the category $\SSS$ of \emph{sample-smooth spaces} over a mixed site. The test objects
are the products $\Omega_n := \R^n \times \Omega$ of a Cartesian space with the universal Hilbert
cube $\Omega$ carrying all universally measurable sets, and a space is a set with a family of
admissible \emph{plots} $\Omega_n \to \Xcal$ closed under precomposition. Smoothness and
measurability are then not two structures glued along an axiom, but one structure indexed by one
site. The site has finite non-empty products, because $\Omega$ absorbs its own square; its Karoubi
envelope already contains every $\R^n$; and it has \emph{mixed} morphisms
$\omega \mapsto (W(\omega),\Phi(\omega))$, which turn measurability of a smooth family from an
axiom into a consequence.

$\SSS$ is a concrete quasitopos: complete, cocomplete, cartesian closed and locally cartesian
closed, with a classifier for embeddings. Morphisms of Cartesian spaces are exactly the $C^\infty$
maps and manifolds embed full and faithfully, both without Boman's theorem. Every object has tangent
and cotangent spaces and every morphism a differential. The modalities sit in an adjoint string
$\Pi \dashv \flat \dashv \Und \dashv \sharp \dashv \Lambda$, making $\SSS$ cohesive over
$\QUS$.

The point is the probability monad. Defining the plots of $\PrPf(\Xcal)$ as push-forwards of
$\Xcal$-plots at \emph{every} test object, $\PrPf$ is an unconditional strong commutative affine
monad on all of $\SSS$ --- functor, unit, the product of kernels, the multiplication and the monad
laws are each one line of seed splitting --- and its Kleisli category, the category of differentiable
simulators, is a Markov category. The reparametrisation trick holds by construction: every Kleisli
morphism is a sampler plot-wise, stably under composition, which is what the closing section
differentiates. A reflection theorem locates the whole gain in a single plot family.
\end{abstract}



\DeclareTOCStyleEntries[numwidth=2.9em]{tocline}{subsection}
\DeclareTOCStyleEntries[numwidth=3.6em]{tocline}{subsubsection}
{
\tableofcontents
}
\end{titlepage}

\section{Introduction}
\label{sec:intro}

A differentiable probabilistic program manipulates two kinds of value at once. Some are
differentiated --- parameters, losses, the output of a network --- and some are sampled --- latent
variables, noise, observations. A semantics for such programs therefore wants a category that is
\emph{convenient} in the sense of \cite{Ste67}: cartesian closed, with enough limits and colimits to
interpret the type formers; \emph{correct}, in that its morphisms between Cartesian spaces are the
$C^\infty$ maps and finite-dimensional manifolds embed; and \emph{probabilistic}, in that it carries
a probability monad whose Kleisli category is a Markov category in the sense of \cite{Fri20}.

\subsection*{What such a category is for}

The reason to want one category rather than two is that in a differentiable probabilistic language
the two kinds of value are not kept apart by a phase distinction. A sampler's parameters are
differentiated, a loss is an expectation over samples, and a higher-order program may take a sampler
as an argument and return another one. A category with the three properties above supplies, in a
single place, most of what a denotational semantics of such a language needs. Cartesian closure
interprets function types, so that samplers are first-class values and the higher-order combinators
of a practical system --- an optimiser that consumes a model, a variational family parametrised by a
network, a transformation sending a program to a batched version of itself --- denote morphisms
rather than remaining syntax. Local cartesian closure supplies the dependent products,
\Cref{sss:prp:internal-hom}, and finite limits the dependent sums, \Cref{sss:prp:fibre}: the
structure a model of dependent type theory is built from, once the usual coherence for substitution
is arranged. Completeness and cocompleteness interpret products and sums, with the proviso of
\Cref{sss:rem:no-jump} that a plot into a coproduct stays in one summand, so that a randomly chosen
branch is typed by a modality rather than by a sum, \Cref{ml:sec:nonsmooth}. The classifier of
\Cref{sss:prp:classifier} interprets a predicate on a type, in the shape of the subsets carrying the
subspace structure; a quasitopos differs from a topos precisely in that an injection may present its
domain with a \emph{finer} structure, fewer admissible plots than it inherits, and those
monomorphisms are not classified, \Cref{sss:cor:not-topos}. The monad is the language's
\textsf{sample}, \textsf{return} and \textsf{bind}; its Kleisli category is the category of
programs-with-randomness, and the Markov structure records two of the equational laws a compiler
appeals to when it rewrites one --- an unused sample may be discarded (affineness), independent
samples may be drawn in either order (commutativity). Finally, every object has tangent and
cotangent spaces and every morphism a differential, with no hypothesis on either,
\Cref{sss:sec:tangent}, and on Cartesian spaces and manifolds these are the classical ones; so
differentiation too is an operation on morphisms rather than a syntactic transformation defined on a
well-behaved fragment.

What the list does not contain is worth saying at once. Expectation is not a morphism: the map
$\mu\mapsto\int h\,d\mu$ fails to be sample-smooth already for $h = \cos$ on the real line,
\Cref{prob:thm:expectation}, so the loss of a variational program is not itself interpreted by the
structure that interprets the program. Conditioning is not addressed either; \Cref{prob:thm:markov}
gives a Markov category and nothing beyond it, and conditionals are left to a sequel,
\Cref{prob:rem:conditionals}. In general the tangent space is a cone and not a vector space,
\Cref{tan:prp:cotangent}. These are stated where they arise, and collected in
\Cref{disc:sec:status}.

For a reader coming from machine learning the shortest description is this. A Kleisli morphism
$\R^m\to\PrPf(\Ycal)$ \emph{is} a differentiable sampler: a program consuming a parameter vector
and a stream of randomness and returning a value, taken up to equality of law at each parameter
value, \Cref{prob:thm:representation,prob:cor:sampler-bijection}. On this reading the
reparametrisation trick is not a technique that happens to apply to certain distributions; it is the
definition of the morphisms. And because the seed is split rather than re-derived, it survives
composition: a stochastic computation graph is a diagram in this category, and the composite of two
differentiable samplers is a third one that gets written down rather than argued for,
\Cref{ml:eg:composite}. What does not survive composition is the integrability hypothesis on the
sampler under which one may differentiate under the integral sign: it is the one assumption in the
paper that composition does not preserve, it is needed for the pathwise gradient of
\Cref{ml:thm:pathwise}, and \Cref{ml:eg:composite} shows it failing for a composite of two stages
that each satisfy it.

\subsection*{The problem}

Each half is settled on its own. On the measurable side the quasi-Borel spaces of \cite{HKSY17} and
the quasi-universal spaces of \cite{For21} give a quasitopos with a push-forward probability monad;
on the smooth side the diffeological spaces of \cite{Sou80,IZ13} give a quasitopos containing
manifolds, with the Fr\"olicher variants and their comparison in \cite{BKW25}. The difficulty is
the combination, and it is sharper than one expects.

The obvious way to combine them is to carry both families of admissible maps. An object is a set
$\Xcal$ together with admissible random variables $\Omega\to\Xcal$ out of a sample space and
admissible smooth curves $\R\to\Xcal$ out of the line, subject to one compatibility axiom, that a
smooth curve read at a measurable time is an admissible random variable. Everything one wants of the
base category follows. One then puts $\PrPf(\Xcal) := \lC X_*P\rC$ and calls a curve of measures
smooth when it is \emph{reparametrisable}, $\mu_t = (X_t)_*U$ for a smooth family $X$ of random
variables. Functor, unit, the product of measures and the strength are all routine.

The multiplication is not. To say that $\Mbb(\Gamma) := \int\nu\,\Gamma(d\nu)$ is a morphism one must show
that for a family $\Ncal$ of measures, indexed by a smooth parameter $t$ and a seed $\omega$, the
mixture $t\mapsto\int_\Omega\Ncal(t,\omega)\,U(d\omega)$ is again reparametrisable. What one is
given about $\Ncal$ is
\begin{enumerate}
    \item[(a)] for each $\omega$, the curve $t\mapsto\Ncal(t,\omega)$ is reparametrisable, and
    \item[(b)] for each measurable substitution $\omega\mapsto\lp W(\omega),\Phi(\omega)\rp$, the
        map $\omega\mapsto\Ncal\lp W(\omega),\Phi(\omega)\rp$ is an admissible random measure.
\end{enumerate}
What one needs is
\begin{enumerate}
    \item[(c)] a single $Y$, smooth in $t$ and admissible in the seed, with
        $\Ncal(t,\omega) = \lp Y(t,\omega,-)\rp_*U$,
\end{enumerate}
and (a) supplies a representative for each $\omega$
separately, with no coherence between them. Producing one representative from all of them is a
measurable selection problem: choose $Y^\omega$ measurably in $\omega$ and smoothly in $t$ at once.
Selections of Jankov--von Neumann type return universally measurable functions, which the sample
space here admits, but they give no control in the $t$-direction, which is the whole difficulty; the
problem is open.

\subsection*{The diagnosis, and the fix}

The obstruction is not a defect of the sample space and not a defect of the smooth side. It is that
in a presentation by two families the \emph{mixed} data --- families that are smooth in a parameter
and admissible in a seed at the same time --- is not primitive. Given $\Xcal^\Omega$ and the curves,
the mixed families are \emph{forced} to be the internal hom, which is exactly (a) and (b); whereas
the probability monad needs them to be push-forward images, which is the existence of $Y$. The two
notions need not agree, and the multiplication is available only where they do.

So we make the mixed data primitive. The test objects of this paper are the products
\begin{align}
    \Omega_n &:= \R^n \times \Omega, \qquad n \ge 0,
\end{align}
of a Cartesian space with the sample space, and a \emph{sample-smooth space} is a set together with,
for each $n$, a family of admissible \emph{plots} $\Omega_n\to\Xcal$, closed under precomposition.
Three properties of this site make it work, and they are the content of \Cref{sec:site}. It has
finite non-empty products, because $\Omega$ absorbs its own square. Its Karoubi envelope already contains every
$\R^n$ and the one-point space, so nothing is lost by leaving purely smooth test objects out. And it
has \emph{mixed} morphisms, in which the smooth coordinate is read off the seed; these have no
counterpart when the two directions are kept apart, and they are what turns the compatibility axiom
above from an axiom into a consequence, \Cref{sss:lem:Q4}.

The definition of $\PrPf(\Xcal)$ is then the one the monad wants: its plots at \emph{every} test
object are the push-forwards of $\Xcal$-plots, \Cref{prob:def:PrPf}. The selection problem does not
arise, because nothing has to be selected.

\subsection*{What is gained, and what it costs}

$\SSS$ is a concrete quasitopos --- complete, cocomplete, cartesian closed and locally cartesian
closed, with an explicit classifier for embeddings, \Cref{sss:thm:quasitopos}. Morphisms between
Cartesian spaces are exactly the $C^\infty$ maps and manifolds embed full and faithfully,
\Cref{sss:thm:euclidean,sss:thm:manifolds}; both are immediate, because the site contains
a plot that is the identity in disguise, so Boman's theorem \cite{Bom67} disappears from the
foundations. Every object has tangent and cotangent spaces, and on Cartesian spaces and manifolds the tangent
bundle, the differential and the Jacobian are the classical ones, \Cref{sss:sec:tangent}. The rigid and loose modalities sit in an
adjoint string $\Pi \dashv \flat \dashv \Und \dashv \sharp \dashv \Lambda$ making $\SSS$ cohesive
over $\QUS$ in the sense of \cite{Law07}, \Cref{modal:thm:string}.

The two results the paper is written around are in \Cref{sec:prob}, and neither carries a hypothesis
on the objects:
\begin{center}
\nopagebreak
\renewcommand{\arraystretch}{1.3}
\begin{tabular}{@{}l p{0.68\textwidth}@{}}
    \Cref{prob:thm:monad} & $\PrPf$ is a strong commutative affine monad on \emph{all} of $\SSS$ \\
    \Cref{prob:thm:markov} & its Kleisli category is a Markov category, on \emph{all} of $\SSS$
\end{tabular}
\end{center}
Every piece of the structure --- functor, unit, the product of kernels dependent as well as
independent, the strength, the multiplication, the monad laws --- is one line of seed splitting.
Alongside them, \Cref{prob:thm:representation,prob:cor:sampler-bijection} identify the
Kleisli morphisms out of $\R^m$ with samplers modulo equality in law, so that the reparametrisation
trick of the machine-learning literature holds by construction and is stable under composition; that
is what \Cref{sec:ml} then differentiates.

The price is stated plainly in \Cref{sss:rem:curves-do-not-suffice,prob:rem:comparison}. An object of $\SSS$ is not \emph{presented} by its curves and its
random variables: the mixed plot family is separate data, constrained only by the inclusion
$\Xcal^{\Omega_n}\ins\widehat\Xcal^{\Omega_n}$ into the largest compatible choice. That is precisely
the room in which $\PrPf$ is defined --- and whether the inclusion is ever strict is itself open,
\Cref{sss:rem:curves-do-not-suffice}. It is also why the difficulty above is relocated rather than
dissolved --- it becomes the question of whether the plots of
$\PrPf(\Xcal)$ exhaust the ones the two families would force, and nothing in the paper depends on
the answer.

\subsection*{Related work}

The measurable half is \cite{HKSY17,For21}, the smooth half \cite{Sou80,IZ13,BKW25}, and the
machinery of concrete sheaves on a concrete site \cite{BH11,Dub79,MMS22}. The closest point of
contact is \cite{MMS22}: when it combines two concrete sites it does so by a \emph{sum}, whose
objects carry two independent plot families related only by constants --- the shape diagnosed above
--- whereas the site here is a \emph{product}. \Cref{disc:sec:related} is the full discussion,
including the semantics literature on differentiable and probabilistic programming.

\subsection*{Reader's guide}

\Cref{sec:qus} fixes the sample space and recalls what is used from \cite{For21}; a reader of that
paper can start at \Cref{sec:site}. \Cref{sec:site} is the site and is short. \Cref{sec:sss} is the
category and its constructions; a reader who wants only the monad needs
\Cref{sss:sec:def}, \Cref{sss:thm:cartesian-closed,sss:prp:prod-coprod} from it.
Two subsections of \Cref{sec:sss} stand apart: \Cref{sss:prp:reflection} names the comparison with
a presentation by $n$-parameter families and random variables, and is used in
\Cref{sec:discussion}; \Cref{sss:sec:tangent} is the
differential calculus and is used nowhere else, \Cref{ml:rem:not-a-morphism} explaining why.
\Cref{sec:modal} is the adjoint string; nothing in \Cref{sec:prob} depends on it, and it is used
only in \Cref{ml:sec:nonsmooth}. \Cref{sec:prob} is the heart of the paper, and from
\Cref{sec:sss} it uses only products and coproducts, cartesian closedness and the function-space
formula \Cref{eq:function-space}. \Cref{sec:ml} is an application and \Cref{sec:discussion} takes
stock.

\subsection*{Notation}

Beyond the standard, the symbols of \Cref{tab:notation} recur.

\begin{table}[htbp]
\centering
\renewcommand{\arraystretch}{1.25}
\begin{tabular}{@{}l p{0.60\textwidth}@{}}
    \toprule
    $\Omega,\ \Bcal_\Omega,\ \Bcal_\otimes,\ U,\ \PrPfO$ & the sample space, its universally
        measurable sets, its Borel sets, the uniform measure, all probability measures on it,
        \Cref{qus:def:sample-space} \\
    $\Theta = (\Theta_1,\Theta_2)$ & the seed-splitting isomorphism
        $\Omega\bij\Omega\times\Omega$, ($\Omega$2) of
        \Cref{qus:thm:omega-properties} \\
    $\Omega^\Omega$ & the monoid of universally measurable self-maps of $\Omega$,
        \Cref{qus:def:sample-space} \\
    $\Omega_n,\ \Site,\ \Site_\pt$ & the test objects, the site and its enlargement,
        \Cref{site:def:site,site:def:site-plus} \\
    $\Cmix(\Omega_m,\R^n)$ & the mixed-smooth maps, (M1) and (M2), \Cref{site:def:site} \\
    $\pi_n,\ \iota^n_{t_0},\ \lb W,\Phi\rb,\ \Xi^{(n)},\ \varphi^{(n)}$ & the distinguished
        substitutions, \Cref{site:not:morphisms} \\
    $\Xcal^{\Omega_n},\ \Xcal^\Omega,\ \Xcal^{\R^n},\ \Bcal_\Xcal$ & plots, random variables,
        $n$-parameter families, induced $\sigma$-algebra, \Cref{sss:not:derived} \\
    $\Bcal^\circ_{\R^n}$ & the \emph{Borel} $\sigma$-algebra of $\R^n$, as against the induced
        $\Bcal_{\R^n}$, \Cref{sss:eg:euclidean} \\
    $\one,\ \zero,\ \twoc,\ \twob,\ \twop$ & the point, the empty space and the three two-point
        spaces, \Cref{sss:eg:euclidean,sss:eg:two} \\
    $\ihom_\Zcal(f,g)$ & the internal hom over a base, \Cref{sss:prp:internal-hom} \\
    $\widehat\Xcal,\ \SSS_{\mathrm{pr}}$ & the maximal extension and the presented objects,
        \Cref{sss:rem:curves-do-not-suffice,sss:prp:reflection} \\
    $\Pi,\ \flat,\ \Und,\ \sharp,\ \Lambda$ & the adjoint string over $\QUS$,
        \Cref{modal:thm:string} \\
    $\PrPf,\ \Law,\ \delta,\ \Mbb,\ \otimes,\ \rho,\ \tau$ & the probability monad, the law of
        a plot, unit, multiplication, product of kernels, costrength, strength, \Cref{sec:prob} \\
    $\Kleisli(\PrPf)$ & the Kleisli category, i.e.\ the category of differentiable simulators,
        \Cref{prob:def:kleisli} \\
    $T\Xcal,\ T^*_x\Xcal,\ df_x$ & tangent bundle, cotangent space, differential,
        \Cref{sss:sec:tangent} \\
    \bottomrule
\end{tabular}
\caption{Recurring notation.}
\label{tab:notation}
\end{table}

\section{Quasi-Universal Spaces: the Measurable Half}
\label{sec:qus}

This section fixes the sample space and recalls, without proof, what we use from \cite{For21}. A
reader of that paper can skip to \Cref{sec:site}; the only item that is not verbatim is
\Cref{qus:lem:caratheodory}, which is classical but is used constantly below. Each quoted statement
carries a pointer to the section of \cite{For21} it comes from; those pointers refer to the revised
version \texttt{arXiv:2109.11631v2}, whose numbering differs from that of the first version.

\subsection{The Sample Space}
\label{qus:sec:sample}

\begin{Def}[The sample space: the universal Hilbert cube]
    \label{qus:def:sample-space}
    Throughout, the sample space is the \emph{universal Hilbert cube} of \cite[\S7]{For21}:
    \begin{align}
        \Omega &:= [0,1]^\N, &
        \Bcal_\Omega &:= \lp \textstyle\bigotimes_{n \in \N} \Bcal_{[0,1]}\rp_\uc, \\
        \Omega^\Omega &:= \Meas\lp(\Omega,\Bcal_\Omega),(\Omega,\Bcal_\Omega)\rp, &
        \PrPfO &:= \Pcal(\Omega,\Bcal_\Omega),
    \end{align}
    where $\Bcal_\otimes := \bigotimes_{n\in\N}\Bcal_{[0,1]}$ is the Borel $\sigma$-algebra of the
    cube and $(\cdot)_\uc$ is universal completion, so that $\Bcal_\Omega$ consists of \emph{all}
    universally measurable subsets, $\Omega^\Omega$ is the monoid of all universally measurable
    self-maps, and $\PrPfO$ is the set of all probability measures on $\Bcal_\Omega$. We write
    $U := \bigotimes_{n\in\N}U[0,1]$ for the uniform measure, and $\Pcal(E,\Bcal_E)$ for the set
    of all probability measures on an ordinary measurable space $(E,\Bcal_E)$.
\end{Def}

\begin{Thm}[The properties of $\Omega$ that get used, {\cite[\S\S5--7]{For21}}]
    \label{qus:thm:omega-properties}
    The sample space of \Cref{qus:def:sample-space} has the following properties.
    \begin{enumerate}
        \item[($\Omega$1)] $\Omega^\Omega$ is a monoid under composition containing $\id_\Omega$ and
            all constant maps.
        \item[($\Omega$2)] \emph{Product isomorphism.} The zip-locker map
            $\Theta = (\Theta_1,\Theta_2):\,\Omega \bij \Omega\times\Omega$, which de-interleaves
            the coordinates of a sequence, is a Borel isomorphism; so
            $\Omega\times\Omega \cong \Omega$ and, iterating, $\Omega^\N \cong \Omega$. We write
            $\Theta^{-1}(\omega,\omega')$ for the inverse, so that
            $\Theta_1\Theta^{-1}(\omega,\omega') = \omega$ and
            $\Theta_2\Theta^{-1}(\omega,\omega') = \omega'$.
        \item[($\Omega$3)] \emph{Fubini property.} For $P_1,P_2 \in \PrPfO$ the two iterated
            integrals of a bounded $\Bcal_{\Omega\times\Omega}$-measurable function agree, so that
            $P_1\otimes P_2$ is well defined on $\Bcal_{\Omega\times\Omega}$ and symmetric. Here
            $\Bcal_{\Omega\times\Omega}$ is the \emph{induced} $\sigma$-algebra of the
            quasi-universal product, \Cref{qus:def:induced-sigma} below; it is in general strictly
            larger than $\Bcal_\Omega\otimes\Bcal_\Omega$, and the strength of ($\Omega$3) lies
            in that.
        \item[($\Omega$4)] \emph{Universal completeness.} $\Bcal_\Omega$ consists of all universally
            measurable sets; in particular it contains every analytic subset of $\Omega$.
        \item[($\Omega$5)] \emph{Universality.} $(\Omega,\Bcal_\otimes)$ is standard Borel, every
            standard Borel space embeds into it with Borel image, $U$ is atomless, and for every
            $P \in \PrPfO$ there is $\psi \in \Omega^\Omega$ with $\psi_*U = P$.
        \item[($\Omega$6)] $\PrPfO$ consists of \emph{all} probability measures on $\Bcal_\Omega$,
            so push-forwards may be taken freely.
    \end{enumerate}
\end{Thm}

\begin{Lem}[Borel suffices, and $\Theta$ is an isomorphism, {\cite[\S7]{For21}}]
    \label{qus:lem:borel-suffices}
    \begin{enumerate}
        \item A self-map of the cube is admissible as soon as it is Borel-measurable:
            \begin{align}
                \Omega^\Omega &= \Meas\lp(\Omega,\Bcal_\Omega),(\Omega,\Bcal_\otimes)\rp .
                \nonumber
            \end{align}
        \item If $\Phi_1,\Phi_2 \in \Omega^\Omega$ then
            $\Theta^{-1}(\Phi_1,\Phi_2) \in \Omega^\Omega$. Consequently
            $\Theta:\,\Omega\to\Omega\times\Omega$ is an isomorphism of quasi-universal spaces, and
            $A\mapsto\Theta(A)$ is a bijection $\Bcal_\Omega \bij \Bcal_{\Omega\times\Omega}$.
    \end{enumerate}
    \begin{proof}
        (1) ``$\ins$'' holds because $\Bcal_\otimes \ins \Bcal_\Omega$. For ``$\sni$'' let $f$ be
        Borel-measurable and let $S \in \Bcal_\Omega$. For $P \in \PrPfO$ the push-forward $f_*P$ is
        a Borel probability measure on the standard Borel space $(\Omega,\Bcal_\otimes)$, and $S$ is
        universally measurable by ($\Omega$4), so there are $B_1 \ins S \ins B_2$ in $\Bcal_\otimes$
        with $f_*P(B_2\sm B_1) = 0$. Then $f^{-1}(B_1) \ins f^{-1}(S) \ins f^{-1}(B_2)$ with
        $P\lp f^{-1}(B_2)\sm f^{-1}(B_1)\rp = 0$, so $f^{-1}(S)$ lies in the $P$-completion of
        $\Bcal_\Omega$. As $P$ was arbitrary and $\Bcal_\Omega$ is its own universal completion by
        ($\Omega$4), $f^{-1}(S) \in \Bcal_\Omega$.

        (2) The rectangles $A_1\times A_2$ with $A_i \in \Bcal_\otimes$ generate
        $\Bcal_\otimes\otimes\Bcal_\otimes$, and
        $(\Phi_1,\Phi_2)^{-1}(A_1\times A_2) = \Phi_1^{-1}(A_1)\cap\Phi_2^{-1}(A_2) \in
        \Bcal_\Omega$; so $(\Phi_1,\Phi_2)$ is
        $\Bcal_\Omega$-$(\Bcal_\otimes\otimes\Bcal_\otimes)$-measurable, and composing with the
        Borel map $\Theta^{-1}$ of ($\Omega$2) and applying (1) gives
        $\Theta^{-1}(\Phi_1,\Phi_2) \in \Omega^\Omega$. Since $\Theta_1,\Theta_2$ are Borel,
        $\Theta$ is quasi-measurable, and what has just been shown says that its inverse is; an
        isomorphism of quasi-universal spaces induces a bijection of the induced $\sigma$-algebras
        of \Cref{qus:def:induced-sigma}.
    \end{proof}
\end{Lem}

\begin{Lem}[Splitting the uniform seed, {\cite[\S\S5, 7]{For21}}]
    \label{qus:lem:split-U}
    $\Theta_*U = U\otimes U$. Under $U$ the two halves $\Theta_1,\Theta_2$ of a seed are
    independent and each is again uniform, so splitting a seed needs no auxiliary measure.
\end{Lem}

This single identity carries every computation in \Cref{sec:prob}. It is worth saying once, in
advance, what the pattern is: whenever a construction needs two independent sources of randomness
--- a parameter and a sample, two factors of a product, a mixture and the measure it mixes --- we
split one seed into two with $\Theta$ and use \Cref{qus:lem:split-U} to see that the two halves are
independent uniforms. No step below does anything else.

\subsection{Quasi-Universal Spaces}
\label{qus:sec:qus}

\begin{Def}[Quasi-universal spaces, {\cite[\S\S2, 7]{For21}}]
    \label{qus:def:qus}
    A \emph{quasi-universal space} $(\Scal,\Scal^\Omega)$ is a set $\Scal$ together with a set
    $\Scal^\Omega$ of maps $X:\,\Omega\to\Scal$, the \emph{admissible random variables}, containing
    all constant maps and satisfying $X\circ\Phi \in \Scal^\Omega$ for all $X \in \Scal^\Omega$ and
    $\Phi \in \Omega^\Omega$. A map $g:\,\Scal\to\Tcal$ is \emph{quasi-measurable} if
    $g\circ X \in \Tcal^\Omega$ for all $X \in \Scal^\Omega$. The resulting category is $\QUS$, and
    $\Scal^\Omega = \QUS(\Omega,\Scal)$.
\end{Def}

\begin{Def}[The induced $\sigma$-algebra, {\cite[\S3]{For21}}]
    \label{qus:def:induced-sigma}
    Put
    \begin{align}
        \Bcal_\Scal &:= \lC A \ins \Scal \st \forall X \in \Scal^\Omega.\;
        X^{-1}(A) \in \Bcal_\Omega \rC ,
    \end{align}
    the final $\sigma$-algebra of the family $\Scal^\Omega$. Every admissible random variable and
    every quasi-measurable map is measurable for these $\sigma$-algebras, and
    \begin{align}
        \Bcal_\Scal &= \bigcap_{X \in \Scal^\Omega,\, P \in \PrPfO} X_*\lp\Bcal_\Omega\rp_P ,
        \label{eq:induced-sigma-explicit}
    \end{align}
    so $\Bcal_\Scal$ is universally complete. The inclusion
    $\Scal^\Omega \ins \Meas\lp(\Omega,\Bcal_\Omega),(\Scal,\Bcal_\Scal)\rp$ is in general strict.
\end{Def}

\begin{Thm}[$\QUS$ is a quasitopos, {\cite[\S2]{For21}}]
    \label{qus:thm:quasitopos}
    $\QUS$ is complete, cocomplete, cartesian closed and locally cartesian closed, and it is a
    quasitopos. Products and function spaces are
    \begin{align}
        \lp\textstyle\prod_i\Scal_i\rp^\Omega &= \textstyle\prod_i \Scal_i^\Omega, &
        \lp\Tcal^\Scal\rp^\Omega &= \lC Y \st \forall \Phi \in \Omega^\Omega\ \forall X \in
        \Scal^\Omega.\ \lp\omega\mapsto Y(\Phi(\omega))(X(\omega))\rp \in \Tcal^\Omega\rC ,
    \end{align}
    coproducts are \emph{rigid}, $\lp\coprod_i\Scal_i\rp^\Omega = \bigcup_i\iota_i\circ
    \Scal_i^\Omega$, and the indiscrete two-point space $\twop$ of \Cref{sss:eg:two} classifies
    embeddings.
\end{Thm}

\subsection{The Push-forward Probability Monad on \texorpdfstring{$\QUS$}{QUS}}
\label{qus:sec:monad}

\begin{Def}[The object of measures, {\cite[\S5]{For21}}]
    \label{qus:def:PrPf}
    For $\Scal \in \QUS$,
    \begin{align}
        \PrPf(\Scal) &:= \lC X_*P \st X \in \Scal^\Omega,\, P \in \PrPfO\rC \ins
        \Pcal(\Scal,\Bcal_\Scal), \\
        \PrPf(\Scal)^\Omega &:= \lC \lp\omega\mapsto X(\omega)_*P\rp \st
        X \in \lp\Scal^\Omega\rp^\Omega,\, P \in \PrPfO\rC .
    \end{align}
\end{Def}

\begin{Thm}[The push-forward monad on $\QUS$, {\cite[\S\S5, 7]{For21}}]
    \label{qus:thm:PrPf}
    $(\PrPf,\delta,\Mbb)$, with $\delta_\Scal(x) := \delta_x$ and
    $\Mbb_\Scal(\Gamma) := \int \nu\,\Gamma(d\nu)$, is a strong commutative affine monad on $\QUS$. The
    product of measures $\mu\otimes\nu$ is defined on all of $\Bcal_{\Scal\times\Tcal}$, not merely
    on $\Bcal_\Scal\otimes\Bcal_\Tcal$, by
    \begin{align}
        (\mu\otimes\nu)(C) &:= \int\!\!\int \I\lB \lp X(\omega_1),Y(\omega_2)\rp \in C \rB
        \, U(d\omega_2)\, U(d\omega_1), \qquad C \in \Bcal_{\Scal\times\Tcal},
        \label{eq:otimes-section}
    \end{align}
    for any $X,Y$ with $X_*U = \mu$, $Y_*U = \nu$; the right-hand side is well defined by
    ($\Omega$3), independent of the choice of $X,Y$, and restricts to the usual product measure on
    $\Bcal_\Scal\otimes\Bcal_\Tcal$.
\end{Thm}

Everything in \Cref{sec:prob} is a lift of \Cref{qus:thm:PrPf} along the site of \Cref{sec:site},
and the reader will find that each proof there is the corresponding proof of \cite{For21} with one
extra argument carried along.

\subsection{One Classical Lemma}
\label{qus:sec:caratheodory}

\begin{Lem}[Carath\'eodory]
    \label{qus:lem:caratheodory}
    Let $(A,\Acal)$ be a measurable space, $T$ a separable metrizable space, $E$ a metrizable
    space, and $g:\,T\times A \to E$ a map which is continuous in the first argument for each fixed
    second argument and $\Acal$-measurable in the second for each fixed first. Then $g$ is
    $\Bcal_T\otimes\Acal$-measurable, where $\Bcal_T$ is the Borel $\sigma$-algebra of $T$.
    \begin{proof}
        Let $D \ins T$ be countable dense and, for $k \ge 1$, let $\lC T_{k,d}\rC_{d \in D}$ be the
        Borel partition of $T$ obtained from the cover by the balls $B(d,1/k)$, $d\in D$, by
        disjointification along a fixed enumeration of $D$. Put
        $g_k(t,a) := g(d_k(t),a)$, where $d_k(t)$ is the unique $d$ with $t \in T_{k,d}$. Each
        $g_k$ is $\Bcal_T\otimes\Acal$-measurable, being a countable ``staircase'': on
        $T_{k,d}\times A$ it equals the measurable map $g(d,-)$ composed with the projection.
        Since $d_k(t) \to t$ and $g$ is continuous in the first argument, $g_k \to g$ pointwise,
        and a pointwise limit of measurable maps into a metrizable space is measurable.
    \end{proof}
\end{Lem}

We use \Cref{qus:lem:caratheodory} chiefly through the following consequence, which is what makes
the composition law of the site of \Cref{sec:site} work: a map $\R^m\times\Omega \to \R^n$ which
is smooth in the first argument and $\Bcal_\Omega$-measurable in the second is automatically
$\Bcal^\circ_{\R^m}\otimes\Bcal_\Omega$-measurable --- here and throughout,
$\Bcal^\circ_{\R^k}$ denotes the \emph{Borel} $\sigma$-algebra of $\R^k$, the superscript
distinguishing it from the induced $\sigma$-algebra $\Bcal_{\R^k}$ of
\Cref{qus:def:induced-sigma} --- and hence may be precomposed with a measurable pair
$\omega\mapsto\lp W(\omega),\Phi(\omega)\rp$ without further hypotheses. It is used once more,
with a manifold target, in \Cref{sss:eg:manifolds}.

\section{The Site of Test Objects}
\label{sec:site}

The whole design of this paper is in the choice of test objects. We take them to be the products
\begin{align}
    \Omega_n &:= \R^n \times \Omega, \qquad n \ge 0,
\end{align}
of the Cartesian space $\R^n$ with the sample space of \Cref{qus:def:sample-space}, so that
$\Omega_0 = \Omega$. A plot of a space will be a map out of some $\Omega_n$ which is smooth in the
$\R^n$-direction and admissible in the $\Omega$-direction \emph{at once}. The point of insisting on
this, rather than on two separate families of maps out of $\R^n$ and out of $\Omega$, is not that
one obtains more plots --- \Cref{site:prp:karoubi} shows that the purely smooth test objects are
redundant --- but that one obtains more \emph{morphisms between test objects}. That is
\Cref{site:rem:mixed}, and it is the reason the probability monad of \Cref{sec:prob} exists.

\subsection{The Site}
\label{site:sec:def}

\begin{Def}[The site $\Site$]
    \label{site:def:site}
    Let $\Site$ be the category with objects $\Omega_n$, $n \ge 0$, and with
    \begin{align}
        \Site(\Omega_m,\Omega_n) &:= \lC (g,\Phi) \st g \in \Cmix(\Omega_m,\R^n),\
        \Phi \in \Omega^\Omega \rC, &
        (g,\Phi)(t,\omega) &:= \lp g(t,\omega),\, \Phi(\omega)\rp,
    \end{align}
    where $\Cmix(\Omega_m,\R^n)$, the set of \emph{mixed-smooth} maps, is the set of maps
    $g:\,\Omega_m = \R^m\times\Omega\to\R^n$ such that
    \begin{enumerate}
        \item[(M1)] $g(-,\omega) \in C^\infty(\R^m,\R^n)$ for every $\omega \in \Omega$, and
        \item[(M2)] $g(t,-)$ is $\Bcal_\Omega$-$\Bcal^\circ_{\R^n}$-measurable for every
            $t \in \R^m$, where $\Bcal^\circ_{\R^n}$ is the \emph{Borel} $\sigma$-algebra.
    \end{enumerate}
    Note the asymmetry: the smooth component may depend on the seed, whereas the seed component
    $\Phi$ may \emph{not} depend on $t$. This is not a stipulation of convenience --- it is forced
    if $\Site$ is to be a category at all, \Cref{site:rem:flat}.
\end{Def}

\begin{Lem}[$\Site$ is a category]
    \label{site:lem:category}
    Identities lie in $\Site$ and composition of the maps above is again of the same form, so
    \Cref{site:def:site} defines a category, and every $\Site$-morphism is in particular
    measurable from $\Bcal^\circ_{\R^m}\otimes\Bcal_\Omega$ to
    $\Bcal^\circ_{\R^n}\otimes\Bcal_\Omega$.
    \begin{proof}
        For the identity take $g = \pr_{\R^n}$ and $\Phi = \id_\Omega$. Let
        $(g,\Phi) \in \Site(\Omega_m,\Omega_n)$ and $(g',\Phi') \in \Site(\Omega_n,\Omega_k)$. The
        composite is $(g'',\Phi'\circ\Phi)$ with
        $g''(t,\omega) = g'\lp g(t,\omega),\Phi(\omega)\rp$, and $\Phi'\circ\Phi \in \Omega^\Omega$
        by ($\Omega$1). For (M1), fix $\omega$: $g(-,\omega)$ is $C^\infty$ and
        $g'(-,\Phi(\omega))$ is $C^\infty$, so the composite is $C^\infty$. For (M2), fix $t$: by
        \Cref{qus:lem:caratheodory} applied to $g'$ --- which is continuous in the first argument by
        (M1) and measurable in the second by (M2) --- the map $g'$ is
        $\Bcal^\circ_{\R^n}\otimes\Bcal_\Omega$-measurable; and
        $\omega \mapsto \lp g(t,\omega),\Phi(\omega)\rp$ is measurable into
        $\lp\R^n\times\Omega,\Bcal^\circ_{\R^n}\otimes\Bcal_\Omega\rp$ by (M2) for $g$ and by
        $\Phi \in \Omega^\Omega$. The composite of the two is measurable. Associativity and the unit
        laws are inherited from composition of maps. The last claim is
        \Cref{qus:lem:caratheodory} again.
    \end{proof}
\end{Lem}

\begin{Not}[Distinguished morphisms]
    \label{site:not:morphisms}
    We name the substitutions used throughout. Let $n,m \ge 0$.
    \begin{align}
        \pi_n &:= (\,!\,,\id_\Omega):\; \Omega_n \to \Omega_0, &
        \pi_n(t,\omega) &= \omega, \\
        \iota^n_{t_0} &:= (t_0,\id_\Omega):\; \Omega_0 \to \Omega_n, &
        \iota^n_{t_0}(\omega) &= (t_0,\omega), \\
        \lb W,\Phi\rb &:= (W,\Phi):\; \Omega_0 \to \Omega_n, &
        \lb W,\Phi\rb(\omega) &= \lp W(\omega),\Phi(\omega)\rp, \\
        \Xi^{(n)} &:= (\pr_{\R^n},\Xi):\; \Omega_n \to \Omega_n, &
        \Xi^{(n)}(t,\omega) &= \lp t,\Xi(\omega)\rp, \\
        \varphi^{(n)} &:= (\varphi\circ\pr_{\R^m},\id_\Omega):\; \Omega_m \to \Omega_n, &
        \varphi^{(n)}(t,\omega) &= \lp\varphi(t),\omega\rp,
    \end{align}
    for $W \in \Meas\lp(\Omega,\Bcal_\Omega),(\R^n,\Bcal^\circ_{\R^n})\rp$,
    $\Phi,\Xi \in \Omega^\Omega$, $t_0 \in \R^n$
    and $\varphi \in C^\infty(\R^m,\R^n)$. All five are $\Site$-morphisms. The third, in which the
    smooth coordinate is read off the seed, is the \emph{mixed} morphism of
    \Cref{site:rem:mixed}; it exists only because the target carries both directions, and it has no
    counterpart if the two directions are kept apart.
\end{Not}

\subsection{Finite Products}
\label{site:sec:products}

\begin{Prp}[$\Site$ has binary products]
    \label{site:prp:products}
    For $m,n \ge 0$ the object $\Omega_{m+n}$, together with
    \begin{align}
        q_1 &:= \lp \pr_{\R^m}\circ\pr_{\R^{m+n}},\ \Theta_1 \rp:\; \Omega_{m+n} \to \Omega_m, &
        q_2 &:= \lp \pr_{\R^n}\circ\pr_{\R^{m+n}},\ \Theta_2 \rp:\; \Omega_{m+n} \to \Omega_n,
    \end{align}
    where $\R^{m+n} = \R^m\times\R^n$ and $\Theta = (\Theta_1,\Theta_2)$ is the isomorphism of
    ($\Omega$2), is a product of $\Omega_m$ and $\Omega_n$ in $\Site$. Consequently $\Site$ has all
    finite non-empty products. It has \emph{no} terminal object: $\Site(\Omega_0,\Omega_0) = \Omega^\Omega$ has more than
    one element, and for $n \ge 1$ so does $\Site(\Omega_0,\Omega_n)$, which contains the distinct
    morphisms $\iota^n_{t_0}$, $t_0 \in \R^n$.
    \begin{proof}
        $q_1,q_2$ are $\Site$-morphisms, their smooth parts being linear projections and their seed
        parts lying in $\Omega^\Omega$ by ($\Omega$2) together with
        \Cref{qus:lem:borel-suffices}(1). Let $f_i = (g_i,\Phi_i) \in
        \Site(\Omega_k,\Omega_{n_i})$ for $(n_1,n_2) = (m,n)$. A morphism
        $h = (g,\Phi) \in \Site(\Omega_k,\Omega_{m+n})$ satisfies $q_i \circ h = f_i$ if and only if
        $g = (g_1,g_2)$ and $\Theta_i\circ\Phi = \Phi_i$, i.e.\ $\Phi = \Theta^{-1}(\Phi_1,\Phi_2)$.
        This determines $h$ uniquely; and $h$ is a $\Site$-morphism, since $(g_1,g_2)$ satisfies
        (M1) and (M2) componentwise and $\Theta^{-1}(\Phi_1,\Phi_2) \in \Omega^\Omega$ by
        \Cref{qus:lem:borel-suffices}(2).
    \end{proof}
\end{Prp}

The absence of a terminal object is harmless and is repaired in \Cref{site:lem:products-plus}; what
matters is that a site of \emph{products} of the two kinds of test object closes up under products
at all, and it does so for exactly one reason: $\Omega \times \Omega \cong \Omega$. A site of
disjoint smooth and measurable test objects does not.

\subsection{What the Site Already Contains}
\label{site:sec:karoubi}

One might expect to need the purely smooth test objects $\R^n$ as well, and the one-point space
$\pt$. Neither is necessary: both are retracts of objects of $\Site$, hence lie in its Karoubi
envelope, and a plot family on a retract is determined.

\begin{Def}[The enlarged site]
    \label{site:def:site-plus}
    Let $\Site_\pt$ have objects $\Omega_n$ and $\R^n$, $n \ge 0$, with $\R^0 = \pt$ the one-point
    space, and morphisms
    \begin{align}
        \Site_\pt(\Omega_m,\Omega_n) &:= \Site(\Omega_m,\Omega_n), &
        \Site_\pt(\Omega_m,\R^n) &:= \Cmix(\Omega_m,\R^n), \\
        \Site_\pt(\R^m,\R^n) &:= C^\infty(\R^m,\R^n), &
        \Site_\pt(\R^m,\Omega_n) &:= C^\infty(\R^m,\R^n) \times \Omega ,
    \end{align}
    the last being the pairs $(\varphi,\omega_0)$ giving $t \mapsto (\varphi(t),\omega_0)$, so that
    a map into the seed direction out of a purely smooth object is constant. Composition is
    composition of maps, and that it is well defined is checked case by case: the two purely mixed
    cases are \Cref{site:lem:category}; the cases $\Omega_m\to\R^n\to\R^k$ and
    $\R^m\to\Omega_n\to\R^k$ use (M1), (M2) and the chain rule; $\R^m\to\R^n\to\R^k$ is
    composition of $C^\infty$ maps; and in the three remaining cases, all with target $\Omega_k$,
    the seed component of the composite is a constant, $\omega\mapsto\Phi'(\omega_0)$ resp.\
    $\omega\mapsto\omega_0$, which lies in $\Omega^\Omega$ by ($\Omega$1).
\end{Def}

\begin{Lem}[$\Site_\pt$ has all finite products]
    \label{site:lem:products-plus}
    In $\Site_\pt$ the object $\pt = \R^0$ is terminal, and binary products exist:
    \begin{align}
        \R^m\times\R^n &= \R^{m+n}, &
        \R^m\times\Omega_n &= \Omega_{m+n}, &
        \Omega_m\times\Omega_n &\;\cong\; \Omega_{m+n},
        \label{eq:site-plus-products}
    \end{align}
    the first two with the projections of the underlying sets, under the usual identifications
    $\R^m\times\R^n = \R^{m+n}$ and $\R^m\times\lp\R^n\times\Omega\rp =
    \R^{m+n}\times\Omega$. Only the third requires the zip-locker isomorphism $\Theta$, and it is
    an isomorphism rather than an equality for that reason.
    \begin{proof}
        $\Site_\pt(\Omega_m,\pt) = \Cmix(\Omega_m,\R^0)$ and $\Site_\pt(\R^m,\pt) =
        C^\infty(\R^m,\R^0)$ are singletons, so $\pt$ is terminal.

        For $\R^m\times\R^n = \R^{m+n}$: the two projections are smooth. A cone from $\R^k$ is a
        pair $\varphi_i \in C^\infty(\R^k,\R^{n_i})$ and factors uniquely through
        $(\varphi_1,\varphi_2) \in C^\infty(\R^k,\R^{m+n})$; a cone from $\Omega_k$ is a pair
        $f_i \in \Cmix(\Omega_k,\R^{n_i})$ and factors uniquely through $(f_1,f_2)$, which lies in
        $\Cmix(\Omega_k,\R^{m+n})$ because (M1) and (M2) are conditions on the components.

        For $\R^m\times\Omega_n = \Omega_{m+n}$: the projections are $\pr_{\R^m} \in
        \Cmix(\Omega_{m+n},\R^m)$ and $\lp\pr_{\R^n},\id_\Omega\rp \in
        \Site(\Omega_{m+n},\Omega_n)$, where $\pr_{\R^m},\pr_{\R^n}$ are taken on the
        $\R^{m+n}$ factor. A cone from $\Omega_k$ is a pair $f_1 \in \Cmix(\Omega_k,\R^m)$ and
        $f_2 = (g,\Phi) \in \Site(\Omega_k,\Omega_n)$, and $h = (\tilde g,\tilde\Phi) \in
        \Site(\Omega_k,\Omega_{m+n})$ makes both triangles commute if and only if
        $\pr_{\R^m}\circ\tilde g = f_1$, $\pr_{\R^n}\circ\tilde g = g$ and
        $\tilde\Phi = \Phi$, that is $h = \lp (f_1,g),\Phi\rp$, which is a $\Site$-morphism
        because (M1) and (M2) hold componentwise. A cone from $\R^k$ is a pair
        $\varphi \in C^\infty(\R^k,\R^m)$ and $(\psi,\omega_0) \in \Site_\pt(\R^k,\Omega_n)$,
        and factors uniquely through $\lp(\varphi,\psi),\omega_0\rp$. Note that the seed is
        \emph{not} duplicated here, because $\R^m$ carries none; this is why no $\Theta$ appears
        and the product is the object $\Omega_{m+n}$ on the nose.

        For $\Omega_m\times\Omega_n \cong \Omega_{m+n}$: the cone of \Cref{site:prp:products}
        already has the universal property against the test objects $\Omega_k$. Against $\R^k$, a
        cone is a pair $(\varphi_i,\omega_i) \in \Site_\pt(\R^k,\Omega_{n_i})$, and
        $q_i\circ(\varphi,\omega_0) = \lp\pr_{\R^{n_i}}\circ\varphi,\ \Theta_i\omega_0\rp$,
        so the unique factorisation is
        $\lp(\varphi_1,\varphi_2),\ \Theta^{-1}(\omega_1,\omega_2)\rp$.

        Finite products now follow, the empty one being $\pt$.
    \end{proof}
\end{Lem}

\begin{Prp}[$\R^n$ and $\pt$ are retracts]
    \label{site:prp:retracts}
    Fix $\omega_0 \in \Omega$ and let
    \begin{align}
        \iota &:= (\id_{\R^n},\omega_0):\; \R^n \to \Omega_n, &
        \pi &:= \pr_{\R^n}:\; \Omega_n \to \R^n .
    \end{align}
    Then $\pi\circ\iota = \id_{\R^n}$, and $e := \iota\circ\pi \in \Site(\Omega_n,\Omega_n)$,
    $e(t,\omega) = (t,\omega_0)$, is an idempotent of $\Site$ whose splitting is $\R^n$. In
    particular $\pt = \R^0$ splits the idempotent $\omega\mapsto\omega_0$ of $\Omega_0$.
\end{Prp}

The three directions in which a test object can be moved are worth seeing together,
\Cref{fig:directions}.

\begin{figure}[htbp]
\centering
\begin{tikzcd}[column sep=5.5em, row sep=1em]
    \R^n \arrow[r, shift left=.7ex, "\iota"] &
    \Omega_n \arrow[l, shift left=.7ex, "\pi"] \arrow[r, shift left=.7ex, "\pi_n"] &
    \Omega_0 = \Omega \arrow[l, shift left=.7ex, "{\lb W,\Phi\rb}"]
\end{tikzcd}
\caption{The four substitutions between test objects. Only $\lb W,\Phi\rb$ mixes the two
directions.}
\label{fig:directions}
\end{figure}

On the left, $\iota$ and $\pi$ are those of \Cref{site:prp:retracts} and satisfy
$\pi\circ\iota = \id$, so that $\R^n$ is a retract and, by \Cref{site:prp:karoubi} below, carries no
information of its own; these two live in $\Site_\pt$, since $\R^n$ is not an object of $\Site$. On
the right, $\pi_n$ and $\lb W,\Phi\rb$ are those of \Cref{site:not:morphisms} and are genuine
$\Site$-morphisms: $\pi_n$ forgets the parameter, and $\lb W,\Phi\rb$ reads it off the seed. Only
the last of the four mixes the two directions, and it is the one that a disjoint pair of test
objects cannot provide, \Cref{site:rem:mixed}.

\begin{Prp}[Adjoining them changes nothing]
    \label{site:prp:karoubi}
    Let $\Xcal$ be a set and let $\lp\Xcal^{\Omega_n}\rp_{n\ge0}$ be families of maps
    $\Omega_n \to \Xcal$ containing the constants and closed under precomposition with
    $\Site$-morphisms. Then there is exactly one way to extend them to families
    $\Xcal^{\R^n}$ closed under precomposition with $\Site_\pt$-morphisms, namely
    \begin{align}
        \Xcal^{\R^n} &= \lC \gamma:\,\R^n \to \Xcal \st \gamma\circ\pi \in \Xcal^{\Omega_n} \rC ,
        \label{eq:karoubi}
    \end{align}
    and it satisfies $\Xcal^\pt = \Xcal$. Consequently the categories of such structures over
    $\Site$ and over $\Site_\pt$ are isomorphic.
    \begin{proof}
        Suppose the families are given and closed. If $\gamma \in \Xcal^{\R^n}$ then
        $\gamma\circ\pi \in \Xcal^{\Omega_n}$, since $\pi$ is a $\Site_\pt$-morphism; this is
        ``$\ins$'' in \Cref{eq:karoubi}. Conversely if $\gamma\circ\pi \in \Xcal^{\Omega_n}$ then,
        $\iota$ being a $\Site_\pt$-morphism,
        $\gamma = (\gamma\circ\pi)\circ\iota \in \Xcal^{\R^n}$, which is ``$\sni$''. So the
        extension is unique if it exists.

        For existence, define $\Xcal^{\R^n}$ by \Cref{eq:karoubi} and check closure. Let
        $\gamma \in \Xcal^{\R^n}$. For $\varphi \in C^\infty(\R^m,\R^n)$,
        $(\gamma\circ\varphi)\circ\pi = (\gamma\circ\pi)\circ\varphi^{(n)} \in \Xcal^{\Omega_m}$ by
        \Cref{site:not:morphisms}, so $\gamma\circ\varphi \in \Xcal^{\R^m}$. For
        $g \in \Site_\pt(\Omega_m,\R^n)$ we must show $\gamma\circ g \in \Xcal^{\Omega_m}$: indeed
        $(g,\id_\Omega) \in \Site(\Omega_m,\Omega_n)$ and
        $\gamma\circ g = (\gamma\circ\pi)\circ(g,\id_\Omega) \in \Xcal^{\Omega_m}$. Closure of
        $\Xcal^{\Omega_n}$ under $\Site_\pt(\R^m,\Omega_n)$ and $\Site_\pt(\Omega_m,\Omega_n)$ is
        immediate from \Cref{eq:karoubi} and the hypothesis. Finally
        $\Xcal^\pt = \lC x \st \mathrm{const}_x \in \Xcal^\Omega\rC = \Xcal$ because constants
        are plots.
    \end{proof}
\end{Prp}

We therefore work with $\Site$, and use $\Site_\pt$ only when a terminal object is wanted, as in
\Cref{sss:sec:quasitopos}. Note the two readings of \Cref{eq:karoubi}: it says that a purely smooth
plot is the same thing as a seed-independent mixed plot, and hence that the smooth structure is
recoverable from the mixed one, while whether the mixed structure is in turn recoverable from the
smooth and the measurable one is open, \Cref{sss:rem:curves-do-not-suffice,prob:rem:comparison}.

\subsection{Why the Seed Belongs in Every Test Object}
\label{site:sec:why}

\begin{Rem}[The mixed morphisms are the point]
    \label{site:rem:mixed}
    Since \Cref{site:prp:karoubi} says the purely smooth test objects add no plots, one may ask what
    putting $\Omega$ into every test object is for. The answer is
    $\lb W,\Phi\rb \in \Site(\Omega_0,\Omega_n)$ of \Cref{site:not:morphisms}: the substitution in
    which the smooth coordinate is itself read off the seed. Two consequences, both structural.

    First, closure of plot families under $\lb W,\Phi\rb$ says that a smooth family is
    automatically measurable: if $p \in \Xcal^{\Omega_1}$ and $W$ is a measurable real random
    variable, then $\omega\mapsto p(W(\omega),\omega)$ is an admissible random variable. In a theory
    built on two disjoint families of maps out of $\R$ and out of $\Omega$ this is not available
    and has to be imposed as an axiom; here it is an instance of closure.

    Second, and this is what \Cref{sec:prob} lives on, the mixed plot family
    $\Xcal^{\Omega_n}$ is \emph{primitive data} rather than something computed from $\Xcal^\Omega$
    and the curves. This is the obstruction set out as (a), (b), (c) in \Cref{sec:intro}: with two
    separate families, conditions (a) and (b) --- smooth for each seed, admissible under each
    substitution --- are all one can state of a family of measures, whereas the multiplication of
    \Cref{prob:thm:mult} consumes (c), the existence of a single representative $Y$; and getting (c)
    from (a) and (b) is a measurable selection problem whose solvability we do not know,
    \Cref{sss:rem:curves-do-not-suffice}. Put differently, with two families the mixed data is
    \emph{forced} to be an internal hom, whereas the monad needs it to be a push-forward image. Over
    the present site (c) is simply a plot condition, and \Cref{prob:def:PrPf} takes it as the
    definition: the mixed family is ours to choose, and that is the whole of the difference.
\end{Rem}

\begin{Rem}[Flatness of the seed direction]
    \label{site:rem:flat}
    The seed component $\Phi$ of a $\Site$-morphism is required to be independent of $t$. This is
    the statement that $\Omega$ carries no smooth structure, and it is forced: were $\Phi$ allowed
    to depend on $t$, subject only to a measurability condition in $\omega$ --- there being no
    smooth structure on $\Omega$ to constrain it --- then $\Site$ would not be a category, because
    (M1) would not survive composition. For a witness take $m=n=k=1$, $g := \pr_\R$ and
    $\Phi(t,\omega) := \lp\tfrac{|t|}{1+|t|},0,0,\dots\rp \in \Omega$, each of whose sections at
    fixed $t$ is constant in $\omega$ and hence measurable, together with
    $g'(s,\omega') := \omega'_0$, the first coordinate, which is Borel; the composite has smooth
    part $g''(t,\omega) = |t|/(1+|t|)$, which is not differentiable at $t = 0$, violating (M1). It is also why
    $[0,1]^\N$ is an acceptable sample space although it is neither a manifold nor without
    boundary: nothing smooth is ever asked of it.
\end{Rem}

\begin{Rem}[What the site does \emph{not} have]
    \label{site:rem:no-locality}
    Three omissions are deliberate.
    \begin{enumerate}
        \item \emph{No open domains.} We do not include test objects $V\times\Omega$ for open
            $V \ins \R^n$. Unlike $\R^n$ these are not retracts of objects of $\Site$: a
            retraction $V\times\Omega\to\Omega_m\to V\times\Omega$, for any $m \ge 0$,
            restricts at each fixed seed to smooth maps exhibiting $V$ as a smooth retract of
            $\R^m$, and an annulus is a retract of no $\R^m$, since a retract of a contractible
            space has vanishing $H^1$. So
            this is a genuine restriction, not an application of \Cref{site:prp:karoubi}.
        \item \emph{No Grothendieck topology.} We use the trivial topology, for which every
            presheaf is a sheaf, so there is no sheaf condition and no gluing. Plot families are not
            required to be local in $t$.
        \item \emph{No patching in the seed direction.} As in \cite{For21} we do not require
            $\Xcal^{\Omega_n}$ to be closed under gluing along countable measurable partitions of
            $\Omega$; see \cite{HKSY17} for the variant that does.
    \end{enumerate}
    Items (1) and (2) go together and are discussed in \Cref{disc:rem:locality}: locality would buy
    gluing, which is what differential-geometric constructions on non-manifold objects need, and
    would cost the identification of $\PrPf(\Xcal)^{\Omega_n}$ with an image, which is what
    \Cref{sec:prob} needs.
\end{Rem}

\begin{Rem}[Relation to the literature]
    \label{site:rem:literature}
    Concrete sheaves on a concrete site are the standard machinery for convenient categories of
    smooth spaces, \cite{BH11,Dub79}, and \cite{MMS22} develops the $\omega$-concrete version with
    worked sites for standard Borel spaces, for Cartesian spaces and for piecewise-analytic
    domains. When \cite{MMS22} combines two such sites it does so by a \emph{sum}, identifying
    terminal and initial objects and adjoining constant maps between them; the resulting objects
    carry two independent plot families related only by constants, which is precisely the shape a
    presentation by curves and random variables would have. The present site is a \emph{product}:
    its
    objects mix the two directions, and by \Cref{site:rem:mixed} that is where all the difference
    lies. We are not aware of this instance in the literature, although the framework of
    \cite{BH11,Dub79,MMS22} accommodates it without modification.
\end{Rem}

\section{Sample-Smooth Spaces}
\label{sec:sss}

\subsection{Definition}
\label{sss:sec:def}

\begin{Def}[Sample-smooth space]
    \label{sss:def:sample-smooth}
    A \emph{sample-smooth space} is a set $\Xcal$ together with, for every $n \ge 0$, a set
    \begin{align}
        \Xcal^{\Omega_n} &\ins \Sets\lp\Omega_n,\ \Xcal\rp
    \end{align}
    of \emph{admissible plots}, such that
    \begin{enumerate}
        \item[(P1)] every constant map $\Omega_n\to\Xcal$ lies in $\Xcal^{\Omega_n}$;
        \item[(P2)] $p \in \Xcal^{\Omega_n}$ and $h \in \Site(\Omega_m,\Omega_n)$ imply
            $p \circ h \in \Xcal^{\Omega_m}$.
    \end{enumerate}
    A map $f:\,\Xcal\to\Ycal$ between sample-smooth spaces is \emph{sample-smooth} if
    $f\circ p \in \Ycal^{\Omega_n}$ for every $n$ and every $p \in \Xcal^{\Omega_n}$. The resulting category is
    $\SSS$.
\end{Def}

That is the whole definition. It has the same shape as \Cref{qus:def:qus} --- a set with families of
admissible maps out of test objects, closed under precomposition --- with the single test object
$\Omega$ replaced by the site $\Site$. The name follows the usage of \emph{sample-continuous} for a
process whose sample paths are continuous: a plot is smooth in the parameter for each fixed sample,
which is (M1) of \Cref{site:def:site}, and admissible in the sample for each fixed parameter, which
is (M2). The hyphen is adverbial --- \emph{smooth samplewise} --- and implies no order on the two
factors of $\Omega_n$; we write the smooth coordinate first, as in the customary $X_t(\omega)$.

\begin{Not}[Random variables, curves, the underlying quasi-universal space]
    \label{sss:not:derived}
    For $\Xcal \in \SSS$ we write
    \begin{align}
        \Xcal^\Omega &:= \Xcal^{\Omega_0}, &
        \Xcal^{\R^n} &:= \lC \gamma:\,\R^n\to\Xcal \st \gamma\circ\pr_{\R^n} \in \Xcal^{\Omega_n}\rC, &
        \Xcal^\R &:= \Xcal^{\R^1},
    \end{align}
    the \emph{admissible random variables}, the \emph{admissible $n$-parameter families} and the
    \emph{admissible curves}; by \Cref{site:prp:karoubi} these are the plot sets at the retracts
    $\R^n$ of $\Omega_n$, and $\Xcal^{\R^0} = \Xcal$. Throughout, $\R^n$ denotes the Cartesian space
    of \Cref{sss:eg:euclidean} when it is used as a shape and its underlying set, vector space or
    Borel space when it is not; the reading is always fixed by the position, and the one place where
    the difference is substantive --- the Borel $\sigma$-algebra $\Bcal^\circ_{\R^n}$ against the
    induced one $\Bcal_{\R^n}$ --- is marked by the superscript $\circ$. We write
    $\Und\Xcal := \lp\Xcal,\Xcal^\Omega\rp$, a quasi-universal space by (P1) and by closure under
    $\Xi^{(0)} = \Xi$, and $\Bcal_\Xcal := \Bcal_{\Und\Xcal}$ for its induced $\sigma$-algebra,
    \Cref{qus:def:induced-sigma}.
\end{Not}

\begin{Lem}[Curves are measurable, for free]
    \label{sss:lem:Q4}
    Let $\Xcal \in \SSS$ and $\gamma \in \Xcal^{\R^n}$, let
    $W \in \Meas\lp(\Omega,\Bcal_\Omega),(\R^n,\Bcal^\circ_{\R^n})\rp$
    and let $\Phi \in \Omega^\Omega$. Then $\lp\omega\mapsto\gamma(W(\omega))\rp \in \Xcal^\Omega$, and
    more generally $\lp\omega\mapsto p\lp W(\omega),\Phi(\omega)\rp\rp \in \Xcal^\Omega$ for every
    $p \in \Xcal^{\Omega_n}$.
    \begin{proof}
        The second statement is (P2) for the mixed morphism $\lb W,\Phi\rb$ of
        \Cref{site:not:morphisms}; the first is the second applied to
        $p := \gamma\circ\pr_{\R^n}$.
    \end{proof}
\end{Lem}

In a theory whose objects are pairs (random variables, curves), the statement of
\Cref{sss:lem:Q4} is an axiom --- it is the compatibility axiom linking the two families. Here it is
a one-line consequence, and this is the first of several places where the mixed site pays for
itself.

\begin{Lem}[Plots are homs]
    \label{sss:lem:plots-are-homs}
    Give $\Omega_n$ the plot structure $\Omega_n^{\Omega_m} := \Site(\Omega_m,\Omega_n)$. This is a
    sample-smooth space, and for every $\Xcal \in \SSS$
    \begin{align}
        \Xcal^{\Omega_n} &= \SSS\lp\Omega_n,\Xcal\rp, &
        \Xcal^{\R^n} &= \SSS\lp\R^n,\Xcal\rp,
    \end{align}
    the second with $\R^n$ the Cartesian space of \Cref{sss:eg:euclidean}. Moreover
    $p(-,\omega_0) \in \Xcal^{\R^n}$ for every $p \in \Xcal^{\Omega_n}$ and every
    $\omega_0 \in \Omega$.
    \begin{proof}
        (P1) holds because a constant map $\Omega_m\to\Omega_n$ is the $\Site$-morphism
        $(\mathrm{const}_{t_0},\mathrm{const}_{\omega_0})$, and (P2) is composition in $\Site$,
        \Cref{site:lem:category}. If $f \in \SSS(\Omega_n,\Xcal)$ then
        $f = f\circ\id_{\Omega_n} \in \Xcal^{\Omega_n}$, since $\id_{\Omega_n} \in \Omega_n^{\Omega_n}$; and if
        $f \in \Xcal^{\Omega_n}$ then $f\circ h \in \Xcal^{\Omega_m}$ for every $h \in \Omega_n^{\Omega_m}$ by (P2), so
        $f \in \SSS(\Omega_n,\Xcal)$.

        For the second equality, let $\gamma \in \Xcal^{\R^n}$, so
        $\gamma\circ\pr_{\R^n} \in \Xcal^{\Omega_n}$ by \Cref{sss:not:derived}. If
        $q \in (\R^n)^{\Omega_m} = \Cmix(\Omega_m,\R^n)$ then $(q,\id_\Omega)$ is a
        $\Site$-morphism $\Omega_m\to\Omega_n$ and
        $\gamma\circ q = (\gamma\circ\pr_{\R^n})\circ(q,\id_\Omega) \in \Xcal^{\Omega_m}$,
        so $\gamma \in \SSS(\R^n,\Xcal)$; conversely $\pr_{\R^n} \in (\R^n)^{\Omega_n}$, so a
        sample-smooth $\gamma$ has $\gamma\circ\pr_{\R^n} \in \Xcal^{\Omega_n}$. The last claim
        is (P2) applied to the $\Site$-morphism
        $(\pr_{\R^n},\mathrm{const}_{\omega_0}):\,\Omega_n\to\Omega_n$, whose seed part lies in
        $\Omega^\Omega$ by ($\Omega$1), followed by \Cref{sss:not:derived}.
    \end{proof}
\end{Lem}

\begin{Eg}[Cartesian spaces, and the point]
    \label{sss:eg:euclidean}
    $\R^k$ carries $\lp\R^k\rp^{\Omega_n} := \Cmix(\Omega_n,\R^k)$ of \Cref{site:def:site}, a plot
    structure by \Cref{site:lem:category}. Then $(\R^k)^\Omega = \Meas\lp(\Omega,\Bcal_\Omega),(\R^k,\Bcal^\circ_{\R^k})\rp$ is the set of
    \emph{universally measurable} maps $\Omega\to\R^k$ --- for which, by the argument of
    \Cref{qus:lem:borel-suffices}(1), it makes no difference whether one asks measurability into
    $\Bcal^\circ_{\R^k}$ or into its universal completion --- $(\R^k)^{\R^n} = C^\infty(\R^n,\R^k)$, and
    the induced $\sigma$-algebra $\Bcal_{\R^k}$ of \Cref{sss:not:derived} is the universal
    completion of $\Bcal^\circ_{\R^k}$: ``$\sni$'' because a universally measurable set has
    universally measurable preimage under a universally measurable map, and ``$\ins$'' because, by
    ($\Omega$5), there is a Borel isomorphism $j$ of $\R^k$ onto a Borel subset of $\Omega$; its
    inverse, extended by a constant, is a Borel and hence admissible $X \in (\R^k)^\Omega$, and an
    arbitrary Borel probability measure $\mu$ on $\R^k$ gives $P := j_*\mu \in \PrPfO$. For
    $A \in \Bcal_{\R^k}$ the set $X^{-1}(A)$ is $P$-measurable, and pulling a Borel sandwich back
    along $j$ exhibits $A$ as $\mu$-measurable. Note that the
    site could not have been phrased with $\Bcal_{\R^k}$ in place of $\Bcal^\circ_{\R^k}$: the
    induced $\sigma$-algebra is defined from the plot structure, which is defined from the site.
    We call $\R^k$ with this structure the \emph{Cartesian space} of dimension $k$. Further, $\one$ is the one-point space and $\zero$ the empty one.
\end{Eg}

\begin{Eg}[The three two-point spaces]
    \label{sss:eg:two}
    On $\lC0,1\rC$ we use
    \begin{align}
        \twoc^{\Omega_n} &:= \lC\text{constants}\rC, &
        \twob^{\Omega_n} &:= \lC \I_B \st B \in \Bcal_{\Omega_n} \rC, &
        \twop^{\Omega_n} &:= \Sets\lp\Omega_n,\lC0,1\rC\rp,
    \end{align}
    the \emph{constant}, the \emph{discrete} and the \emph{indiscrete} two-point space, where
    $\Bcal_{\Omega_n}$ is the induced $\sigma$-algebra of the object $\Omega_n$ of
    \Cref{sss:lem:plots-are-homs}. Each is a plot structure; for $\twob$ this is because a
    $\Site$-morphism is sample-smooth, hence measurable for the induced $\sigma$-algebras. Note that
    $\Bcal_{\Omega_n}$ is universally complete by \Cref{eq:induced-sigma-explicit} and contains the
    plain product $\sigma$-algebra $\Bcal^\circ_{\R^n}\otimes\Bcal_\Omega$ that
    \Cref{site:lem:category} delivers; the latter is used only as a generating family inside
    proofs.
\end{Eg}

\begin{Eg}[Rigid and loose spaces]
    \label{sss:eg:flat-sharp}
    For a quasi-universal space $\Scal$ put
    \begin{align}
        \lp\flat\Scal\rp^{\Omega_n} &:= \lC X \circ \pi_n \st X \in \Scal^\Omega\rC, &
        \lp\sharp\Scal\rp^{\Omega_n} &:= \lC p \st \forall h \in \Site(\Omega_0,\Omega_n).\ p \circ h \in
        \Scal^\Omega \rC .
    \end{align}
    Both are plot structures. $\flat\Scal$ has only the plots that ignore the smooth
    coordinate --- it is \emph{smoothly rigid}, $\flat\Scal^{\R^n} = $ constants --- and
    $\sharp\Scal$ has every plot that is compatible with the seed direction. These are the two
    modalities of \Cref{sec:modal}.
\end{Eg}

\begin{Eg}[Manifolds]
    \label{sss:eg:manifolds}
    A finite-dimensional smooth manifold $M$, second-countable and Hausdorff, carries
    \begin{align}
        M^{\Omega_n} &:= \Big\{\, p:\,\Omega_n\to M \ \Big|\ p(-,\omega) \in C^\infty(\R^n,M)
        \text{ for all } \omega, \nonumber \\
        &\qquad\qquad\qquad\qquad\qquad\ \
        p(t,-) \in \Meas\lp(\Omega,\Bcal_\Omega),(M,\Bcal_M)\rp \text{ for all } t \,\Big\},
        \label{eq:manifold-plots}
    \end{align}
    with $\Bcal_M$ the Borel $\sigma$-algebra. We refer to the two conditions as (M1$_M$) and
    (M2$_M$); for $M = \R^k$ they are (M1) and (M2) of \Cref{site:def:site}. This is a plot structure: (P1) is clear, and for (P2)
    argue as in \Cref{site:lem:category}, using \Cref{qus:lem:caratheodory} for the metrizable
    target $M$.
\end{Eg}

\subsection{The Abstract Criteria: Concrete Presheaves}
\label{sss:sec:quasitopos}

We now check that $\SSS$ is an instance of the standard machinery of \cite{BH11,Dub79,MMS22}, so
that the quasitopos property may be quoted rather than reproved. Throughout this subsection we use
the enlarged site $\Site_\pt$ of \Cref{site:def:site-plus}, which by \Cref{site:prp:karoubi} carries
the same structures.

\begin{Def}[Concrete site, concrete presheaf, \cite{BH11}]
    \label{sss:def:concrete-site}
    A \emph{concrete site} is a site $(\Ecal,J)$ such that
    \begin{enumerate}
        \item[(C1)] $\Ecal$ has a terminal object $\pt$;
        \item[(C2)] the functor $|\cdot| := \Ecal(\pt,-):\,\Ecal\to\Sets$ is faithful;
        \item[(C3)] for every covering family $\lC f_i:\,C_i\to C\rC_{i\in I} \in J(C)$ the maps
            $|f_i|$ are jointly surjective onto $|C|$.
    \end{enumerate}
    A presheaf $F$ on $\Ecal$ is \emph{concrete} if for every object $C$ the map
    \begin{align}
        F(C) &\longrightarrow \Sets\lp|C|,\ F(\pt)\rp, &
        u &\longmapsto \lp c \mapsto F(c)(u) \rp,
    \end{align}
    is injective. We write $\CPsh(\Ecal,J)$ for the full subcategory of concrete sheaves.
\end{Def}

\begin{Prp}[$\Site_\pt$ with the trivial topology is a concrete site]
    \label{sss:prp:concrete-site}
    Let $J_{\mathrm{triv}}$ be the trivial topology, whose only covering families are the singletons
    $\lC\id_C\rC$. Then $(\Site_\pt,J_{\mathrm{triv}})$ is a concrete site, every presheaf on it is
    a sheaf, and
    \begin{align}
        |\Omega_n| &= \R^n\times\Omega, & |\R^n| &= \R^n, & |\pt| &= \pt .
    \end{align}
    \begin{proof}
        (C1) $\pt = \R^0$ is terminal, \Cref{site:lem:products-plus}.

        (C2) $\Site_\pt(\pt,\Omega_n) = C^\infty(\R^0,\R^n)\times\Omega = \R^n\times\Omega$ and
        $\Site_\pt(\pt,\R^n) = \R^n$, and under these identifications $|f|$ is the underlying map of
        $f$. Since the morphisms of $\Site_\pt$ \emph{are} maps of the underlying sets, $|\cdot|$
        is faithful.

        (C3) is vacuous for singleton covers by identities, and for the same reason every presheaf
        satisfies the sheaf condition.
    \end{proof}
\end{Prp}

\begin{Prp}[$\SSS$ is the category of concrete presheaves]
    \label{sss:prp:sss-is-cpsh}
    $\SSS \cong \CPsh\lp\Site_\pt,J_{\mathrm{triv}}\rp$.
    \begin{proof}
        Let $F$ be a concrete presheaf, $\Xcal := F(\pt)$, and identify $F(C)$ with a subset of
        $\Sets(|C|,\Xcal)$ by (C2) and concreteness. Functoriality says exactly that these subsets
        are closed under precomposition with $\Site_\pt$-morphisms, and for $x \in \Xcal$ the
        element $F(!_C)(x) \in F(C)$ has underlying map $c \mapsto F(c)F(!_C)(x) = F(!_C\circ c)(x)
        = x$, the constant, so (P1) holds. Restricting to the objects $\Omega_n$ gives a
        sample-smooth space by \Cref{sss:def:sample-smooth}, and by \Cref{site:prp:karoubi} the values at
        $\R^n$ are recovered from it. Conversely a sample-smooth space extends uniquely to
        $\Site_\pt$ by \Cref{eq:karoubi}, and the extension is a concrete presheaf. Morphisms
        correspond on both sides to maps of underlying sets commuting with the structure.
    \end{proof}
\end{Prp}

\begin{Thm}[$\SSS$ is a quasitopos]
    \label{sss:thm:quasitopos}
    $\SSS$ is a concrete quasitopos: it is complete and cocomplete, cartesian closed and locally
    cartesian closed, and the indiscrete two-point space $\twop$ of \Cref{sss:eg:two} classifies
    strong monomorphisms. It is not a topos.
    \begin{proof}
        By \Cref{sss:prp:concrete-site,sss:prp:sss-is-cpsh}, $\SSS$ is the category of
        concrete sheaves on a concrete site, which is a quasitopos by \cite{BH11}; see also
        \cite{Dub79} and, for the version with additional structure, \cite{MMS22}. The explicit
        constructions, and direct verifications of completeness, cocompleteness, cartesian
        closedness and the classifier, are given in \Cref{sss:sec:constructions,sss:sec:cc}; \Cref{sss:cor:not-topos} shows it is not a topos.
    \end{proof}
\end{Thm}

We give the explicit constructions anyway, for three reasons: they are what one computes with; the
direct proofs are short; and the reader who wants to compare with \cite{For21} will find them line
for line the constructions there, with the single test object $\Omega$ replaced by $\Site$.

\subsection{Limits, Colimits, Subspaces and Quotients}
\label{sss:sec:constructions}

\begin{Lem}[Initial and final structures]
    \label{sss:lem:initial-final}
    Let $\Xcal$ be a set.
    \begin{enumerate}
        \item For any family of maps $f_j:\,\Xcal\to\Ycal_j$ into sample-smooth spaces, the families
            $\Xcal^{\Omega_n} := \lC p \st \forall j.\ f_j\circ p \in \Ycal_j^{\Omega_n}\rC$ form a sample-smooth
            structure, the \emph{initial} one: it is the largest for which all $f_j$ are
            sample-smooth, and a map $g:\,\Zcal\to\Xcal$ is sample-smooth iff every $f_j\circ g$ is.
        \item For any family of maps $e_j:\,\Wcal_j\to\Xcal$ out of sample-smooth spaces, the
            families $\Xcal^{\Omega_n} := \lC \text{constants}\rC \cup \bigcup_j e_j\circ\Wcal_j^{\Omega_n}$ form a
            sample-smooth structure, the \emph{final} one: it is the smallest for which all $e_j$ are
            sample-smooth, and $g:\,\Xcal\to\Zcal$ is sample-smooth iff every $g\circ e_j$ is.
    \end{enumerate}
    In particular $\SSS$ is topological over $\Sets$, hence complete and cocomplete, with limits and
    colimits computed on underlying sets and equipped with the initial resp.\ final structure.
    \begin{proof}
        Both families contain the constants and are closed under precomposition by (P2) for the
        $\Ycal_j$ resp.\ $\Wcal_j$; the universal properties are immediate from the definitions.
        Existence of initial structures for arbitrary sources is the definition of a topological
        category over $\Sets$, and such categories are complete and cocomplete with limits and
        colimits lifted from $\Sets$.
    \end{proof}
\end{Lem}

\begin{Prp}[Products and coproducts]
    \label{sss:prp:prod-coprod}
    Let $\lp\Xcal_i\rp_{i \in I}$ be sample-smooth spaces. Then
    \begin{align}
        \lp\textstyle\prod_i\Xcal_i\rp^{\Omega_n} &= \textstyle\prod_i \Xcal_i^{\Omega_n}, &
        \lp\textstyle\coprod_i\Xcal_i\rp^{\Omega_n} &= \textstyle\bigcup_i \iota_i \circ \Xcal_i^{\Omega_n},
    \end{align}
    the product taken on underlying sets. In particular a plot of a coproduct takes its values in a
    single summand, so $\one \sqcup \one = \twoc$ and not $\twob$: coproducts are \emph{rigid}.
    \begin{proof}
        The product carries the initial structure for the projections, which is the displayed
        family; the coproduct carries the final structure for the injections, which is the displayed
        family since the constants are already contained in it (for $I \ne \emptyset$).
    \end{proof}
\end{Prp}

\begin{Rem}[A plot cannot jump]
    \label{sss:rem:no-jump}
    \Cref{sss:prp:prod-coprod} says more than the corresponding statement for curves alone, namely
    that a smooth curve into a coproduct stays in one summand, which is a connectedness fact about
    $\R$. Here it also says that a \emph{random} plot cannot jump: a map
    $\Omega_n\to\Xcal_1\sqcup\Xcal_2$ that is a plot must land in one summand for all
    $(t,\omega)$ at once, not merely for each $\omega$ separately. This is the rigidity of
    \cite{For21} on the $\Omega$-side and of connectedness on the $\R$-side, in one statement, and
    it is why we do not impose patching, \Cref{site:rem:no-locality}(3).
\end{Rem}

\begin{Def}[Embeddings and quotients]
    \label{sss:def:emb-quot}
    A sample-smooth $f:\,\Xcal\to\Ycal$ is an \emph{embedding} if it is injective and $\Xcal$ carries
    the initial structure, $\Xcal^{\Omega_n} = \lC p \st f \circ p \in \Ycal^{\Omega_n}\rC$, and a \emph{quotient}
    if it is surjective and $\Ycal$ carries the final structure,
    $\Ycal^{\Omega_n} = \lC\text{constants}\rC\cup f\circ\Xcal^{\Omega_n}$.
\end{Def}

\begin{Prp}[Equalizers, coequalizers, monos, epis]
    \label{sss:prp:eq-coeq}
    Let $f,g:\,\Xcal\to\Ycal$ be sample-smooth. The equalizer of $(f,g)$ is the set
    $\lC x \st f(x) = g(x)\rC$ with the subspace structure of \Cref{sss:def:emb-quot}, and the
    coequalizer is the set-theoretic coequalizer with the final structure. Monomorphisms are exactly
    the injective sample-smooth maps and epimorphisms exactly the surjective ones; the strong
    monomorphisms are exactly the embeddings and the strong epimorphisms exactly the quotients.
    \begin{proof}
        The first two are \Cref{sss:lem:initial-final}. A sample-smooth map is monic iff it is
        injective, because $\one$ is a generator: two maps $\Zcal\to\Xcal$ agreeing after $f$ agree
        pointwise. It is epic iff it is surjective, because $\twop$ is a cogenerator: every map of
        underlying sets into $\twop$ is sample-smooth, so if $f:\,\Xcal\to\Ycal$ is not surjective
        then $\I_{f(\Xcal)}$ and $\mathrm{const}_1$ are distinct morphisms $\Ycal\to\twop$
        agreeing after $f$; the converse is immediate on underlying maps.

        A strong mono is a mono right-orthogonal to all epis. Let $m:\,\Xcal\to\Ycal$ be a strong
        mono, write $\Xcal'$ for the underlying set of $\Xcal$ with the initial structure along $m$,
        and note that $\id:\,\Xcal\to\Xcal'$ is sample-smooth --- since
        $\Xcal^{\Omega_n}\ins\Xcal'^{\Omega_n}$ --- and epic, being surjective. Orthogonality in
        the square $m\circ\id = m$ produces a sample-smooth lift $\Xcal'\to\Xcal$ over the
        identity, whence $\Xcal'^{\Omega_n}\ins\Xcal^{\Omega_n}$ and $m$ is an embedding.
        Conversely an embedding has the lifting property, because a lift exists on underlying sets
        (the epi being surjective, the mono injective) and is sample-smooth by initiality.

        Dually, a quotient is left-orthogonal to every mono, the lift existing on underlying sets
        and being sample-smooth by finality, so quotients are strong epis; and if $f$ is a strong
        epi then, writing $\Ycal''$ for the underlying set of $\Ycal$ with the final structure
        along $f$, the map $\id:\,\Ycal''\to\Ycal$ is a mono, and orthogonality in
        $\id\circ f = f$ gives a sample-smooth $\Ycal\to\Ycal''$ over the identity, so
        $\Ycal^{\Omega_n}\ins\Ycal''^{\Omega_n}$ and $f$ is a quotient.
    \end{proof}
\end{Prp}

\begin{Prp}[$\twop$ classifies embeddings]
    \label{sss:prp:classifier}
    Let $\top:\,\one\to\twop$ pick out $1$. For every embedding $m:\,\Xcal\inj\Ycal$ there is a
    unique sample-smooth $\chi_m:\,\Ycal\to\twop$ making
    \begin{align}
        \begin{array}{ccc}
            \Xcal & \longrightarrow & \one \\
            \downarrow m & & \downarrow \top \\
            \Ycal & \stackrel{\chi_m}{\longrightarrow} & \twop
        \end{array}
    \end{align}
    a pullback, namely the indicator of $m(\Xcal)$.
    \begin{proof}
        Every map into $\twop$ is sample-smooth, since $\twop^{\Omega_n}$ is all maps; so $\chi_m$ is a
        morphism and is unique as a map of sets. The pullback of $\top$ along $\chi_m$ is
        $\chi_m^{-1}(1) = m(\Xcal)$ with the initial structure, \Cref{sss:lem:initial-final}(1),
        which is $\Xcal$ precisely because $m$ is an embedding.
    \end{proof}
\end{Prp}

\begin{Cor}[Not a topos]
    \label{sss:cor:not-topos}
    $\SSS$ is not a topos. Indeed the identity of $\lC0,1\rC$ is a monomorphism
    $\twoc \to \twob$ which is not an embedding: the initial structure it induces on the source is
    $\twob$, whereas $\twoc^\Omega$ consists of the constants alone, and the two differ.
\end{Cor}

\subsection{Function Spaces, Fibre Products and Internal Homs}
\label{sss:sec:cc}

\begin{Def}[The function space]
    \label{sss:def:function-space}
    For $\Xcal,\Ycal \in \SSS$ let $\Ycal^\Xcal := \SSS(\Xcal,\Ycal)$ with
    \begin{align}
        \lp\Ycal^\Xcal\rp^{\Omega_n} &:= \Big\{\, F:\,\Omega_n \to \Ycal^\Xcal \ \Big|\
        \forall m\ \forall h \in \Site(\Omega_m,\Omega_n)\ \forall q \in \Xcal^{\Omega_m}: \nonumber \\
        &\qquad\qquad\qquad\qquad\qquad\qquad\quad
        \lp u \mapsto F(h(u))\lp q(u)\rp\rp \in \Ycal^{\Omega_m} \,\Big\} ,
        \label{eq:function-space}
    \end{align}
    where $u$ ranges over $\Omega_m$.
\end{Def}

\begin{Lem}[The function space is well defined]
    \label{sss:lem:function-space-wd}
    \Cref{eq:function-space} defines a sample-smooth structure. Moreover, if one starts from maps
    $F:\,\Omega_n\to\Sets(\Xcal,\Ycal)$ and imposes \Cref{eq:function-space}, then every
    value $F(v)$ is automatically sample-smooth, so no separate hypothesis is needed.
    \begin{proof}
        (P1): for constant $F \equiv f$ with $f$ sample-smooth, $u \mapsto f(q(u))$ lies in
        $\Ycal^{\Omega_m}$. (P2): for $h_0 \in \Site(\Omega_k,\Omega_n)$ the condition for $F\circ h_0$
        quantifies over $h_0\circ h$, which ranges inside $\Site(\Omega_m,\Omega_n)$.

        For the last claim fix $v \in \Omega_n$ and take $h$ to be the constant morphism
        $\Omega_m\to\Omega_n$ at $v$, which lies in $\Site$ by \Cref{sss:lem:plots-are-homs}. Then
        $u \mapsto F(v)(q(u))$ lies in $\Ycal^{\Omega_m}$ for every $q \in \Xcal^{\Omega_m}$, i.e.\ $F(v)$ is
        sample-smooth.
    \end{proof}
\end{Lem}

\begin{Thm}[Cartesian closedness]
    \label{sss:thm:cartesian-closed}
    $\SSS$ is cartesian closed: evaluation $\ev:\,\Ycal^\Xcal\times\Xcal\to\Ycal$ is sample-smooth
    and for every $\Zcal$ the map
    \begin{align}
        \cur:\; \SSS\lp\Zcal\times\Xcal,\ \Ycal\rp &\bij \SSS\lp\Zcal,\ \Ycal^\Xcal\rp, &
        \cur(f)(z)(x) &= f(z,x),
    \end{align}
    is a bijection, natural in all three variables.
    \begin{proof}
        \emph{Evaluation.} A plot of $\Ycal^\Xcal\times\Xcal$ on $\Omega_m$ is a pair $(F,q)$ with
        $F \in (\Ycal^\Xcal)^{\Omega_m}$ and $q \in \Xcal^{\Omega_m}$, by \Cref{sss:prp:prod-coprod}. Taking
        $h := \id_{\Omega_m}$ in \Cref{eq:function-space} gives
        $u\mapsto F(u)(q(u)) \in \Ycal^{\Omega_m}$, which is $\ev\circ(F,q)$.

        \emph{$\cur$ is well defined.} Let $f \in \SSS(\Zcal\times\Xcal,\Ycal)$ and
        $r \in \Zcal^{\Omega_n}$. For $h \in \Site(\Omega_m,\Omega_n)$ and $q \in \Xcal^{\Omega_m}$ the pair
        $(r\circ h,\,q)$ is a plot of $\Zcal\times\Xcal$ on $\Omega_m$, so
        $u\mapsto f\lp r(h(u)),q(u)\rp$ lies in $\Ycal^{\Omega_m}$; this is \Cref{eq:function-space} for
        $\cur(f)\circ r$, so $\cur(f)$ is sample-smooth.

        \emph{$\cur$ is bijective.} Its inverse is $g \mapsto \unc(g)$, $\unc(g)(z,x) := g(z)(x)$,
        which is sample-smooth by the evaluation step applied to $(g\circ r,q)$. The two
        constructions are mutually inverse on underlying maps.
    \end{proof}
\end{Thm}

\begin{Prp}[Fibre products]
    \label{sss:prp:fibre}
    For sample-smooth $f:\,\Xcal\to\Zcal$ and $g:\,\Ycal\to\Zcal$ the fibre product
    $\Xcal\times_\Zcal\Ycal$ is the set-theoretic one with the subspace structure inherited from
    $\Xcal\times\Ycal$, i.e.
    \begin{align}
        \lp\Xcal\times_\Zcal\Ycal\rp^{\Omega_n} &= \lC (p,q) \in \Xcal^{\Omega_n}\times
        \Ycal^{\Omega_n} \st f\circ p = g\circ q \rC .
    \end{align}
    \begin{proof}
        \Cref{sss:lem:initial-final}(1) applied to the equalizer of $f\circ\pr_\Xcal$ and
        $g\circ\pr_\Ycal$ inside $\Xcal\times\Ycal$.
    \end{proof}
\end{Prp}

\begin{Prp}[Internal homs over a base]
    \label{sss:prp:internal-hom}
    Let $f:\,\Xcal\to\Zcal$ and $g:\,\Ycal\to\Zcal$ be sample-smooth, and give each fibre
    $f^{-1}(z)$, $g^{-1}(z)$ its subspace structure. Put
    \begin{align}
        \ihom_\Zcal(f,g) &:= \bigsqcup_{z \in \Zcal} \SSS\lp f^{-1}(z),\ g^{-1}(z)\rp,
    \end{align}
    with $\varpi:\,\ihom_\Zcal(f,g) \to \Zcal$ sending the summand at $z$ to $z$, and let
    \begin{align}
        \ihom_\Zcal(f,g)^{\Omega_n} &:= \Big\{\, F \ \Big|\ \varpi\circ F \in \Zcal^{\Omega_n}
        \text{ and, for all } m,\ h \in \Site(\Omega_m,\Omega_n) \text{ and} \nonumber \\
        &\qquad q \in \Xcal^{\Omega_m} \text{ with } f\circ q = \varpi\circ F\circ h, \quad
        \lp u \mapsto F(h(u))\lp q(u)\rp\rp \in \Ycal^{\Omega_m} \,\Big\} .
        \label{eq:internal-hom}
    \end{align}
    Then $\ihom_\Zcal(f,g)$ is the exponential $g^f$ in the slice category $\SSS_{/\Zcal}$. In
    particular $\SSS$ is locally cartesian closed.
    \begin{proof}
        \Cref{eq:internal-hom} is a sample-smooth structure, and the adjunction
        \begin{align}
            \SSS_{/\Zcal}\lp e \times_\Zcal f,\ g\rp &\cong
            \SSS_{/\Zcal}\lp e,\ \ihom_\Zcal(f,g)\rp
        \end{align}
        is verified exactly as in \Cref{sss:thm:cartesian-closed}, every step read over $\Zcal$
        and the fibre products computed by \Cref{sss:prp:fibre}; for $\Zcal = \one$ it \emph{is}
        \Cref{sss:thm:cartesian-closed}. Local cartesian closedness is also part of
        \Cref{sss:thm:quasitopos}, and we record the formula because it is what one computes with.
    \end{proof}
\end{Prp}

\subsection{Correctness on Cartesian Spaces and Manifolds}
\label{sss:sec:correctness}

\begin{Thm}[Euclidean correctness]
    \label{sss:thm:euclidean}
    $\SSS\lp\R^k,\R^l\rp = C^\infty\lp\R^k,\R^l\rp$, and more generally
    $\lp\R^l\rp^{\R^k} = C^\infty(\R^k,\R^l)$.
    \begin{proof}
        A $C^\infty$ map $f$ preserves plots: if $p \in \lp\R^k\rp^{\Omega_n}$ then $f\circ p$ satisfies (M1) by
        the chain rule and (M2) because $f$ is Borel. Conversely let $f$ be sample-smooth. The map
        $\mathrm{pr}:\,(t,\omega)\mapsto t$ lies in $\lp\R^k\rp^{\Omega_k}$, so $f\circ\mathrm{pr} \in \lp\R^l\rp^{\Omega_k}$,
        which by (M1) says that $t \mapsto f(t)$ is $C^\infty$. The second claim is the same
        argument.
    \end{proof}
\end{Thm}

\begin{Rem}[No Boman theorem is needed]
    \label{sss:rem:no-boman}
    In a theory whose smooth test object is the line, \Cref{sss:thm:euclidean} is exactly Boman's
    theorem \cite{Bom67}: a map $\R^k\to\R^l$ taking smooth curves to smooth curves is $C^\infty$.
    Here it is immediate, because the site contains $\Omega_k$ and hence the plot $(t,\omega)\mapsto
    t$, which is the identity of $\R^k$ in disguise. The same remark applies to
    \Cref{sss:thm:manifolds}. This is a genuine simplification of the foundations, and it is bought
    with the loss recorded in \Cref{sss:rem:curves-do-not-suffice}.
\end{Rem}

\begin{Thm}[Manifolds embed]
    \label{sss:thm:manifolds}
    With the structure of \Cref{sss:eg:manifolds}, the assignment $M \mapsto M$ is a full and
    faithful functor from the category of finite-dimensional second-countable Hausdorff smooth
    manifolds and smooth maps into $\SSS$, and it preserves finite products.
    \begin{proof}
        Faithfulness is clear, morphisms being maps of underlying sets. If $f:\,M\to N$ is smooth
        then it preserves plots, by the chain rule and by Borel measurability of $f$. Conversely
        let $f$ be sample-smooth and let $x \in M$. Choose a chart around $x$ whose image is an open
        ball and compose with a diffeomorphism $\R^m \bij$ ball, giving
        $\varphi \in C^\infty(\R^m,M)$ open onto its image with $x$ in that image. Then
        $(t,\omega)\mapsto\varphi(t)$ lies in $M^{\Omega_m}$, so $f\circ\varphi \in N^{\Omega_m}$, whose (M1$_N$) says
        $f\circ\varphi \in C^\infty(\R^m,N)$; as $\varphi$ is a diffeomorphism onto an open
        neighbourhood of $x$, $f$ is smooth near $x$. Products: the plot family of $M\times N$ in
        the sense of \Cref{sss:eg:manifolds} is the product family of
        \Cref{sss:prp:prod-coprod}, since smoothness and measurability into a product are
        componentwise.
    \end{proof}
\end{Thm}

\begin{Prp}[The underlying diffeological space]
    \label{sss:prp:diffeology}
    For $\Xcal \in \SSS$ let $\mathrm{Dfg}(\Xcal)$ be the set $\Xcal$ with, for each $n \ge 0$ and
    each open $V \ins \R^n$, the plots
    \begin{align}
        \mathrm{Dfg}(\Xcal)_V &:= \Big\{\, p:\,V\to\Xcal \ \Big|\ \text{every point of } V
        \text{ has a neighbourhood } V' \nonumber \\
        &\qquad\qquad\qquad\qquad \text{on which } p \text{ factors as } \gamma\circ\varphi
        \,\Big\},
    \end{align}
    the factorisation running over $m \ge 0$, $\gamma \in \Xcal^{\R^m}$ and
    $\varphi \in C^\infty(V',\R^m)$. Then $\mathrm{Dfg}(\Xcal)$ is a diffeological space in the
    sense of \cite{Sou80,IZ13}, $\mathrm{Dfg}$ is a functor $\SSS\to\Diff$, and for a manifold $M$
    as in \Cref{sss:eg:manifolds} the diffeology $\mathrm{Dfg}(M)$ is the usual one, consisting of
    all smooth maps $V \to M$.
    \begin{proof}
        Constants factor with $m := 0$. If $\psi \in C^\infty(V'',V)$ and $p$ is a plot then
        $p\circ\psi$ is one, since locally $p\circ\psi = \gamma\circ(\varphi\circ\psi)$. The sheaf
        condition holds because the defining property is local by construction. For functoriality,
        a sample-smooth $f$ sends $\Xcal^{\R^m}$ into $\Ycal^{\R^m}$, so it sends a local
        factorisation $\gamma\circ\varphi$ to $(f\circ\gamma)\circ\varphi$.

        For a manifold, $M^{\R^m} = C^\infty(\R^m,M)$ by \Cref{sss:thm:manifolds}, so every
        $\mathrm{Dfg}(M)$-plot is locally smooth, hence smooth. Conversely let $p \in C^\infty(V,M)$
        and $x \in V$; choose a ball $V' \ni x$ inside $V$ and a diffeomorphism
        $\varphi:\,V' \bij \R^n$, and put $\gamma := p|_{V'}\circ\varphi^{-1} \in C^\infty(\R^n,M)$.
        Then $p|_{V'} = \gamma\circ\varphi$.
    \end{proof}
\end{Prp}

\begin{Rem}[There is no canonical functor back]
    \label{sss:rem:diffeology}
    \Cref{sss:prp:diffeology} goes one way only, and the reason is the subject of this paper. A
    diffeological space carries no admissible random variables, so producing a sample-smooth space
    from one requires supplying them; and once they are supplied, the mixed plots are \emph{still} a
    further choice, by \Cref{site:rem:mixed}. The supplied data must first be compatible in the
    sense of \Cref{sss:lem:Q4}, that $\gamma\circ W \in \Xcal^\Omega$ for every global plot
    $\gamma$ of the diffeology and every measurable $W$; without that the two kinds of datum do not
    fit together at all, since (P2) applied to $\lb W,\Phi\rb$ forces $\gamma\circ W$ into the
    random variables. Given compatibility, the generated structure --- the smallest plot family
    containing $\lC\gamma\circ\pr_{\R^m}\rC \cup \lC X\circ\pi_n\rC$, for $\gamma$ a global
    plot and $X$ a supplied random variable, and closed under $\Site$ --- is the smallest choice and
    \Cref{eq:maximal-extension} the largest; whether the two ends of that interval can differ is
    \Cref{sss:rem:curves-do-not-suffice}. Note also that $\mathrm{Dfg}$ forgets in the
    smooth direction as well, since a diffeology is determined by its plots on all open $V$ whereas
    a sample-smooth space knows only the global ones; that is the locality discussed in
    \Cref{site:rem:no-locality,disc:rem:locality}.
\end{Rem}

\begin{Rem}[Curves no longer suffice]
    \label{sss:rem:curves-do-not-suffice}
A sample-smooth space is not \emph{presented} by $\lp\Xcal^\Omega,\Xcal^\R\rp$: the plot
    families at $\Omega_n$, $n \ge 1$, are separate data rather than something computed from the
    random variables and the curves. What they are constrained by is one inclusion. The largest
    structure with given $\Xcal^\Omega$ and $\Xcal^{\R^n}$ is
    \begin{align}
        \widehat\Xcal^{\Omega_n} &:= \lC p \st p(-,\omega) \in \Xcal^{\R^n}\ \forall \omega,\ \text{ and }\
        p \circ h \in \Xcal^\Omega\ \forall h \in \Site(\Omega_0,\Omega_n) \rC ,
        \label{eq:maximal-extension}
    \end{align}
    --- that this is the largest such structure is immediate, since by
    \Cref{sss:lem:plots-are-homs} every $p \in \Xcal^{\Omega_n}$ satisfies both conditions ---
    and $\Xcal^{\Omega_n} \ins \widehat\Xcal^{\Omega_n}$ may in principle be strict. This is the
    price of \Cref{site:rem:mixed}, and it is not a defect to be repaired: it is exactly the room in
    which \Cref{prob:def:PrPf} is made. \Cref{prob:rem:comparison} identifies the question of
    equality in \Cref{eq:maximal-extension} for $\Xcal := \PrPf(\Ycal)$ as a measurable selection
    problem.

    Three things should be said precisely. First, $\widehat\Xcal$ is itself an object of $\SSS$:
    \Cref{eq:maximal-extension} contains the constants and is closed under precomposition, because
    $\Xcal^{\R^n}$ is closed under precomposition with smooth maps by \Cref{site:prp:karoubi} and
    because $h_0\circ h$ again lies in $\Site(\Omega_0,\Omega_n)$. Second, $\widehat\Xcal$ has the
    same underlying set, the same random variables and the same $n$-parameter families as $\Xcal$
    --- see the proof of \Cref{sss:prp:reflection}(1) --- and the identity is a sample-smooth
    map $\Xcal\to\widehat\Xcal$; so the comparison is between
    two genuine objects, and it is an isomorphism exactly when \Cref{eq:maximal-extension} is an
    equality. Third, and honestly: we know of \emph{no} $\Xcal$ for which the inclusion is strict.
    It is an equality for $\R^k$ and for manifolds, where the first condition of
    \Cref{eq:maximal-extension} is the smoothness condition defining the plots and the second, taken
    at $h := \iota^n_t$, is the measurability condition; for
    $\flat\Scal$, whose plots are the $t$-independent ones on both sides; and for $\sharp\Scal$,
    where both sides are defined by the same substitution condition; and it is inherited by
    products, both conditions being componentwise. Whether it can fail at all --- and in particular
    whether it fails at $\PrPf(\Ycal)$, which is the question of \Cref{prob:rem:comparison} --- we
    do not know. Until that is settled one cannot rule out that every object of $\SSS$ is after all
    recovered from its random variables together with its $n$-parameter families --- which is less
    than being recovered from its \emph{curves}, since we have not shown that $\Xcal^{\R^n}$ is
    determined by $\Xcal^\R$ either.
\end{Rem}

The comparison of \Cref{eq:maximal-extension} is not merely an inclusion of sets: it is a
reflection, and saying so turns the informal phrase ``presented by its curves and its random
variables'' into the name of a subcategory.

\begin{Prp}[The maximal extension is a reflection]
    \label{sss:prp:reflection}
    The assignment $\Xcal \mapsto \widehat\Xcal$ of \Cref{eq:maximal-extension} extends to a
    functor $\SSS\to\SSS$ which is the identity on underlying sets and on maps, and:
    \begin{enumerate}
        \item $\widehat{\widehat\Xcal} = \widehat\Xcal$, and the identity
            $\eta_\Xcal:\,\Xcal\to\widehat\Xcal$ is a natural transformation
            $\id_\SSS\Rightarrow\widehat{(-)}$;
        \item writing $\SSS_{\mathrm{pr}}$ for the full subcategory of the \emph{presented}
            objects, those with $\Xcal = \widehat\Xcal$, the functor $\widehat{(-)}$ is left
            adjoint to the inclusion $\SSS_{\mathrm{pr}}\ins\SSS$; that is,
            $\SSS_{\mathrm{pr}}$ is a reflective subcategory and $\widehat{(-)}$ its reflector,
            with $\eta$ the unit;
        \item $\SSS_{\mathrm{pr}}$ contains $\R^k$, every manifold, every rigid and every loose
            object, and is closed under products.
    \end{enumerate}
    An object lies in $\SSS_{\mathrm{pr}}$ exactly when its mixed plots are determined by its
    random variables together with its $n$-parameter families, so $\SSS_{\mathrm{pr}}$ is the
    category a presentation by those two towers would give --- which is less than a presentation by
    curves, as the last sentence of \Cref{sss:rem:curves-do-not-suffice} notes; and
    \Cref{prob:rem:comparison} is the question whether $\PrPf$ lands in it.
    \begin{proof}
        Let $f \in \SSS(\Xcal,\Ycal)$. Then $f$ carries $\Xcal^{\R^n}$ into $\Ycal^{\R^n}$,
        since $\gamma\circ\pr_{\R^n} \in \Xcal^{\Omega_n}$ gives
        $f\circ\gamma\circ\pr_{\R^n} \in \Ycal^{\Omega_n}$, and it carries $\Xcal^\Omega$
        into $\Ycal^\Omega$. So for $p \in \widehat\Xcal^{\Omega_n}$ both defining conditions
        of \Cref{eq:maximal-extension} pass to $f\circ p$, and
        $f \in \SSS(\widehat\Xcal,\widehat\Ycal)$: the assignment is functorial. It is the
        identity on underlying sets by construction, so naturality of $\eta$ is the commuting of a
        square of identities, and $\eta_\Xcal$ is sample-smooth because
        $\Xcal^{\Omega_n} \ins \widehat\Xcal^{\Omega_n}$.

        (1) $\widehat\Xcal$ has the same random variables and the same $n$-parameter families as
        $\Xcal$. For the first, take $h := \id_\Omega \in \Site(\Omega_0,\Omega_0)$ in the
        second condition of \Cref{eq:maximal-extension}, giving
        $\widehat\Xcal^\Omega \ins \Xcal^\Omega$, the reverse inclusion being
        $\Xcal^{\Omega_n}\ins\widehat\Xcal^{\Omega_n}$ at $n=0$. For the second, if
        $\gamma\circ\pr_{\R^n} \in \widehat\Xcal^{\Omega_n}$ then the first condition at any
        $\omega$ gives $\gamma \in \Xcal^{\R^n}$; again the reverse inclusion is immediate. Since
        \Cref{eq:maximal-extension} depends on $\Xcal$ only through those two,
        $\widehat{\widehat\Xcal} = \widehat\Xcal$.

        (2) Let $\Ycal = \widehat\Ycal$ and let $f \in \SSS(\Xcal,\Ycal)$; we must show that
        the same map of sets lies in $\SSS(\widehat\Xcal,\Ycal)$, and uniqueness is then
        automatic, $\eta_\Xcal$ being the identity on underlying sets. Let
        $p \in \widehat\Xcal^{\Omega_n}$. For each $\omega$ we have
        $p(-,\omega) \in \Xcal^{\R^n}$, hence $f\circ p(-,\omega) \in \Ycal^{\R^n}$; and for
        each $h \in \Site(\Omega_0,\Omega_n)$ we have $p\circ h \in \Xcal^\Omega$, hence
        $f\circ p\circ h \in \Ycal^\Omega$. So $f\circ p \in \widehat\Ycal^{\Omega_n} =
        \Ycal^{\Omega_n}$.

        (3) is the list of equalities established in
        \Cref{sss:rem:curves-do-not-suffice}.
    \end{proof}
\end{Prp}

\begin{Rem}[What the reflection is good for]
    \label{sss:rem:reflection}
    Three things. It names the comparison: the whole difference between the present category and one
    presented by $n$-parameter families and random variables is the difference between $\SSS$ and
    $\SSS_{\mathrm{pr}}$, and \Cref{prob:rem:comparison} asks whether that difference is visible
    at $\PrPf(\Ycal)$ --- indeed whether it is visible anywhere. It sorts out the two fallbacks
    alluded to in \Cref{disc:sec:related}: on $\SSS_{\mathrm{pr}}$ one has the endofunctor
    $\widehat{\PrPf(-)}$, whose multiplication has no reason to survive the reflection, and one has
    $\PrPf$ itself read as a functor $\SSS_{\mathrm{pr}}\to\SSS$, which is a monad
    \emph{relative} to the inclusion \cite{ACU15,Man76} --- automatically so, the inclusion being
    fully faithful and $\PrPf$ a monad on $\SSS$. And it locates the cost: $\SSS_{\mathrm{pr}}$, being
    reflective, inherits limits from $\SSS$ but computes colimits by reflecting them, so the two
    categories agree on everything in \Cref{sec:sss} that is built from limits and differ at most
    on quotients --- of which \Cref{prob:def:PrPf} is one.
\end{Rem}

\subsection{Tangent and Cotangent Spaces, Differentials and Jacobians}
\label{sss:sec:tangent}

Every object of $\SSS$ has tangent and cotangent spaces, defined from curves in the usual kinematic
way, and on Cartesian spaces and manifolds they are the classical ones. Nothing in this subsection is
used later; it is here because a category advertised for differentiable programming should be shown
to differentiate.

\begin{Def}[Tangent bundle]
    \label{tan:def:tangent}
    Let $\Xcal \in \SSS$ and let $\Xcal^{\R} = \SSS(\R,\Xcal)$ be the function space of
    \Cref{sss:def:function-space}, whose underlying set is the set of admissible curves. Define an
    equivalence relation on it by
    \begin{align}
        \gamma \sim \eta \quad:\Longleftrightarrow\quad \gamma(0) = \eta(0) \ \text{ and }\
        \lp f\circ\gamma\rp'(0) = \lp f\circ\eta\rp'(0) \ \text{ for all } f \in \SSS(\Xcal,\R),
        \label{eq:tangent-equiv}
    \end{align}
    which makes sense because $f\circ\gamma \in \R^{\R} = C^\infty(\R,\R)$ by
    \Cref{sss:thm:euclidean}. Put $T\Xcal := \Xcal^{\R}/\!\sim$ with the quotient structure of
    \Cref{sss:def:emb-quot}, write $[\gamma]$ for the class of $\gamma$, and let
    $\pi_\Xcal:\,T\Xcal\to\Xcal$, $[\gamma]\mapsto\gamma(0)$. The \emph{tangent space at $x$} is the
    subobject $T_x\Xcal := \pi_\Xcal^{-1}(x)$.
\end{Def}

\begin{Prp}[$T$ is a functor with a scalar action]
    \label{tan:prp:functor}
    $\pi_\Xcal$ is sample-smooth. For sample-smooth $f:\,\Xcal\to\Ycal$ the assignment
    $Tf[\gamma] := [f\circ\gamma]$ is well defined and sample-smooth, $T$ is an endofunctor of
    $\SSS$ with $\pi$ natural, and
    \begin{align}
        \R\times T\Xcal &\longrightarrow T\Xcal, &
        \lp\lambda,[\gamma]\rp &\longmapsto \lambda\cdot[\gamma] := \lB t\mapsto\gamma(\lambda t)\rB,
    \end{align}
    is sample-smooth and restricts to a scalar action on each $T_x\Xcal$ fixing
    $[\mathrm{const}_x]$. We write $df_x := T_xf:\,T_x\Xcal\to T_{f(x)}\Ycal$ for the
    \emph{differential} of $f$ at $x$; the transpose $df^*_x$ on cotangent vectors is
    \Cref{tan:prp:cotangent} below.
    \begin{proof}
        $\pi_\Xcal$ composed with the quotient map is $\ev_0$, which is sample-smooth by
        \Cref{sss:thm:cartesian-closed}, so $\pi_\Xcal$ is sample-smooth by finality,
        \Cref{sss:lem:initial-final}(2). If $\gamma\sim\eta$ then, for $g \in \SSS(\Ycal,\R)$, we
        have $g\circ f \in \SSS(\Xcal,\R)$, so $(g\circ f\circ\gamma)'(0) = (g\circ f\circ\eta)'(0)$;
        together with $f(\gamma(0)) = f(\eta(0))$ this gives $f\circ\gamma\sim f\circ\eta$, so $Tf$
        is well defined, and it is sample-smooth by
        finality because $\gamma\mapsto f\circ\gamma$ is sample-smooth
        $\Xcal^{\R}\to\Ycal^{\R}$. Functoriality is immediate. For the scalar action,
        $(\lambda,\gamma)\mapsto\gamma(\lambda\,\cdot)$ is sample-smooth
        $\R\times\Xcal^{\R}\to\Xcal^{\R}$ by \Cref{sss:thm:cartesian-closed}, it respects
        $\sim$ because $\lp f\circ\gamma(\lambda\,\cdot)\rp'(0) = \lambda\,(f\circ\gamma)'(0)$, and
        $\id_\R\times q$ is again a quotient map: writing $q$ for the quotient map of
        \Cref{tan:def:tangent}, one has $\lp T\Xcal\rp^{\Omega_n} = q\circ\lp\Xcal^{\R}
        \rp^{\Omega_n}$, the constants being images of constant plots, whence
        $\lp\R\times T\Xcal\rp^{\Omega_n} = \lp\id_\R\times q\rp\circ\lp\R\times\Xcal^{\R}
        \rp^{\Omega_n}$ by \Cref{sss:prp:prod-coprod}. So the map descends and is sample-smooth by
        finality.
    \end{proof}
\end{Prp}

\begin{Thm}[Cartesian spaces; the Jacobian]
    \label{tan:thm:euclidean}
    The map $\gamma\mapsto\lp\gamma(0),\gamma'(0)\rp$ induces an isomorphism
    \begin{align}
        T\R^n &\;\cong\; \R^n\times\R^n,
    \end{align}
    under which $T_x\R^n \cong \R^n$ for every $x$, and for $f \in \SSS(\R^m,\R^n) =
    C^\infty(\R^m,\R^n)$ the map $Tf$ becomes
    \begin{align}
        Tf(x,v) &= \lp f(x),\ Df(x)\,v\rp ,
    \end{align}
    with $Df$ the Jacobian. In particular $Df:\,\R^m\to\R^{n\times m}$ is sample-smooth, and all
    higher derivatives are obtained by iterating $T$.
    \begin{proof}
        Write $\beta(\gamma) := \lp\gamma(0),\gamma'(0)\rp$. Since $\SSS(\R^n,\R) = C^\infty(\R^n,\R)$
        by \Cref{sss:thm:euclidean}, the relation \Cref{eq:tangent-equiv} holds precisely when
        $\beta(\gamma) = \beta(\eta)$: ``$\Leftarrow$'' is the chain rule, and ``$\Rightarrow$''
        follows by taking for $f$ the $n$ coordinate projections.

        $\beta$ is sample-smooth: a plot of $\lp\R^n\rp^{\R}$ on $\Omega_m$ is an $F$ for which
        $G(s,\omega,t) := F(s,\omega)(t)$ is smooth in $(s,t)$ jointly and measurable in $\omega$,
        by \Cref{eq:function-space}; then $G(-,-,0)$ and $\partial_t G(-,-,0)$ are smooth in $s$ and
        measurable in $\omega$, the latter as a pointwise limit of difference quotients.

        Finally $\beta$ is a quotient map, because a plot $(a,b)$ of $\R^n\times\R^n$ on $\Omega_m$
        lifts to the plot $u \mapsto \lp t\mapsto a(u) + t\,b(u)\rp$. Hence the induced bijection
        $T\R^n\to\R^n\times\R^n$ is an isomorphism, both sides carrying the final structure for
        $\beta$. The formula for $Tf$ is the chain rule, and sample-smoothness of $Df$
        follows by \Cref{sss:thm:euclidean} from smoothness of $(x,v)\mapsto Df(x)v$.
    \end{proof}
\end{Thm}

\begin{Prp}[Correctness on manifolds]
    \label{tan:prp:manifolds}
    For a finite-dimensional smooth manifold $M$ as in \Cref{sss:eg:manifolds}, $T M$ is the
    classical tangent bundle, $\pi_M$ is the classical projection, and $Tf$ is the classical
    differential for every smooth $f$.
    \begin{proof}
        Write $T^{\mathrm{cl}}M$ for the classical tangent bundle, itself a manifold and so an
        object by \Cref{sss:eg:manifolds}, and put $\beta_M(\gamma) := \gamma'(0) \in
        T^{\mathrm{cl}}M$, whose base point is $\gamma(0)$.

        The fibres of $\beta_M$ are the $\sim$-classes: ``$\Leftarrow$'' is the chain rule, and
        ``$\Rightarrow$'' holds because $\SSS(M,\R) = C^\infty(M,\R)$ by \Cref{sss:thm:manifolds},
        because $(f\circ\gamma)'(0) = df_{\gamma(0)}\lp\gamma'(0)\rp$, and because smooth real
        functions separate classical tangent vectors.

        $\beta_M$ is sample-smooth. Smoothness in the parameter is local and may be read in a chart;
        measurability in the seed may not, since the chart around $G(s,\omega,0)$ depends on
        $\omega$, so we argue globally. Let $G \in M^{\R}{}^{\Omega_m}$ be a plot, i.e.\ a map
        $G:\,\R^m\times\Omega\times\R\to M$ smooth in $(s,t)$ and admissible in $\omega$, and
        fix $s$. Cover $M$ by countably many chart domains $U_i$ with charts $\kappa_i$, which is
        possible because $M$ is second-countable, and put
        \begin{align}
            E_{i,k} &:= \bigcap_{t \in \Q,\ |t|\le 1/k}
            \lC \omega \st G(s,\omega,t) \in U_i \rC ,
        \end{align}
        which lies in $\Bcal_\Omega$ by (M2$_M$) for $G$. The $E_{i,k}$ cover $\Omega$, since each
        curve $G(s,\omega,-)$ is continuous at $0$. On
        $E_{i,k}$ the classical derivative $\partial_tG(s,\omega,0)$ is the pointwise limit as
        $j\to\infty$ of $j\lp\kappa_i G(s,\omega,1/j) - \kappa_i G(s,\omega,0)\rp$, $j\ge k$,
        each term measurable in $\omega$; so $\omega\mapsto\beta_M(G(s,\omega,-))$ is measurable
        on each $E_{i,k}$, hence on $\Omega$.

        $\beta_M$ is a quotient map. Here a \emph{global} lift is needed, since
        \Cref{site:rem:no-locality} leaves no way to subdivide the domain of a plot and reglue.
        Equip $M$ with a complete Riemannian metric, which exists because $M$ is second-countable
        and Hausdorff, so that $\exp:\,T^{\mathrm{cl}}M \to M$ is defined on all of
        $T^{\mathrm{cl}}M$ and smooth. Given a plot $w:\,\Omega_m\to T^{\mathrm{cl}}M$ put
        \begin{align}
            F(u) &:= \lp t \mapsto \exp\lp t \cdot w(u)\rp\rp, \qquad u \in \Omega_m,
        \end{align}
        with $t\cdot w(u)$ the fibrewise scalar multiple. Then $F$ is a plot of $M^{\R}$, being
        smooth in the two smooth arguments jointly and measurable in the seed, and
        $\beta_M\lp F(u)\rp = w(u)$ because
        $\tfrac{d}{dt}\big|_{t=0}\exp(t v) = v$. Hence $T^{\mathrm{cl}}M$ carries the final
        structure for $\beta_M$ and the induced bijection $TM \to T^{\mathrm{cl}}M$ is an
        isomorphism. The statements about $\pi_M$ and $Tf$ follow.
    \end{proof}
\end{Prp}

\begin{Def}[Cotangent space and the pairing]
    \label{tan:def:cotangent}
    For $x \in \Xcal$ let $f \sim_x g$ if $\lp f\circ\gamma\rp'(0) = \lp g\circ\gamma\rp'(0)$ for
    every $\gamma \in \Xcal^{\R}$ with $\gamma(0) = x$, and put
    $T^*_x\Xcal := \SSS(\Xcal,\R)/\!\sim_x$.
\end{Def}

\begin{Prp}[The cotangent space is a vector space; the tangent space is only a cone]
    \label{tan:prp:cotangent}
    $T^*_x\Xcal$ is a real vector space, and
    \begin{align}
        \lA -,-\rA:\; T_x\Xcal \times T^*_x\Xcal &\longrightarrow \R, &
        \lA [\gamma],[f]\rA &:= \lp f\circ\gamma\rp'(0),
    \end{align}
    is well defined, separates points in each variable and is homogeneous in the first. For a
    sample-smooth $f:\,\Xcal\to\Ycal$ the transpose
    $df^*_x:\,T^*_{f(x)}\Ycal\to T^*_x\Xcal$, $df^*_x[g] := [g\circ f]$, is well defined and
    linear, and $\lA df_x(v),[g]\rA = \lA v,df^*_x[g]\rA$ for $df_x$ of
    \Cref{tan:prp:functor}. The pairing is not in general bilinear, because $T_x\Xcal$ need not
    carry a linear structure making it so: for $\Xcal := \lC xy=0\rC \ins \R^2$ the injection
    $d\iota_0:\,T_0\Xcal\inj T_0\R^2 \cong \R^2$ has as image the union of the two axes, which
    is stable under scalars but is not a linear subspace.
    \begin{proof}
        $f\mapsto \lp \gamma\mapsto (f\circ\gamma)'(0)\rp$ is $\R$-linear into the vector space of
        functions on the curves through $x$, and $\sim_x$ is the congruence of its kernel, so the
        quotient is a vector space and the pairing is well defined and separating in the second
        variable; it is separating in the first by \Cref{eq:tangent-equiv}. Homogeneity is
        \Cref{tan:prp:functor}. For the transpose, if $g\sim_{f(x)}\tilde g$ then for every
        $\gamma\in\Xcal^\R$ with $\gamma(0)=x$ the curve $f\circ\gamma$ lies in $\Ycal^\R$ with
        $(f\circ\gamma)(0)=f(x)$, so
        $(g\circ f\circ\gamma)'(0) = (\tilde g\circ f\circ\gamma)'(0)$,
        i.e.\ $g\circ f\sim_x \tilde g\circ f$; linearity and the adjunction identity are then immediate
        from the definitions.

        For the last claim take $\Xcal := \lC(x,y) \in \R^2 \st xy = 0\rC$ with the subspace
        structure. A tangent vector at the origin is determined by the pair
        $\lp\gamma_1'(0),\gamma_2'(0)\rp$: let $f \in \SSS(\Xcal,\R)$ and put $h(s) := f(s,0)$,
        $k(s) := f(0,s)$, which lie in $C^\infty(\R,\R)$ because the two axis inclusions are
        sample-smooth; since at every $t$ at least one coordinate of $\gamma(t)$ vanishes,
        \begin{align}
            f\lp\gamma(t)\rp &= h\lp\gamma_1(t)\rp + k\lp\gamma_2(t)\rp - f(0,0),
        \end{align}
        so $\lp f\circ\gamma\rp'(0) = h'(0)\,\gamma_1'(0) + k'(0)\,\gamma_2'(0)$ depends on $\gamma$
        only through that pair. Both $(1,0)$ and $(0,1)$ occur, from the two axis curves. The pair
        $(1,1)$ does not: if $\gamma_1'(0)\ne0\ne\gamma_2'(0)$ then each $\gamma_i$ is non-zero on a
        punctured neighbourhood of $0$, contradicting $\gamma_1\gamma_2 \equiv 0$. So the image of
        $T_0\Xcal$ in $\R^2$ is the union of the two axes, as claimed.
    \end{proof}
\end{Prp}

\begin{Rem}[Where $T$ degenerates]
    \label{tan:rem:flat}
    A rigid object has no non-constant curves, so $T\flat\Scal \cong \flat\Scal$ and every tangent
    space is a point: on the image of $\flat$ the tangent functor is the identity. This is the
    expected reading --- data typed by $\flat$ is data one may not differentiate --- and it is also
    a warning, since it says that $T$ carries no information there. What $T\PrPf(\R^n)$ looks like
    we do not know; see \Cref{disc:sec:open}.
\end{Rem}

\section{The Modal Functors and the Adjoint String}
\label{sec:modal}

The measurable half sits inside $\SSS$ in two extreme ways, by giving a quasi-universal space as few
plots as possible or as many as possible. The two are adjoint to the forgetful functor on either
side, and the resulting string is one longer at each end.

\subsection{The Forgetful Functor and the Two Modalities}
\label{modal:sec:triple}

\begin{Lem}[$\Und$, $\flat$ and $\sharp$ are functors]
    \label{modal:lem:functors}
    $\Und:\,\SSS\to\QUS$, $\flat:\,\QUS\to\SSS$ and $\sharp:\,\QUS\to\SSS$ of
    \Cref{sss:not:derived,sss:eg:flat-sharp} are functors, and
    $\Und\flat = \Und\sharp = \id_{\QUS}$.
    \begin{proof}
        Functoriality is immediate in each case. $\lp\flat\Scal\rp^\Omega = \Scal^\Omega$ because
        $\pi_0 = \id_{\Omega_0}$; and $\lp\sharp\Scal\rp^\Omega = \Scal^\Omega$ because $\id_{\Omega_0}$ is one
        of the $h$ quantified over, giving ``$\ins$'', while ``$\sni$'' is closure of $\Scal^\Omega$
        under $\Omega^\Omega$.
    \end{proof}
\end{Lem}

\begin{Thm}[The adjoint triple]
    \label{modal:thm:triple}
    $\flat \dashv \Und \dashv \sharp$, and $\flat$ and $\sharp$ are full and faithful. Consequently
    $\Und$ preserves all limits and colimits, $\flat$ preserves colimits and $\sharp$ preserves
    limits, and $\QUS$ is a full subcategory of $\SSS$ in two ways.
    \begin{proof}
        \emph{$\flat\dashv\Und$.} Let $f:\,\Scal\to\Xcal$ be a map. If $f \in \QUS(\Scal,\Und\Xcal)$
        then for $X \in \Scal^\Omega$ we get $f\circ X \in \Xcal^\Omega$ and hence
        $f\circ (X\circ\pi_n) = (f\circ X)\circ\pi_n \in \Xcal^{\Omega_n}$ by (P2); as the plots of
        $\flat\Scal$ are exactly the $X\circ\pi_n$, $f$ is sample-smooth. Conversely, taking $n=0$
        and $\pi_0 = \id$ recovers $f\circ X \in \Xcal^\Omega$.

        \emph{$\Und\dashv\sharp$.} Let $f:\,\Xcal\to\Scal$ be a map. It is sample-smooth into
        $\sharp\Scal$ iff $f\circ p \circ h \in \Scal^\Omega$ for every $n$, every
        $p \in \Xcal^{\Omega_n}$ and every $h \in \Site(\Omega_0,\Omega_n)$. Since $p\circ h \in \Xcal^\Omega$
        this is implied by $f \in \QUS(\Und\Xcal,\Scal)$, and taking $n=0$, $h=\id$ gives the
        converse.

        Full faithfulness of $\flat$ and $\sharp$ follows from $\Und\flat = \Und\sharp = \id$,
        \Cref{modal:lem:functors}, together with the two adjunctions.
    \end{proof}
\end{Thm}

\begin{Rem}[Reading the modalities]
    \label{modal:rem:reading}
    $\flat\Scal$ is $\Scal$ with \emph{no} differentiable structure: its only $n$-parameter families
    are the constant ones, so a sample-smooth map out of a connected smooth object into $\flat\Scal$
    is constant. $\sharp\Scal$ is $\Scal$ with \emph{no} smoothness requirement: a map into
    $\sharp\Scal$ is sample-smooth as soon as it is quasi-measurable. In the intended reading of
    $\SSS$ as a semantics, $\flat$ marks data that must not be differentiated and $\sharp$ marks data
    whose differentiable structure is not being tracked.
\end{Rem}

\subsection{The Outermost Left Adjoint: the Shape}
\label{modal:sec:shape}

\begin{Def}[Plot-connectedness and the shape]
    \label{modal:def:shape}
    For $\Xcal \in \SSS$ let $\sim$ be the equivalence relation on $\Xcal$ generated by
    \begin{align}
        p(t,\omega) &\sim p(t',\omega), \qquad p \in \Xcal^{\Omega_1},\ t,t' \in \R,\ \omega \in \Omega .
    \end{align}
    Let $\Pi\Xcal := \Xcal/\!\sim$ with the final quasi-universal structure along the projection
    $q_\Xcal:\,\Und\Xcal \to \Pi\Xcal$, i.e.\
    $\lp\Pi\Xcal\rp^\Omega = q_\Xcal \circ \Xcal^\Omega$ when $\Xcal \ne \zero$.
\end{Def}

\begin{Lem}[$\sim$ sees all arities]
    \label{modal:lem:shape-arities}
    For $p \in \Xcal^{\Omega_n}$, $t,t' \in \R^n$ and $\omega \in \Omega$ one has
    $p(t,\omega)\sim p(t',\omega)$.
    \begin{proof}
        Pick $\varphi \in C^\infty(\R,\R^n)$ with $\varphi(0)=t$, $\varphi(1)=t'$, for instance an
        affine path. Then $p \circ \varphi^{(n)} \in \Xcal^{\Omega_1}$ by \Cref{site:not:morphisms}, and it
        joins the two points at the same $\omega$.
    \end{proof}
\end{Lem}

\begin{Thm}[$\Pi \dashv \flat$]
    \label{modal:thm:shape-adjoint}
    $\Pi:\,\SSS\to\QUS$ is a functor and $\Pi \dashv \flat$; moreover $\Pi\flat \cong \id_\QUS$.
    \begin{proof}
        Let $f:\,\Xcal\to\Scal$ be a map. Then $f$ is sample-smooth as a map $\Xcal\to\flat\Scal$ iff
        for every $n$ and $p \in \Xcal^{\Omega_n}$ the composite $f\circ p$ is of the form $X\circ\pi_n$
        with $X \in \Scal^\Omega$; that is, iff (a) $f\circ p$ does not depend on the $\R^n$
        coordinate, and (b) $f\circ p\circ\iota^n_{t_0} \in \Scal^\Omega$ for some, hence by (a) for
        every, $t_0$.

        By \Cref{modal:lem:shape-arities}, (a) for all $n,p$ holds iff $f$ is constant on
        $\sim$-classes, i.e.\ iff $f = \bar f \circ q_\Xcal$ for a unique $\bar f$. Given (a),
        condition (b) ranges over exactly the elements $p\circ\iota^n_{t_0}$ of $\Xcal^\Omega$, and
        every element of $\Xcal^\Omega$ occurs, so (b) says $f \in \QUS(\Und\Xcal,\Scal)$, which by
        finality of $\Pi\Xcal$ says $\bar f \in \QUS(\Pi\Xcal,\Scal)$. The correspondence
        $f \leftrightarrow \bar f$ is the required natural bijection. Finally $\flat\Scal$ has only
        $t$-independent plots, so $\sim$ is trivial on it and
        $\Pi\flat\Scal = \Scal$ with $q = \id$.
    \end{proof}
\end{Thm}

\begin{Prp}[$\Pi$ preserves finite products]
    \label{modal:prp:shape-products}
    $\Pi\one = \one$ and the canonical map
    $\Pi(\Xcal\times\Ycal) \to \Pi\Xcal\times\Pi\Ycal$ is an isomorphism of quasi-universal spaces.
    \begin{proof}
        The map is surjective because $(x,y)$ maps to $([x],[y])$. For injectivity, suppose
        $x \sim x'$ in $\Xcal$ and $y \sim y'$ in $\Ycal$. It suffices to treat a single generating
        step in each factor. If $p \in \Xcal^{\Omega_1}$ joins $x$ to $x'$ at $\omega$, then
        $\lp p,\mathrm{const}_y\rp \in (\Xcal\times\Ycal)^{\Omega_1}$ by \Cref{sss:prp:prod-coprod} joins
        $(x,y)$ to $(x',y)$ at $\omega$; symmetrically for the second factor. Concatenating the two
        chains gives $(x,y)\sim(x',y')$.

        For the structures: $\lp\Pi\Xcal\times\Pi\Ycal\rp^\Omega =
        \lC (q_\Xcal\circ X,\,q_\Ycal\circ Y) \st X \in \Xcal^\Omega,\ Y \in \Ycal^\Omega\rC$ by
        \Cref{qus:thm:quasitopos,modal:def:shape}, while
        $\Pi(\Xcal\times\Ycal)^\Omega = \lC q\circ(X,Y)\rC$, and $(X,Y)$ ranges over all pairs by
        \Cref{sss:prp:prod-coprod}. The bijection matches the two families.
    \end{proof}
\end{Prp}

\subsection{The Outermost Right Adjoint}
\label{modal:sec:lambda}

\begin{Lem}[Factorisation of loose plots]
    \label{modal:lem:factorisation}
    Let $\Scal \in \QUS$, $n\ge0$ and $p \in \lp\sharp\Scal\rp^{\Omega_n}$. Then there are $X \in \Scal^\Omega$
    and $q \in \lp\sharp\Omega\rp^{\Omega_n}$ with $p = X \circ q$.
    \begin{proof}
        By ($\Omega$5) there is a Borel injection $\iota:\,\R^n\to\Omega$ with Borel image $B$ and
        Borel inverse; let $\rho:\,\Omega\to\R^n$ be Borel with $\rho\circ\iota = \id_{\R^n}$, for
        instance $\rho := \iota^{-1}$ on $B$ and $0$ off $B$. Put
        \begin{align}
            q(t,\omega) &:= \Theta^{-1}\lp\iota(t),\,\omega\rp, &
            X(\sigma) &:= p\lp \rho(\Theta_1\sigma),\ \Theta_2\sigma \rp .
        \end{align}
        Then $X(q(t,\omega)) = p(\rho\iota(t),\omega) = p(t,\omega)$. For
        $q \in \lp\sharp\Omega\rp^{\Omega_n}$ take $h = \lb W,\Phi\rb \in \Site(\Omega_0,\Omega_n)$; then
        $q\circ h = \Theta^{-1}\lp\iota\circ W,\Phi\rp \in \Omega^\Omega$ by
        \Cref{qus:lem:borel-suffices}. And $X = p \circ \lb\rho\circ\Theta_1,
        \Theta_2\rb$ with $\lb\rho\circ\Theta_1,\Theta_2\rb \in \Site(\Omega_0,\Omega_n)$, so
        $X \in \Scal^\Omega$ because $p \in \lp\sharp\Scal\rp^{\Omega_n}$.
    \end{proof}
\end{Lem}

\begin{Thm}[$\sharp\dashv\Lambda$]
    \label{modal:thm:lambda}
    Let $\Lambda\Xcal := \lp \Xcal,\ \SSS(\sharp\Omega,\Xcal)\rp$, where a sample-smooth map
    $\sharp\Omega\to\Xcal$ is read as a map $\Omega\to\Xcal$. Then $\Lambda$ is a functor
    $\SSS\to\QUS$ and $\sharp\dashv\Lambda$.
    \begin{proof}
        $\Lambda\Xcal$ is a quasi-universal space: constants are sample-smooth, and for
        $\Phi \in \Omega^\Omega$ the map $\sharp\Phi:\,\sharp\Omega\to\sharp\Omega$ is sample-smooth
        by \Cref{modal:lem:functors}, so $X\circ\Phi = X\circ\sharp\Phi$ is again sample-smooth.
        Functoriality is postcomposition.

        Let $f:\,\Scal\to\Xcal$ be a map. If $f$ is sample-smooth $\sharp\Scal\to\Xcal$ then for
        every $X \in \Scal^\Omega$ the composite $f\circ X = f\circ\sharp X$ is sample-smooth
        $\sharp\Omega\to\Xcal$, i.e.\ lies in $(\Lambda\Xcal)^\Omega$; so
        $f \in \QUS(\Scal,\Lambda\Xcal)$. Conversely assume the latter and let
        $p \in \lp\sharp\Scal\rp^{\Omega_n}$. Write $p = X\circ q$ as in \Cref{modal:lem:factorisation}. Then
        $f\circ X \in \SSS(\sharp\Omega,\Xcal)$ by hypothesis and $q \in \lp\sharp\Omega\rp^{\Omega_n}$, so
        $f\circ p = (f\circ X)\circ q \in \Xcal^{\Omega_n}$. Naturality is clear.
    \end{proof}
\end{Thm}

\begin{Thm}[The adjoint string and cohesion]
    \label{modal:thm:string}
    There is an adjoint string
    \begin{align}
        \Pi \;\dashv\; \flat \;\dashv\; \Und \;\dashv\; \sharp \;\dashv\; \Lambda,
        \qquad \Pi,\Und,\Lambda:\,\SSS\to\QUS, \quad \flat,\sharp:\,\QUS\to\SSS,
    \end{align}
    in which $\flat$ and $\sharp$ are full and faithful and $\Pi$ preserves finite products. In
    particular $\SSS$ is \emph{cohesive} over $\QUS$ in the sense of \cite{Law07}, with the
    additional right adjoint $\Lambda$.
    \begin{proof}
        \Cref{modal:thm:shape-adjoint}, \Cref{modal:thm:triple}, \Cref{modal:thm:lambda,modal:prp:shape-products}.
    \end{proof}
\end{Thm}

\begin{Rem}[What the string is for]
    \label{modal:rem:string}
    Two uses. First, it identifies the two ways in which the measurable theory of \cite{For21} sits
    inside the present one, and \Cref{prob:thm:functor-unit}(3) shows that the monad commutes
    with the rigid embedding, so that \cite{For21} is recovered on the nose. Second, in a semantics
    the modalities are how non-differentiable data is typed: a value in $\flat\Scal$ cannot be
    differentiated because it cannot vary smoothly, and by \Cref{modal:thm:triple} that is a
    property of the object rather than a side condition on morphisms. We do not pursue the
    programming consequences here.
\end{Rem}

\section{The Probability Monad}
\label{sec:prob}

This is what the site was chosen for, and it contains the two results the paper is written
around: $\PrPf$ is a monad on all of $\SSS$, \Cref{prob:thm:monad}, and its Kleisli category is a
Markov category on all of $\SSS$, \Cref{prob:thm:markov}. We define the plots of $\PrPf(\Xcal)$ at
\emph{every} test object to be the push-forwards of $\Xcal$-plots, and every piece of monad
structure then falls out of the seed splitting of \Cref{qus:lem:split-U}. Nothing in this section
carries a hypothesis on the objects: no closure property, no regularity, no selection theorem.
\Cref{prob:rem:comparison} explains where the difficulty went.

\subsection{The Object of Measures}
\label{prob:sec:object}

\begin{Not}[The law of a plot]
    \label{prob:not:law}
    For $\Xcal \in \SSS$, $n \ge 0$ and $Y \in \Xcal^{\Omega_n}$ define
    \begin{align}
        \Law Y:\; \Omega_n &\longrightarrow \Pcal\lp\Xcal,\Bcal_\Xcal\rp, &
        \Law Y(t,\omega) &:= \lp \omega' \mapsto Y\lp t,\Theta^{-1}(\omega,\omega')\rp \rp_* U .
        \label{eq:law}
    \end{align}
    The seed of $Y$ is split by $\Theta$ into a \emph{parameter} half, retained as the second
    argument of $\Law Y$, and a \emph{sampling} half, integrated out against $U$.
\end{Not}

\begin{Def}[The space of measures]
    \label{prob:def:PrPf}
    For $\Xcal \in \SSS$ put
    \begin{align}
        \PrPf(\Xcal) &:= \lC X_*P \st X \in \Xcal^\Omega,\ P \in \PrPfO \rC, &
        \PrPf(\Xcal)^{\Omega_n} &:= \lC \Law Y \st Y \in \Xcal^{\Omega_n} \rC .
        \label{eq:PrPf-plots}
    \end{align}
\end{Def}

\begin{Lem}[$\PrPf(\Xcal)$ is a sample-smooth space]
    \label{prob:lem:object}
    \Cref{eq:PrPf-plots} is well defined --- every value of $\Law Y$ lies in $\PrPf(\Xcal)$ ---
    and satisfies (P1) and (P2).
    \begin{proof}
        \emph{Values.} Fix $(t,\omega)$ and put
        $\iota := \lb \mathrm{const}_t,\ \Theta^{-1}(\omega,-)\rb$, a morphism
        $\Omega_0\to\Omega_n$ by \Cref{site:not:morphisms}, its seed part lying in
        $\Omega^\Omega$ by ($\Omega$2) with \Cref{qus:lem:borel-suffices}(1). Then $Y\circ\iota$, which sends $\omega'$ to
        $Y\lp t,\Theta^{-1}(\omega,\omega')\rp$, lies in $\Xcal^\Omega$, and its law under $U$
        lies in $\PrPf(\Xcal)$.

        \emph{(P1).} Let $\mu = X_*P \in \PrPf(\Xcal)$. By ($\Omega$5) choose
        $\psi \in \Omega^\Omega$ with $\psi_*U = P$ and put
        $Y := X\circ\psi\circ\Theta_2\circ\pi_n \in \Xcal^{\Omega_n}$. Then
        $\Law Y(t,\omega) = \lp\omega'\mapsto X(\psi(\omega'))\rp_*U = X_*P = \mu$ for all
        $(t,\omega)$.

        \emph{(P2).} Let $h = (g,\Phi) \in \Site(\Omega_m,\Omega_n)$ and put
        \begin{align}
            Y'(s,\sigma) &:= Y\Big( g\lp s,\Theta_1\sigma\rp,\ \Theta^{-1}\big(\Phi(\Theta_1\sigma),
            \ \Theta_2\sigma\big)\Big).
            \label{eq:PrPf-substitution}
        \end{align}
        The map $(s,\sigma)\mapsto \lp g(s,\Theta_1\sigma),\,\Theta^{-1}(\Phi(\Theta_1\sigma),
        \Theta_2\sigma)\rp$ is a $\Site$-morphism $\Omega_m\to\Omega_n$: its smooth part satisfies
        (M1) because $g$ does and (M2) by ($\Omega$2), and its seed part is Borel and independent of
        $s$. Hence $Y' \in \Xcal^{\Omega_m}$. Substituting $\sigma := \Theta^{-1}(\omega,\omega')$, so that
        $\Theta_1\sigma = \omega$ and $\Theta_2\sigma = \omega'$, gives
        $Y'\lp s,\Theta^{-1}(\omega,\omega')\rp = Y\lp g(s,\omega),\Theta^{-1}(\Phi(\omega),
        \omega')\rp$ and therefore $\Law Y' = \lp\Law Y\rp\circ h$.
    \end{proof}
\end{Lem}

\begin{Prp}[The measurable half is unchanged]
    \label{prob:prp:underlying}
    $\Und\PrPf(\Xcal) = \PrPf(\Und\Xcal)$, with $\PrPf$ on the right the push-forward monad of
    \Cref{qus:def:PrPf}. In particular $\Bcal_{\PrPf(\Xcal)}$ is the $\sigma$-algebra of
    \cite{For21} and no measure-theoretic notion changes.
    \begin{proof}
        The underlying sets agree by \Cref{eq:PrPf-plots}. For the random variables, let
        $Y \in \Xcal^\Omega$ and put $X(\omega) := Y\lp\Theta^{-1}(\omega,-)\rp$. Then
        $X \in \lp\Xcal^\Omega\rp^\Omega$: for $\Phi,Z \in \Omega^\Omega$ the map
        $\omega\mapsto X(\Phi(\omega))(Z(\omega)) = Y\lp\Theta^{-1}(\Phi(\omega),Z(\omega))\rp$ is in
        $\Xcal^\Omega$ by \Cref{qus:lem:borel-suffices}(2), which is the function-space criterion of
        \Cref{qus:thm:quasitopos}; and $\Law Y(\omega) = X(\omega)_*U$. Conversely let
        $X \in \lp\Xcal^\Omega\rp^\Omega$ and $P \in \PrPfO$, choose $\psi$ with $\psi_*U = P$ by
        ($\Omega$5) and put $Y(\sigma) := X(\Theta_1\sigma)\lp\psi(\Theta_2\sigma)\rp$, which lies in
        $\Xcal^\Omega$ by the same criterion. Then
        $\Law Y(\omega) = \lp\omega'\mapsto X(\omega)(\psi(\omega'))\rp_*U = X(\omega)_*P$.
    \end{proof}
\end{Prp}

\begin{Prp}[The families of measures are the reparametrisable ones]
    \label{prob:prp:curves}
    For every $n \ge 0$,
    \begin{align}
        \PrPf(\Xcal)^{\R^n} &= \lC \lp t\mapsto \lp Y(t,-)\rp_*U\rp \st Y \in \Xcal^{\Omega_n} \rC .
        \label{eq:reparametrisable}
    \end{align}
    \begin{proof}
        ``$\ins$'': if $\gamma\circ\pr_{\R^n} = \Law Y$ then, fixing any $\omega_0$,
        $\gamma(t) = \Law Y(t,\omega_0) = \lp \tilde Y(t,-)\rp_*U$ where
        $\tilde Y := Y\circ\lp\pr_{\R^n},\Theta^{-1}(\omega_0,-)\rp \in \Xcal^{\Omega_n}$.
        ``$\sni$'': given $\tilde Y \in \Xcal^{\Omega_n}$ put $Y := \tilde Y \circ \Theta_2^{(n)}$, i.e.\
        $Y(t,\sigma) := \tilde Y(t,\Theta_2\sigma)$; then
        $\Law Y(t,\omega) = \lp\tilde Y(t,-)\rp_*U$ for every $\omega$, so $\Law Y$ is
        $\omega$-independent and its associated $\gamma$ is the given family.
    \end{proof}
\end{Prp}

\Cref{prob:prp:underlying,prob:prp:curves} say that $\PrPf$ is what one would expect on
everything a presentation by curves and random variables could see: the same measures, the same
random variables, and the same smooth families --- the right-hand side of
\Cref{eq:reparametrisable} is what one would call a \emph{reparametrisable} family, the law of a
smoothly varying random variable. All the new content sits at $\Omega_n$ for $n \ge 1$, and
\Cref{prob:rem:comparison} says what it is.

\subsection{Functor and Unit}
\label{prob:sec:functor}

\begin{Thm}[Functor, unit, and compatibility with $\flat$]
    \label{prob:thm:functor-unit}
    \begin{enumerate}
        \item $\PrPf$ is an endofunctor of $\SSS$: for sample-smooth $f:\,\Xcal\to\Ycal$ the
            push-forward $\PrPf(f)(\mu) := f_*\mu$ is sample-smooth, and
            $\PrPf(f)\circ\Law Y = \Law\lp f\circ Y\rp$.
        \item $\delta_\Xcal:\,\Xcal\to\PrPf(\Xcal)$, $x\mapsto\delta_x$, is sample-smooth and natural.
        \item $\PrPf\circ\flat = \flat\circ\PrPf$.
    \end{enumerate}
    \begin{proof}
        (1) $f_*\lp X_*P\rp = (f\circ X)_*P \in \PrPf(\Ycal)$, so $\PrPf(f)$ is well defined; the
        displayed identity is immediate from \Cref{eq:law}, and $f\circ Y \in \Ycal^{\Omega_n}$.

        (2) For $p \in \Xcal^{\Omega_n}$ put $Y := p\circ\Theta_1^{(n)}$, i.e.\
        $Y(t,\sigma) := p(t,\Theta_1\sigma)$, which lies in $\Xcal^{\Omega_n}$ by
        \Cref{site:not:morphisms}. Then
        $\Law Y(t,\omega) = \lp\omega'\mapsto p(t,\omega)\rp_*U = \delta_{p(t,\omega)}$, so
        $\delta_\Xcal\circ p = \Law Y \in \PrPf(\Xcal)^{\Omega_n}$. Naturality is
        $f_*\delta_x = \delta_{f(x)}$.

        (3) The plots of $\flat\Scal$ are the $t$-independent ones with seed part in $\Scal^\Omega$,
        and $\Law$ of such a plot is again $t$-independent with seed part in $\PrPf(\Scal)^\Omega$
        by \Cref{prob:prp:underlying}; conversely every plot of $\flat\PrPf(\Scal)$ arises this way.
    \end{proof}
\end{Thm}

\begin{Lem}[Push-forward along a varying map]
    \label{prob:lem:pf-enriched}
    For all $\Xcal,\Ycal \in \SSS$ the evaluation-and-push-forward map
    \begin{align}
        \PrPf(\Xcal) \times \Ycal^\Xcal &\longrightarrow \PrPf(\Ycal), &
        (\mu,f) &\longmapsto f_*\mu,
    \end{align}
    is sample-smooth. Equivalently, push-forward is enriched over $\SSS$.
    \begin{proof}
        Let $\lp\Law Y,F\rp$ be a plot of the source on $\Omega_n$, so $Y \in \Xcal^{\Omega_n}$ and
        $F \in \lp\Ycal^\Xcal\rp^{\Omega_n}$. Put $Z(t,\sigma) := F\lp t,\Theta_1\sigma\rp\lp Y(t,\sigma)\rp$.
        Applying \Cref{eq:function-space} with $m := n$, $h := \Theta_1^{(n)}$ and $q := Y$ shows
        $Z \in \Ycal^{\Omega_n}$. Substituting $\sigma := \Theta^{-1}(\omega,\omega')$ gives
        $Z\lp t,\Theta^{-1}(\omega,\omega')\rp = F(t,\omega)\lp Y(t,\Theta^{-1}(\omega,\omega'))\rp$,
        whence $\Law Z(t,\omega) = F(t,\omega)_* \Law Y(t,\omega)$, which is the composite in
        question.
    \end{proof}
\end{Lem}

\subsection{Products of Kernels}
\label{prob:sec:otimes}

\begin{Thm}[The product of measures]
    \label{prob:thm:otimes}
    The product of measures of \Cref{qus:thm:PrPf},
    \begin{align}
        \otimes:\; \PrPf(\Xcal)\times\PrPf(\Ycal) &\longrightarrow \PrPf(\Xcal\times\Ycal), &
        (\mu,\nu) &\longmapsto \mu\otimes\nu,
    \end{align}
    defined on all of $\Bcal_{\Xcal\times\Ycal}$ by \Cref{eq:otimes-section}, is sample-smooth and
    natural in $\Xcal$ and $\Ycal$. No hypothesis is placed on either space.
    \begin{proof}
        Let $\lp\Law Y,\Law Z\rp$ be a plot of the source on $\Omega_n$, with $Y \in \Xcal^{\Omega_n}$ and
        $Z \in \Ycal^{\Omega_n}$. Put
        \begin{align}
            W(t,\sigma) &:= \Big( Y\big(t,\Theta^{-1}(\Theta_1\sigma,\ \Theta_1\Theta_2\sigma)\big),
            \ Z\big(t,\Theta^{-1}(\Theta_1\sigma,\ \Theta_2\Theta_2\sigma)\big)\Big).
        \end{align}
        Each component is $Y$ resp.\ $Z$ precomposed with a $\Site$-morphism whose smooth part is
        the projection and whose seed part is Borel, so $W \in (\Xcal\times\Ycal)^{\Omega_n}$ by
        \Cref{sss:prp:prod-coprod}. Substituting $\sigma := \Theta^{-1}(\omega,\omega')$,
        \begin{align}
            W\lp t,\Theta^{-1}(\omega,\omega')\rp &= \Big(Y\big(t,\Theta^{-1}(\omega,\Theta_1
            \omega')\big),\ Z\big(t,\Theta^{-1}(\omega,\Theta_2\omega')\big)\Big),
        \end{align}
        and under $U$ the pair $(\Theta_1\omega',\Theta_2\omega')$ has law $U\otimes U$ by
        \Cref{qus:lem:split-U}. Writing $A(\omega_1) := Y\lp t,\Theta^{-1}(\omega,\omega_1)\rp$
        and $B(\omega_2) := Z\lp t,\Theta^{-1}(\omega,\omega_2)\rp$, the displayed map is
        $(A,B)\circ\Theta$, which lies in $(\Xcal\times\Ycal)^\Omega$; so for
        $C \in \Bcal_{\Xcal\times\Ycal}$ the set $\lp(A,B)\circ\Theta\rp^{-1}(C)$ lies in
        $\Bcal_\Omega$, and $(A,B)^{-1}(C) \in \Bcal_{\Omega\times\Omega}$ because $\Theta$ is
        an isomorphism of quasi-universal spaces, \Cref{qus:lem:borel-suffices}(2). Only then may
        ($\Omega$3) be applied. So the law of the displayed map is the product of the two marginal
        laws, computed by the iterated integral \Cref{eq:otimes-section}; that is,
        $\Law W(t,\omega) = \Law Y(t,\omega)\otimes\Law Z(t,\omega)$. Naturality is naturality of
        \Cref{eq:otimes-section} in $\Xcal,\Ycal$.
    \end{proof}
\end{Thm}

\begin{Cor}[Independent product of kernels]
    \label{prob:cor:kernel-indep}
    For all $\Xcal,\Ycal,\Zcal \in \SSS$ the map
    \begin{align}
        \otimes:\; \PrPf(\Xcal)^\Zcal \times \PrPf(\Ycal)^\Zcal &\longrightarrow
        \PrPf(\Xcal\times\Ycal)^\Zcal, &
        P(X|Z)\otimes Q(Y|Z) &= \lp P\otimes Q\rp(X,Y|Z),
        \label{eq:kernel-indep}
    \end{align}
    given on points by $(P\otimes Q)(z) := P(z)\otimes Q(z)$, is sample-smooth.
    \begin{proof}
        By \Cref{sss:thm:cartesian-closed} it suffices that the uncurried map
        $\PrPf(\Xcal)^\Zcal\times\PrPf(\Ycal)^\Zcal\times\Zcal\to\PrPf(\Xcal\times\Ycal)$,
        $(P,Q,z)\mapsto P(z)\otimes Q(z)$, be sample-smooth; it is
        $\otimes\circ(\ev\times\ev)\circ\alpha$ for the evident sample-smooth rearrangement
        $\alpha$, and $\ev$ is sample-smooth by \Cref{sss:thm:cartesian-closed}.
    \end{proof}
\end{Cor}

\begin{Cor}[Costrength and strength]
    \label{prob:cor:strength}
    The costrength and the strength
    \begin{align}
        \rho_{\Xcal,\Ycal}:\; \PrPf(\Xcal)\times\Ycal &\to \PrPf(\Xcal\times\Ycal), &
        (\mu,y) &\mapsto \mu\otimes\delta_y, \\
        \tau_{\Xcal,\Ycal}:\; \Xcal\times\PrPf(\Ycal) &\to \PrPf(\Xcal\times\Ycal), &
        (x,\nu) &\mapsto \delta_x\otimes\nu,
    \end{align}
    are sample-smooth and natural, and are related by the braiding.
    \begin{proof}
        $\rho = \otimes\circ\lp\id\times\delta_\Ycal\rp$ by \Cref{prob:thm:otimes,prob:thm:functor-unit}(2); compose with the braiding for $\tau$.
    \end{proof}
\end{Cor}

\subsection{The Multiplication}
\label{prob:sec:mult}

\begin{Thm}[The multiplication is sample-smooth]
    \label{prob:thm:mult}
    For every $\Xcal \in \SSS$ the mixture map
    \begin{align}
        \Mbb_\Xcal:\; \PrPf\lp\PrPf(\Xcal)\rp &\longrightarrow \PrPf(\Xcal), &
        \Mbb_\Xcal(\Gamma) &:= \int \nu\,\Gamma(d\nu),
    \end{align}
    is sample-smooth and natural in $\Xcal$. No hypothesis is placed on $\Xcal$.
    \begin{proof}
        That $\Mbb_\Xcal$ is a well-defined map into $\PrPf(\Xcal)$, and natural as a map of sets,
        is \Cref{qus:thm:PrPf} together with \Cref{prob:prp:underlying}.

        A plot of $\PrPf\PrPf(\Xcal)$ on $\Omega_n$ is $\Law\Ncal$ with
        $\Ncal \in \PrPf(\Xcal)^{\Omega_n}$, and $\Ncal = \Law Y$ with $Y \in \Xcal^{\Omega_n}$, by
        \Cref{eq:PrPf-plots} applied twice. Define the Borel map
        \begin{align}
            \Psi:\;\Omega\times\Omega &\to \Omega, &
            \Psi(\omega,\tilde\omega) &:= \Theta^{-1}\Big(\Theta^{-1}\big(\omega,\Theta_1
            \tilde\omega\big),\ \Theta_2\tilde\omega\Big),
        \end{align}
        and put $Y' := Y \circ \Xi^{(n)}$ with $\Xi(\sigma) := \Psi\lp\Theta_1\sigma,\Theta_2
        \sigma\rp$, so that $\Xi \in \Omega^\Omega$ by \Cref{qus:lem:borel-suffices} and hence
        $Y' \in \Xcal^{\Omega_n}$. We claim
        \begin{align}
            \Mbb_\Xcal \circ \Law\lp\Law Y\rp &= \Law Y' ,
            \label{eq:mult-computation}
        \end{align}
        which proves the theorem.

        Fix $(t,\omega)$ and $A \in \Bcal_\Xcal$, and abbreviate
        $G_{\omega'}(\omega'') := Y\lp t,\Theta^{-1}\big(\Theta^{-1}(\omega,\omega'),\omega''\big)\rp$.
        By the change-of-variables formula for the push-forward defining
        $\Law(\Law Y)(t,\omega)$,
        \begin{align}
            \Mbb_\Xcal\Big(\Law\lp\Law Y\rp(t,\omega)\Big)(A)
            &= \int_\Omega \Law Y\lp t,\Theta^{-1}(\omega,\omega')\rp(A)\ U(d\omega')
            = \int_\Omega U\Big( G_{\omega'}^{-1}(A) \Big) U(d\omega') .
        \end{align}
        The map $G:\,(\omega',\omega'')\mapsto G_{\omega'}(\omega'')$ satisfies
        $G\circ\Theta = \lp\omega''' \mapsto Y(t,\Psi(\omega,\omega'''))\rp$, which lies in
        $\Xcal^\Omega$, being $Y$ precomposed with the $\Site$-morphism
        $\lb\mathrm{const}_t,\Psi(\omega,-)\rb$. Hence $(G\circ\Theta)^{-1}(A) \in \Bcal_\Omega$
        and, $\Theta$ being an isomorphism of quasi-universal spaces by
        \Cref{qus:lem:borel-suffices}(2), $G^{-1}(A) \in \Bcal_{\Omega\times\Omega}$. Therefore ($\Omega$3) applies and
        \begin{align}
            \int_\Omega U\Big(G_{\omega'}^{-1}(A)\Big) U(d\omega')
            &= \lp U\otimes U\rp\lp G^{-1}(A)\rp
            = U\Big(\lp G\circ\Theta\rp^{-1}(A)\Big) \nonumber \\
            &= U\Big(\lC \omega''' \st Y\lp t,\Psi(\omega,\omega''')\rp \in A \rC\Big),
        \end{align}
        the middle equality by \Cref{qus:lem:split-U}. On the other hand, substituting
        $\sigma := \Theta^{-1}(\omega,\omega''')$ in the definition of $Y'$ gives
        $Y'\lp t,\Theta^{-1}(\omega,\omega''')\rp = Y\lp t,\Psi(\omega,\omega''')\rp$, so the last
        display is $\Law Y'(t,\omega)(A)$. This is \Cref{eq:mult-computation}.
    \end{proof}
\end{Thm}

\begin{Thm}[The probability monad]
    \label{prob:thm:monad}
    $\lp\PrPf,\delta,\Mbb\rp$ is a monad on $\SSS$, strong for the cartesian monoidal structure,
    commutative and affine. It is a genuine endofunctor of $\SSS$ and no subcategory, closure
    property or regularity hypothesis is involved.
    \begin{proof}
        Functor, unit and multiplication are sample-smooth by \Cref{prob:thm:functor-unit},
        \Cref{prob:thm:mult}, and all three are natural. The monad laws
        $\Mbb\circ\delta_{\PrPf} = \id = \Mbb\circ\PrPf(\delta)$ and
        $\Mbb\circ\Mbb_{\PrPf} = \Mbb\circ\PrPf(\Mbb)$ are equalities of maps of underlying sets,
        which hold by \Cref{qus:thm:PrPf,prob:prp:underlying}; since a sample-smooth map
        is determined by its underlying map, they hold in $\SSS$.

        Strength is \Cref{prob:cor:strength}. Commutativity: both double strengths
        $\Mbb\circ\PrPf(\tau)\circ\rho$ and $\Mbb\circ\PrPf(\rho)\circ\tau$ equal $\otimes$, by the
        pointwise computation $\int\lp\delta_x\otimes\nu\rp\mu(dx) = \mu\otimes\nu$ and its mirror
        image, which is legitimate now that all three maps are sample-smooth; symmetry of $\otimes$
        is ($\Omega$3). Affineness is $\PrPf(\one)\cong\one$.
    \end{proof}
\end{Thm}

\begin{Thm}[Dependent product of kernels]
    \label{prob:thm:kernel-dep}
    For all $\Xcal,\Ycal,\Zcal \in \SSS$ the product of Markov kernels
    \begin{align}
        \otimes:\; \PrPf(\Xcal)^{\Ycal\times\Zcal} \times \PrPf(\Ycal)^\Zcal &\longrightarrow
        \PrPf(\Xcal\times\Ycal)^\Zcal, &
        K(X|Y,Z) \otimes Q(Y|Z) &= P(X,Y|Z),
        \label{eq:kernel-dep}
    \end{align}
    given on points by
    $\lp K\otimes Q\rp(z) = \Mbb\Big(\lp y\mapsto K(y,z)\otimes\delta_y\rp_* Q(z)\Big)$, is
    sample-smooth. No hypothesis is placed on any of the three spaces.
    \begin{proof}
        We uncurry fully, so that $K$ and $Q$ vary with the rest. The map
        \begin{align}
            \PrPf(\Xcal)^{\Ycal\times\Zcal} \times \Zcal \times \Ycal
            &\longrightarrow \PrPf(\Xcal\times\Ycal), &
            (K,z,y) &\longmapsto \rho_{\Xcal,\Ycal}\lp\ev\lp K,(y,z)\rp,\ y\rp,
        \end{align}
        is sample-smooth, being $\rho_{\Xcal,\Ycal}$ of \Cref{prob:cor:strength} precomposed with
        the evaluation of \Cref{sss:thm:cartesian-closed} and the evident sample-smooth
        rearrangement. Currying it in $\Ycal$
        gives a sample-smooth map into $\PrPf(\Xcal\times\Ycal)^{\Ycal}$; pairing that with
        $\ev:\,\PrPf(\Ycal)^\Zcal\times\Zcal\to\PrPf(\Ycal)$ and applying the enriched
        push-forward of \Cref{prob:lem:pf-enriched} --- which is what is needed, the map along
        which one pushes forward varying with $(K,z)$ --- followed by $\Mbb$, sample-smooth by
        \Cref{prob:thm:mult}, gives a sample-smooth map
        $\PrPf(\Xcal)^{\Ycal\times\Zcal}\times\PrPf(\Ycal)^\Zcal\times\Zcal
        \to\PrPf(\Xcal\times\Ycal)$. Currying in $\Zcal$ is \Cref{sss:thm:cartesian-closed}
        again, and on points the composite is the displayed formula.
    \end{proof}
\end{Thm}

\begin{Rem}[What this costs elsewhere]
    \label{prob:rem:dep-free}
\Cref{eq:kernel-dep} is not an independent statement: it is \emph{equivalent} to sample-smoothness
    of $\Mbb$, as one sees by taking $\Ycal := \PrPf(\Xcal)$, $\Zcal := \one$ and $K := \id$, whose
    $\Xcal$-marginal is $\int\nu\,Q(d\nu) = \Mbb_\Xcal(Q)$. So the dependent product of kernels is
    the multiplication in another guise. Here both are theorems and the equivalence is a curiosity;
    in a setting where the multiplication is problematic it says that the dependent product is
    exactly as problematic.
\end{Rem}

The multiplication integrates \emph{measures} and is sample-smooth, \Cref{prob:thm:mult}.
Integrating a \emph{test function} is not, and this is worth recording, because it is easy to expect
the opposite.

\begin{Thm}[Expectations are not sample-smooth]
    \label{prob:thm:expectation}
    The map
    \begin{align}
        \PrPf(\R) &\longrightarrow \R, & \mu &\longmapsto \int \cos \, d\mu,
        \label{eq:expectation-not-smooth}
    \end{align}
    is \emph{not} sample-smooth, although $\cos$ is sample-smooth, bounded, and has all derivatives
    bounded. Consequently the two-variable evaluation
    $\lp h,\mu\rp \mapsto \int h\,d\mu$ is not sample-smooth on
    $\Ccal_b(\R)\times\PrPf(\R)$, where $\Ccal_b(\Xcal) \ins \R^\Xcal$ is the subobject of
    bounded sample-smooth functions --- on which, unlike on all of $\R^\Xcal$, the integral is
    everywhere defined.
    \begin{proof}
        Let $C(\sigma) := \tan\lp\pi(\sigma_0-\tfrac12)\rp$ for $\sigma_0 \in (0,1)$ and
    $C(\sigma) := 0$ for $\sigma_0 \in \lC0,1\rC$, where $\sigma_0$ is the first coordinate of
    $\sigma \in [0,1]^\N$. Then $C$ is Borel, so $C \in \R^\Omega$, and $C_*U$ is the standard
    Cauchy distribution, the exceptional set being $U$-null. Put $Y(t,\sigma) := t\,C(\Theta_2\sigma)$, which lies in $\R^{\Omega_1}$, being
    smooth in $t$ and Borel in $\sigma$. Then $\Law Y(t,\omega)$ is the law $\mu_t$ of $t\,C$ for
    every $\omega$, so $t\mapsto\mu_t$ is an admissible curve of $\PrPf(\R)$ by
    \Cref{prob:prp:curves}. But
    \begin{align}
        \int \cos \, d\mu_t &= \E\lB\cos(tC)\rB = e^{-\lI t\rI},
    \end{align}
    the characteristic function of the Cauchy distribution, which is not differentiable at $t = 0$.
        The second claim follows by precomposing with the sample-smooth map
        $\mu\mapsto(\cos,\mu)$.
    \end{proof}
\end{Thm}

\begin{Rem}[What that costs, and what it does not]
    \label{prob:rem:expectation}
    So $\PrPf$ is a monad whose objects carry smooth families along which expectations of perfectly
    ordinary functions fail to be smooth. This is not an artefact of \Cref{prob:def:PrPf}: the curve
    in the proof of \Cref{prob:thm:expectation} is the law of a smoothly varying random variable, so
    any notion of smooth family that contains the reparametrisable ones contains it. What it says is
    that a \emph{second}, expectation-based structure is genuinely different from this one, not a
    refinement of it; see \Cref{disc:sec:open}. It also says that the loss function of a variational
    program is in general not interpreted by the structure that interprets the program, and
    \Cref{sec:ml} is where that gap is bridged, by an analytic hypothesis rather than a categorical
    one.
\end{Rem}

\subsection{The Kleisli Category is a Markov Category}
\label{prob:sec:markov}

\begin{Def}[The category of differentiable simulators]
    \label{prob:def:kleisli}
    $\Kleisli(\PrPf)$ has the objects of $\SSS$ and $\Kleisli(\PrPf)(\Xcal,\Ycal) :=
    \SSS\lp\Xcal,\PrPf(\Ycal)\rp$, with composition
    $L \bullet K := \Mbb_\Zcal\circ\PrPf(L)\circ K$ and identities $\delta$. We call it the category
    of \emph{differentiable simulators}: by \Cref{prob:thm:representation} every morphism is, plot
    by plot, a smoothly parametrised sampler, and by \Cref{prob:cor:sampler-bijection} the morphisms
    out of $\R^m$ are exactly the samplers, taken up to equality of law at each parameter value.
\end{Def}

\begin{Thm}[Markov structure]
    \label{prob:thm:markov}
    $\Kleisli(\PrPf)$ is a symmetric monoidal category with $\Xcal\otimes\Ycal :=
    \Xcal\times\Ycal$ and unit $\one$, in which every object carries the commutative comonoid
    \begin{align}
        \lp\mathrm{copy}_\Xcal,\ \mathrm{del}_\Xcal\rp &:= \lp \delta\circ\Delta_\Xcal,\
        \delta\circ !_\Xcal \rp,
    \end{align}
    compatibly with the monoidal structure, and in which $\one$ is terminal --- equivalently,
    $\mathrm{del}$ is natural. That is, $\Kleisli(\PrPf)$ is a Markov category in the sense of
    \cite{Fri20}, for every object of $\SSS$ and with no further hypothesis. It is not cartesian:
    $\mathrm{copy}$ is not natural, since for a non-Dirac $\mu \in \PrPf(\R)$ one has
    $\Delta_*\mu \ne \mu\otimes\mu$. It \emph{is} natural along every morphism in the image of
    the functor $\delta_*:\,\SSS\to\Kleisli(\PrPf)$, so every such morphism is deterministic in
    the sense of \cite{Fri20}; whether the converse holds for a given $\Ycal$ depends on
    $\Bcal_\Ycal$ being countably separating, and we do not pursue it.
    \begin{proof}
        By \Cref{prob:thm:monad} the monad is strong, commutative and affine on the cartesian
        monoidal category $\SSS$; the Kleisli category of such a monad is a Markov category, see
        \cite{Fri20}. Concretely, the monoidal product of $K:\,\Xcal\to\PrPf(\Ycal)$ and
        $L:\,\Zcal\to\PrPf(\Wcal)$ is $K\boxtimes L := \otimes\circ(K\times L)$, which is
        sample-smooth by \Cref{prob:thm:otimes} and functorial by commutativity; the comonoid
        laws are the identities that $\Delta$ and $!$ satisfy in $\SSS$, transported along
        $\delta_*$; and terminality of $\one$, i.e.\ naturality of $\mathrm{del}$, is affineness.
    \end{proof}
\end{Thm}

\begin{Rem}[What the theorem does not claim]
    \label{prob:rem:conditionals}
    \Cref{prob:thm:markov} gives a Markov category and nothing beyond it. The further structure one
    may ask of a Markov category is not addressed here and is left to a sequel. The reason for
    stopping is worth stating, because it is the distinction that runs through the whole paper: the
    present construction removes the \emph{selection} problems, those of choosing a representative
    measurably, and the multiplication was one of them; it does nothing about problems of
    \emph{inverting} a push-forward, and no choice of plots can.
\end{Rem}

\subsection{Every Kleisli Morphism is a Sampler}
\label{prob:sec:representation}

\begin{Thm}[Representation]
    \label{prob:thm:representation}
    Let $\Xcal,\Ycal \in \SSS$ and let $K \in \SSS\lp\Xcal,\PrPf(\Ycal)\rp$ be a Kleisli morphism.
    \begin{enumerate}
        \item For every $n$ and every $p \in \Xcal^{\Omega_n}$ there is $Y \in \Ycal^{\Omega_n}$ with
            \begin{align}
                K\lp p(t,\omega)\rp &= \lp \omega'\mapsto Y\lp t,\Theta^{-1}(\omega,\omega')\rp\rp_*U
                \qquad\text{for all } (t,\omega).
            \end{align}
        \item In particular, for $\Xcal := \R^m$ every sample-smooth $K:\,\R^m\to\PrPf(\Ycal)$ is of
            the form $K(t) = \lp Y(t,-)\rp_*U$ with $Y \in \Ycal^{\Omega_m}$; that is, $K$ is a
            \emph{reparametrisation}, jointly smooth in the parameter and measurable in the seed.
        \item Representations compose: if $K$ is represented by $Y$ on the plot $p$ and
            $L:\,\Ycal\to\PrPf(\Wcal)$ is any Kleisli morphism, then $L\bullet K$ is represented on
            $p$ by some $W \in \Wcal^{\Omega_n}$.
    \end{enumerate}
    \begin{proof}
        (1) $K\circ p \in \PrPf(\Ycal)^{\Omega_n} = \lC\Law Y \st Y \in \Ycal^{\Omega_n}\rC$ by
        \Cref{eq:PrPf-plots}. (2) Apply (1) to $p := \pr_{\R^m} \in \lp\R^m\rp^{\Omega_m}$ and evaluate at a fixed
        $\omega_0$, using \Cref{prob:prp:curves}. (3) $L\bullet K$ is a Kleisli morphism, so (1)
        applies to it.
    \end{proof}
\end{Thm}

\begin{Prp}[Samplers are Kleisli morphisms]
    \label{prob:prp:sampler-to-kernel}
    Let $\Wcal,\Ycal \in \SSS$ and let $F \in \SSS\lp\Wcal\times\flat\Omega,\ \Ycal\rp$ be a
    \emph{sampler}. Then
    \begin{align}
        \widetilde F:\; \Wcal \longrightarrow \PrPf(\Ycal), \qquad
        \widetilde F(w) &:= \lp F(w,-)\rp_*U,
    \end{align}
    is a Kleisli morphism, and $\widetilde F$ depends on $F$ only through the laws
    $\lp F(w,-)\rp_*U$.
    \begin{proof}
        Let $q \in \Wcal^{\Omega_n}$ and put
        \begin{align}
            Y(t,\sigma) &:= F\lp q(t,\Theta_1\sigma),\ \Theta_2\sigma\rp .
        \end{align}
        The pair $(t,\sigma)\mapsto\lp q(t,\Theta_1\sigma),\Theta_2\sigma\rp$ is a plot
        of $\Wcal\times\flat\Omega$ on $\Omega_n$: its first component is $q\circ\Theta_1^{(n)}$,
        a plot by \Cref{site:not:morphisms}, and its second is $\Theta_2\circ\pi_n$, a plot of
        $\flat\Omega$ by \Cref{sss:eg:flat-sharp} since $\Theta_2 \in \Omega^\Omega$. As $F$ is
        sample-smooth, $Y \in \Ycal^{\Omega_n}$. Substituting
        $\sigma := \Theta^{-1}(\omega,\omega')$ gives
        $Y\lp t,\Theta^{-1}(\omega,\omega')\rp = F\lp q(t,\omega),\omega'\rp$, so
        $\Law Y(t,\omega) = \lp F(q(t,\omega),-)\rp_*U = \widetilde F\lp q(t,\omega)\rp$; that is,
        $\widetilde F \circ q = \Law Y$, which lies in $\PrPf(\Ycal)^{\Omega_n}$. The last clause is
        immediate from the formula.
    \end{proof}
\end{Prp}

\begin{Cor}[Kleisli morphisms out of $\R^m$ are exactly samplers modulo law]
    \label{prob:cor:sampler-bijection}
    $\Omega_m \cong \R^m\times\flat\Omega$ in $\SSS$, so that a sampler
    $\R^m\times\flat\Omega\to\Ycal$ is the same thing as an element of $\Ycal^{\Omega_m}$; and the
    law map $Y \mapsto \lp\theta\mapsto\lp Y(\theta,-)\rp_*U\rp$ induces a bijection
    \begin{align}
        \Ycal^{\Omega_m}\big/\!\!\sim \;\;\bij\;\; \Kleisli(\PrPf)\lp\R^m,\Ycal\rp, &&
        Y \sim Y' \;:\Longleftrightarrow\; \lp Y(\theta,-)\rp_*U = \lp Y'(\theta,-)\rp_*U
        \ \ \forall\theta .
    \end{align}
    \begin{proof}
        For the isomorphism, both objects have underlying set $\R^m\times\Omega$ and, by
        \Cref{sss:prp:prod-coprod}, \Cref{sss:eg:euclidean,sss:eg:flat-sharp},
        \begin{align}
            \lp\R^m\times\flat\Omega\rp^{\Omega_n} &= \Cmix(\Omega_n,\R^m) \times
            \lC \Phi\circ\pi_n \st \Phi \in \Omega^\Omega\rC = \Site(\Omega_n,\Omega_m)
            = \Omega_m^{\Omega_n},
        \end{align}
        the last equality by \Cref{sss:lem:plots-are-homs}. The law map is well defined by
        \Cref{prob:prp:sampler-to-kernel}, surjective by \Cref{prob:thm:representation}(2), and its
        fibres are the $\sim$-classes by construction.
    \end{proof}
\end{Cor}

This is the precise form of the slogan attached to \Cref{prob:def:kleisli}: a differentiable
simulator \emph{is} a smoothly parametrised sampler, and two samplers name the same simulator
exactly when they agree in law.

\begin{Rem}[The reparametrisation trick as a definition]
    \label{prob:rem:reparam}
    \Cref{prob:thm:representation} is the statement that underlies the reparametrisation gradient of
    the machine-learning literature: a smooth family of distributions is the law of a smoothly
    varying random variable, so that --- under the integrability hypothesis (D) of
    \Cref{ml:thm:pathwise}, which \Cref{prob:thm:expectation} shows cannot be dropped ---
    $\nabla_\theta \E_{\mu_\theta}[h]$ may be computed by differentiating inside the
    expectation. In a setting where a smooth family of measures is
    defined by a condition rather than by a representation, this is a genuine theorem to be proved,
    and not an easy one. Here it holds by construction and is closed under Kleisli composition, part
    (3) --- which is the property a semantics actually needs, since a program is a composite. The
    honest accounting is that nothing has been proved about representability: we have chosen a
    category in which the smooth families are the represented ones, and \Cref{prob:rem:comparison}
    says what that choice excludes.
\end{Rem}

\begin{Rem}[Three levels, and a witness that they differ]
    \label{prob:rem:three-levels}
    Three properties are easily run together, and it is worth keeping them apart.
    \begin{enumerate}
        \item[(L1)] \emph{Denoting.} The program is a morphism $\Xcal\to\PrPf(\Ycal)$: it names a
            distribution.
        \item[(L2)] \emph{Sampling.} That distribution is the law of $f(x,\omega)$ for some
            \emph{measurable} $f$.
        \item[(L3)] \emph{Sampling smoothly.} Such an $f$ may be found which is in addition smooth
            in $x$, i.e.\ which lies in $\Ycal^{\Omega_m}$.
    \end{enumerate}
    Level (L2) is classical and nearly free: on a standard Borel space the quantile transform
    produces such an $f$. Level (L3) is \Cref{prob:thm:representation}, and it is what the pathwise
    gradient of \Cref{ml:thm:pathwise} consumes. The two are not the same, and the quantile
    transform is already the illustration: it is a step function for a discrete family, so it
    witnesses (L2) while saying nothing about (L3).

    Here is the failure in the present setting. Let $\Scal$ be the two-point quasi-universal space
    whose random variables are all measurable maps $\Omega\to\lC0,1\rC$, and consider the Bernoulli
    family $\theta\mapsto\mathrm{Ber}(\theta)$ with values in $\PrPf\lp\flat\Scal\rp$. Level (L2)
    holds, since $f(\theta,\omega) := \I\lB\omega_0\le\theta\rB$ is jointly measurable, $\omega_0$
    being the first coordinate, and $\mathrm{Ber}(\theta) = \lp f(\theta,-)\rp_*U$. Level (L3)
    fails: every $Y \in \lp\flat\Scal\rp^{\Omega_1}$ is independent of $\theta$ by
    \Cref{sss:eg:flat-sharp}, so $\lp Y(\theta,-)\rp_*U$ is constant in $\theta$, whereas
    $\mathrm{Ber}(\theta)$ is not. Equivalently, by \Cref{prob:thm:functor-unit}(3) the object
    $\PrPf\lp\flat\Scal\rp = \flat\PrPf(\Scal)$ is rigid and admits no non-constant curve at all.

    Two readings, both worth having. First, (L3) is a property of the family \emph{together with
    the structure on the target}, not of the family alone: the same Bernoulli family does satisfy
    (L3) into the loose object $\sharp\Scal$, by $Y(\theta,\sigma) := \I\lB(\Theta_2\sigma)_0 \le
    \theta\rB$, every substitution of which is a measurable indicator. Second, nothing is gained
    that way, because $\SSS\lp\sharp\Scal,\R\rp$ consists of the constants alone --- the map
    $(t,\omega)\mapsto\I\lB t>0\rB$ is a plot of $\sharp\Scal$, and a morphism to $\R$ would carry
    it to a step function, which is not a plot of $\R$ --- so no expectation downstream could be
    differentiated. This is why the standard remedy for a discrete latent is a genuinely smooth
    relaxation and not a coarser structure, \Cref{ml:sec:nonsmooth}.
\end{Rem}

\subsection{What a Presentation by Curves Would Give}
\label{prob:sec:comparison}

\begin{Prp}[Plots of $\PrPf(\Xcal)$ are smooth and measurable]
    \label{prob:prp:necessary}
    Every $p \in \PrPf(\Xcal)^{\Omega_n}$ satisfies
    \begin{enumerate}
        \item $p(-,\omega) \in \PrPf(\Xcal)^{\R^n}$ for every $\omega \in \Omega$, and
        \item $p\circ h \in \PrPf(\Xcal)^\Omega$ for every $h \in \Site(\Omega_0,\Omega_n)$;
    \end{enumerate}
    that is, $\PrPf(\Xcal)^{\Omega_n} \ins \widehat{\PrPf(\Xcal)}^{\Omega_n}$ in the notation of
    \Cref{eq:maximal-extension}.
    \begin{proof}
        (2) is (P2). For (1), $p = \Law Y$ and
        $p(t,\omega) = \lp\tilde Y(t,-)\rp_*U$ with
        $\tilde Y := Y\circ\lp\pr_{\R^n},\Theta^{-1}(\omega,-)\rp \in \Xcal^{\Omega_n}$, so
        \Cref{prob:prp:curves} applies.
    \end{proof}
\end{Prp}

\begin{Rem}[Where the difficulty went]
    \label{prob:rem:comparison}
    The converse of \Cref{prob:prp:necessary} asks: given a family $\Ncal$ of measures which is a
    smooth family for each fixed seed and an admissible random measure for each fixed substitution,
    produce a single $Y \in \Xcal^{\Omega_n}$ representing it. That is a measurable selection
    problem --- one must choose a representative of $\Ncal(t,\omega)$ measurably in $\omega$ and
    smoothly in $t$ at once. Selections of Jankov--von Neumann type \cite{Kec95} return universally
    measurable functions, and universal measurability is exactly what ($\Omega$4) makes admissible;
    but they give no control in the $t$-direction, which is the whole difficulty, so they do not by
    themselves settle anything here. Whether the selection can be made we do not know, and nothing
    above depends on the answer.

    That is the point of the arrangement. Had objects been built from a pair
    $\lp\Xcal^\Omega,\Xcal^{\R}\rp$, the mixed family would have been \emph{forced} to be
    $\widehat{\PrPf(\Xcal)}^{\Omega_1}$ while the monad needs $\PrPf(\Xcal)^{\Omega_1}$, so the
    multiplication would exist only where the two agree --- over a subcategory, and only relatively.
    Here the mixed family is primitive and is defined to be the image. The difficulty is therefore
    not dissolved but relocated, out of the existence of the monad and into the question of whether
    the inclusion of \Cref{prob:prp:necessary} is an equality; and that question is now inert.

    Two consequences of the choice should be stated plainly. First, by
    \Cref{sss:rem:curves-do-not-suffice} an object of $\SSS$ is not presented by its curves and its
    random variables: the mixed family is separate data, and the price of \Cref{sss:rem:no-boman} is
    paid at exactly this point. Second, wherever the inclusion is strict
    a Kleisli morphism must come with a representation, so there are fewer of them than a
    curve-and-random-variable presentation would admit; \Cref{prob:thm:representation}(3) says that
    this is nevertheless a well-behaved class, closed under composition.
\end{Rem}

\section{An Application: Gradients of Expectations}
\label{sec:ml}

The applied reason for wanting a probability monad on a category of smooth spaces is that one
differentiates expectations: variational inference, generative modelling and policy-gradient methods
all reduce to estimating $\nabla_\theta \int h \, d\mu_\theta$ for a parametrised family of
distributions, \cite{KW14,RMW14,MRFM20}. This section records what \Cref{sec:prob} gives for that
problem. It is short by design: the categorical work is done, and what remains is one classical
analytic hypothesis and a remark on data that is not smooth.

\subsection{The Reparametrisation Gradient}
\label{ml:sec:reparam}

A parametrised family of distributions on $\Xcal$ is a Kleisli morphism $\R^q \to \PrPf(\Xcal)$,
i.e.\ a sample-smooth map $\theta \mapsto \mu_\theta$. By \Cref{prob:thm:representation}(2) every
such family is \emph{already} a reparametrisation: there is $Y \in \Xcal^{\Omega_q}$, smooth in
$\theta$ and admissible in the seed, with $\mu_\theta = \lp Y(\theta,-)\rp_*U$. No hypothesis and no
construction is involved; this is what it means to be a morphism into $\PrPf(\Xcal)$.

What is not free is exchanging the derivative with the integral. That is analysis, and by
\Cref{prob:thm:expectation} it genuinely requires a hypothesis.

\begin{Thm}[Pathwise gradient]
    \label{ml:thm:pathwise}
    Let $Y \in \lp\R^d\rp^{\Omega_q}$, let $\mu_\theta := \lp Y(\theta,-)\rp_*U$, and let
    $h \in \SSS\lp\R^d,\R\rp = C^\infty\lp\R^d,\R\rp$ be bounded with $\nabla h$ bounded. Assume
    \begin{enumerate}
        \item[(D)] for every compact $K \ins \R^q$ there is a $U$-integrable
            $B_K:\,\Omega\to\lB0,\infty\rB$ such that
            $\sup_{\theta\in K}\lV \partial_\theta Y(\theta,\omega)\rV \le B_K(\omega)$
            holds for all $\omega \in \Omega$.
    \end{enumerate}
    Then $\theta \mapsto \int h\,d\mu_\theta$ is continuously differentiable on $\R^q$, with
    \begin{align}
        \nabla_\theta \int h \, d\mu_\theta &= \int_\Omega \partial_\theta Y(\theta,\omega)^\top\,
        \nabla h\lp Y(\theta,\omega)\rp \; U(d\omega).
        \label{eq:pathwise}
    \end{align}
    \begin{proof}
        By the change of variables formula $F(\theta) := \int h\,d\mu_\theta = \int_\Omega
        h\lp Y(\theta,\omega)\rp\,U(d\omega)$, the integrand being bounded and, by (M2) of
        \Cref{site:def:site} together with continuity of $h$, measurable; so $F$ is well defined.
        Put
        \begin{align}
            G_i(\theta,\omega) &:= \nabla h\lp Y(\theta,\omega)\rp^\top
            \partial_{\theta_i} Y(\theta,\omega) .
        \end{align}
        Each $G_i$ is measurable in $\omega$, since $\partial_\theta Y$ is a pointwise limit of the
        difference quotients $\omega\mapsto k\lp Y(\theta+e_i/k,\omega)-Y(\theta,\omega)\rp$, each
        admissible by (M2), and $\nabla h$ is continuous; it is continuous in $\theta$ by (M1); and
        on a compact $K$ it is bounded by $\sup_{\R^d}\lV\nabla h\rV \cdot B_K$, which is
        $U$-integrable by (D). In particular the right-hand side of \Cref{eq:pathwise} exists.

        Fix $\theta_0$ and let $K$ be a closed ball around it, so that $K$ is convex. For
        $\theta_0+se_i \in K$ with $s \ne 0$ the map $\varphi(s) := h\lp Y(\theta_0+se_i,\omega)\rp$
        is $C^1$ by (M1) and the chain rule, with $\varphi'(s) = G_i(\theta_0+se_i,\omega)$, so the
        mean value theorem applied along the segment, which lies in $K$, gives
        \begin{align}
            \frac{\lI \varphi(s) - \varphi(0)\rI}{\lI s \rI} &\le
            \sup\nolimits_{\R^d}\lV\nabla h\rV \cdot B_K(\omega),
        \end{align}
        a bound independent of $s$ and $U$-integrable. Dominated convergence along any sequence
        $s\to0$ therefore gives $\partial_{\theta_i}F(\theta_0) = \int G_i(\theta_0,\omega)\,
        U(d\omega)$, which is \Cref{eq:pathwise} componentwise. The same dominating function, with
        dominated convergence applied to $\theta\mapsto G_i(\theta,\omega)$, shows that
        $\theta\mapsto\int G_i(\theta,\omega)\,U(d\omega)$ is continuous on $K$. So every partial
        derivative of $F$ exists and is continuous, and $F$ is $C^1$.
    \end{proof}
\end{Thm}

\begin{Rem}[The hypothesis is not decoration]
    \label{ml:rem:necessity}
    Condition (D) cannot be dropped. The proof of \Cref{prob:thm:expectation} exhibits, in the present
    normalisation, $Y(t,\sigma) = t\,C(\sigma)$
    with $C$ Cauchy and $h = \cos$ --- bounded, with all derivatives bounded --- for which
    $\int h\,d\mu_t = e^{-\lI t\rI}$ is not differentiable at $0$. There
    $\partial_t Y(t,\omega) = C(\omega)$ is not $U$-integrable, which is exactly the failure of (D).
    So the split is clean: the \emph{representation} of a family as a reparametrisation is
    categorical and free, \Cref{prob:thm:representation}; the \emph{estimator} is analysis and costs
    an integrability hypothesis.
\end{Rem}

\begin{Rem}[Why the categorical differential does not apply]
    \label{ml:rem:not-a-morphism}
    \Cref{ml:thm:pathwise} concludes that $F(\theta) = \int h\,d\mu_\theta$ is $C^1$. Nothing in
    it makes $F$ a morphism of $\SSS$: a morphism $\R^q\to\R$ is a $C^\infty$ map by
    \Cref{sss:thm:euclidean}, whereas without hypothesis (D) the composite
    $\theta\mapsto\mu_\theta\mapsto\int h\,d\mu_\theta$ need not even be $C^1$,
    \Cref{prob:thm:expectation}, and this for $h$ bounded with all derivatives bounded. Whether (D)
    alone forces more than $C^1$ we do not address. So in general the tangent functor and
    differentials of \Cref{sss:sec:tangent} do not apply to $F$.

    What \emph{is} a morphism, and what \Cref{sss:sec:tangent} does differentiate, is the sampler:
    $Y \in \lp\R^d\rp^{\Omega_q} = \SSS\lp\Omega_q,\R^d\rp$ is a plot,
    $\theta\mapsto Y(\theta,\omega)$ is a smooth family for each seed, and by
    \Cref{tan:thm:euclidean} its differential is the classical $\partial_\theta Y$.
    \Cref{eq:pathwise} is the
    statement that the gradient of the objective is the $U$-average of the differential of the
    sampler; the categorical content is on the right-hand side, the analytic hypothesis is what
    licenses the equality. This is the precise sense in which the reparametrisation gradient is
    ``differentiating the sampler rather than the objective'', and it is why an automatic
    differentiator applied to such a program is on solid ground while the objective it reports is
    not, by itself, a morphism.
\end{Rem}

\begin{Eg}[Location--scale families and the evidence lower bound]
    \label{ml:eg:location-scale}
    Let $Z \in \lp\R^d\rp^\Omega$ with $\E\lV Z\rV < \infty$, and let
    $m \in C^\infty(\R^q,\R^d)$ and $S \in C^\infty(\R^q,\R^{d\times d})$. Then
    $Y(\theta,\omega) := m(\theta) + S(\theta)\,Z(\omega)$ lies in $\lp\R^d\rp^{\Omega_q}$ and
    satisfies (D), since $\partial_\theta Y$ is bounded on compacta by
    $c_K\lp 1 + \lV Z(\omega)\rV\rp$. Taking $Z$ standard normal, this covers the Gaussian
    family $\mu_\theta = \Nrm\lp m(\theta),\,S(\theta)S(\theta)^\top\rp$ and hence the standard
    reparametrisation of \cite{KW14,RMW14}. Families with no closed-form reparametrisation, such as
    the Gamma, are handled by implicit differentiation of the distribution function \cite{FMM18};
    that is a statement about how to \emph{compute} $\partial_\theta Y$, not about whether $Y$
    exists, which is \Cref{prob:thm:representation}.

    An objective such as the evidence lower bound has a parameter in the integrand as well,
    $\theta \mapsto \int h(\theta,x)\,\mu_\theta(dx)$. This is the same theorem applied to
    $\widetilde Y(\theta,\omega) := \lp\theta,\,Y(\theta,\omega)\rp \in
    \lp\R^{q+d}\rp^{\Omega_q}$, whose law is $\delta_\theta\otimes\mu_\theta$, and to $h$ read as a
    function on $\R^{q+d}$ --- so that $h$ must now be bounded with bounded gradient \emph{jointly}
    in $(\theta,x)$, a real restriction and not a formality, since the integrand of an evidence
    lower bound is typically unbounded in both arguments. Under it, \Cref{eq:pathwise} returns both
    terms,
    \begin{align}
        \nabla_\theta \int h(\theta,x)\,\mu_\theta(dx) &= \int_\Omega \Big[
        \nabla_\theta h\lp\theta,Y(\theta,\omega)\rp + \partial_\theta Y(\theta,\omega)^\top
        \nabla_x h\lp\theta,Y(\theta,\omega)\rp \Big]\, U(d\omega),
    \end{align}
    since $\partial_\theta\widetilde Y = \lp\id,\partial_\theta Y\rp$ satisfies (D) whenever
    $\partial_\theta Y$ does, with dominating function $\sqrt{q} + B_K$.
\end{Eg}

\begin{Rem}[Why a monad, and not just a reparametrisation]
    \label{ml:rem:composition}
    A probabilistic program is not one sampling step but a composite of them, and the reason to
    insist on \Cref{prob:thm:monad} rather than on a single representation theorem is
    \Cref{prob:thm:representation}(3): Kleisli composition of representable families is again
    representable, because a composite of morphisms is a morphism. Concretely, if
    $K:\,\R^q\to\PrPf(\R^d)$ draws a latent and $L:\,\R^d\to\PrPf(\R^e)$ draws an observation from
    it, then $L\bullet K$ is again of the form $\lp W(\theta,-)\rp_*U$ with
    $W \in \lp\R^e\rp^{\Omega_q}$, obtained by seed splitting; one may then apply
    \Cref{ml:thm:pathwise} to the composite. Neither the representation nor its stability under
    composition carries a hypothesis. Condition (D), on the other hand, has to be checked for the
    composite, and is not implied by (D) for the two factors.
    \Cref{ml:eg:composite} carries all of this out.
\end{Rem}

\begin{Eg}[A two-step model, end to end]
    \label{ml:eg:composite}
    Fix $m,s \in C^\infty(\R^q,\R)$ with $s > 0$, a link $g \in C^\infty(\R,\R)$ and $\tau > 0$, and
    let $Z \in \R^\Omega$ be standard normal --- say $Z(\omega) := \Fnorm^{-1}(\omega_0)$ for
    $\omega_0 \in (0,1)$ and $Z(\omega) := 0$ otherwise, where $\Fnorm$ is the standard normal
    distribution function and $\omega_0$ is the first coordinate of $\omega$. Consider the two
    stages
    \begin{align}
        K:\;\R^q \to \PrPf(\R), & \qquad K(\theta) := \Nrm\lp m(\theta),\,s(\theta)^2\rp, \\
        L:\;\R \to \PrPf(\R), & \qquad L(z) := \Nrm\lp g(z),\,\tau^2\rp,
    \end{align}
    a latent drawn from a parametrised Gaussian and an observation drawn given the latent.

    \emph{Each stage is a Kleisli morphism.} Put
    \begin{align}
        F_K(\theta,\omega) &:= m(\theta) + s(\theta)\,Z(\omega), &
        F_L(z,\omega) &:= g(z) + \tau\,Z(\omega).
    \end{align}
    Both are sample-smooth on $\R^q\times\flat\Omega$ resp.\ $\R\times\flat\Omega$, being smooth in
    the first argument and Borel in the second, and their laws are $K$ and $L$; so $K$ and $L$ are
    Kleisli morphisms by \Cref{prob:prp:sampler-to-kernel}. Nothing has been checked beyond
    smoothness of $m,s,g$ and measurability of $Z$.

    \emph{The composite, and its sampler.} By \Cref{prob:thm:representation}(3) the Kleisli
    composite $L \bullet K$ is again a sampler, and the proof of \Cref{prob:thm:mult} says which
    one: run the first stage on one half of the seed and the second on the other. Explicitly, put
    \begin{align}
        W(\theta,\omega) &:= F_L\Big( F_K\lp\theta,\Theta_1\omega\rp,\ \Theta_2\omega \Big)
        = g\Big( m(\theta) + s(\theta)\,Z\lp\Theta_1\omega\rp \Big)
        + \tau\,Z\lp\Theta_2\omega\rp .
        \label{eq:composite-sampler}
    \end{align}
    Then $W \in \R^{\Omega_q}$, being smooth in $\theta$ and Borel in $\omega$, and
    $\lp W(\theta,-)\rp_*U = \lp L\bullet K\rp(\theta)$. Indeed $Z\circ\Theta_1$ and
    $Z\circ\Theta_2$ are independent standard normals by \Cref{qus:lem:split-U}, so under $U$ the
    variable $z := m(\theta)+s(\theta)\,Z(\Theta_1\omega)$ has law $K(\theta)$ and, given it,
    $W(\theta,\omega)$ has law $\Nrm\lp g(z),\tau^2\rp = L(z)$; integrating out $z$ is
    $\Mbb$, and gives $\int L(z)\,K(\theta)(dz) = \lp L\bullet K\rp(\theta)$.

    \emph{The gradient.} Assume in addition that $g'$ is bounded, and let $h \in C^\infty(\R,\R)$ be
    bounded with $h'$ bounded. Then
    \begin{align}
        \partial_\theta W(\theta,\omega) &= g'\Big( m(\theta)+s(\theta)Z(\Theta_1\omega)\Big)
        \Big( \nabla m(\theta) + \nabla s(\theta)\,Z(\Theta_1\omega)\Big),
    \end{align}
    which on a compact $K \ins \R^q$ is bounded by
    $c_K \lV g'\rV_\infty \lp 1 + \lI Z(\Theta_1\omega)\rI\rp$, a $U$-integrable function since
    $\E\lI Z\rI < \infty$. So (D) holds and \Cref{ml:thm:pathwise} applies to the composite:
    \begin{align}
        \nabla_\theta \int h \, d\lp L\bullet K\rp(\theta) &= \int_\Omega
        h'\lp W(\theta,\omega)\rp\, \partial_\theta W(\theta,\omega)\; U(d\omega),
    \end{align}
    which is what an implementation computes when it draws the two seed halves once and
    differentiates the arithmetic of \Cref{eq:composite-sampler}.

    \emph{(D) is not inherited.} Take $p = 1$, $m(\theta) := \theta$, $s :\equiv 1$ and
    $g(z) := e^{z^2}$. Each stage satisfies (D): $\partial_\theta F_K \equiv 1$, and
    $\partial_z F_L = g'$ is bounded on compacta. The composite does not, since at $\theta = 0$
    \begin{align}
        \E\lI \partial_\theta W \rI &= \E\lI g'(Z)\rI = \frac{2}{\sqrt{2\pi}}
        \int_\R \lI z\rI\, e^{z^2/2}\,dz = \infty .
    \end{align}
    The representation composed; the integrability hypothesis did not. This is the division of
    labour of \Cref{ml:rem:necessity}, seen in a program.
\end{Eg}

\subsection{Data That Is Not Smooth}
\label{ml:sec:nonsmooth}

Programs also manipulate values one must not differentiate: discrete latents, comparison bits,
branch indices, labels. Here the framework is deliberately strict.

By \Cref{sss:thm:euclidean}, $\SSS\lp\R,\R\rp = C^\infty(\R,\R)$, so $x\mapsto\max(x,0)$ and
$x\mapsto\I\lB x>0\rB$ are simply not morphisms $\R\to\R$. Nothing is broken by this; the point
is that a map which is not differentiable does not become differentiable by being written down, and
the category says so.

The repair is to change the \emph{sample-smooth structure} --- that is, the object --- and not the
map.

\begin{Def}[The branched line]
    \label{ml:def:branched}
    Let $\R_{\mathrm{br}} := \lp-\infty,0\rp \sqcup \lB0,\infty\rp$ be the coproduct in
    $\SSS$ of the two subobjects of $\R$, so that its underlying set is again $\R$ and, by
    \Cref{sss:prp:prod-coprod},
    \begin{align}
        \R_{\mathrm{br}}^{\Omega_n} &= \lC p \in \R^{\Omega_n} \st p < 0 \text{ everywhere}
        \rC \;\cup\; \lC p \in \R^{\Omega_n} \st p \ge 0 \text{ everywhere} \rC .
    \end{align}
\end{Def}

\begin{Prp}[The rectifier on the branched line]
    \label{ml:prp:relu}
    The map $\max(-,0):\,\R_{\mathrm{br}}\to\R$ is sample-smooth, with derivative $0$ on the
    open summand and $1$ on the closed one. The identity $\R_{\mathrm{br}}\to\R$ is sample-smooth
    but not an isomorphism. Moreover no admissible curve of $\R_{\mathrm{br}}$ approaches the kink
    at non-zero speed from inside the closed summand: if $\gamma \in \R_{\mathrm{br}}^\R$ takes
    values in $\lB0,\infty\rp$ and $\gamma(t_0) = 0$, then $\gamma'(t_0) = 0$.
    \begin{proof}
        A plot of $\R_{\mathrm{br}}$ lands in one summand by \Cref{ml:def:branched}, and
        $\max(-,0)$ restricts on the two summands to the constant $0$ and to the identity, both
        smooth; so it carries plots to plots. The identity is sample-smooth because every plot of
        $\R_{\mathrm{br}}$ is a plot of $\R$, and its inverse is not, because a smooth curve
        crossing $0$ --- say $\gamma(t) = t$ --- is a plot of $\R$ and not of
        $\R_{\mathrm{br}}$. For the last claim, $\gamma$ attains its minimum at $t_0$ and is
        differentiable there.
    \end{proof}
\end{Prp}

So the rectifier is a morphism as soon as one names the object on which it is one; nothing else has
to change --- it composes, and \Cref{sec:prob} applies to it verbatim.

What one pays is visible and is the right price. Refining the source discards every curve
that crosses $0$ --- in particular all those along which $\max(-,0)$ fails to be differentiable, and
some others besides --- which is what an implementation does when it picks a value at the kink and
never looks across it. The last clause of \Cref{ml:prp:relu} says that the framework records the
absence of a two-sided derivative rather than inventing one. For a network with several such units
the summands must be refined to the linear regions, so the object depends on the program; that
bookkeeping is where the interest lies, and we do not pursue it here.

Two coarser repairs are also available, from \Cref{sec:modal}. One may keep the map and coarsen the
\emph{target}: by \Cref{modal:thm:triple} a map $\Xcal\to\sharp\Scal$ is sample-smooth as soon as it
is quasi-measurable, so $\max(-,0)$, and the indicator of any universally measurable set, are
morphisms $\R\to\sharp\Und\R$, and the type then records that the result is not available for
differentiation. Or one may type an input by $\flat\Scal$, which carries no non-constant plot in the
smooth direction, so that nothing downstream can differentiate with respect to it; by
\Cref{prob:thm:functor-unit}(3) the probability monad restricts to the rigid objects,
$\PrPf\circ\flat = \flat\circ\PrPf$, so discrete randomness is handled inside the same monad, just
without gradients.

A third repair, of a different kind, is to replace the family itself by a genuinely smooth one --- a
temperature-controlled relaxation, as in \cite{JGP17,MMT17} --- which is again a change of object,
this time to one on which \Cref{ml:thm:pathwise} applies.

The general shape is worth stating once. Smoothness here is a property of objects, not something one
asserts of a map; a map that is not differentiable becomes a morphism exactly when one names the
object on which it is. We do not develop this further, since none of \Cref{sec:prob} depends on it.

\section{Discussion}
\label{sec:discussion}

\subsection{Status of the Results}
\label{disc:sec:status}

The paper has two main results; everything else is either preparation for them or a
consequence of them.
\begin{center}
\begin{tabular}{@{}l p{0.70\textwidth}@{}}
    \Cref{prob:thm:monad} & $\PrPf$ is a strong commutative affine monad on \emph{all} of $\SSS$
    \\[3pt]
    \Cref{prob:thm:markov} & its Kleisli category is a Markov category, on \emph{all} of $\SSS$
\end{tabular}
\end{center}
The point of the arrangement is that the hypothesis column of \Cref{sec:prob} is empty,
\Cref{tab:hypotheses}.

\begin{table}[htbp]
\centering
\renewcommand{\arraystretch}{1.2}
\begin{tabular}{@{}p{0.42\textwidth} l l@{}}
    \toprule
    \emph{structure} & \emph{where} & \emph{hypothesis} \\
    \midrule
    functor and unit & \Cref{prob:thm:functor-unit} & none \\
    push-forward along a varying map & \Cref{prob:lem:pf-enriched} & none \\
    product of measures & \Cref{prob:thm:otimes} & none \\
    product of independent kernels & \Cref{prob:cor:kernel-indep} & none \\
    strength and costrength & \Cref{prob:cor:strength} & none \\
    multiplication & \Cref{prob:thm:mult} & none \\
    monad; strong, commutative, affine & \Cref{prob:thm:monad} & none \\
    product of dependent kernels & \Cref{prob:thm:kernel-dep} & none \\
    Markov category & \Cref{prob:thm:markov} & none \\
    representation by samplers & \Cref{prob:thm:representation} & none \\
    samplers modulo law & \Cref{prob:cor:sampler-bijection} & none \\
    \bottomrule
\end{tabular}
\caption{Every structural result of \Cref{sec:prob}, and what it assumes of the objects.}
\label{tab:hypotheses}
\end{table}
Every numbered statement outside \Cref{sec:qus} is proved here, and there are no conjectures;
\Cref{sec:qus} recalls \cite{For21} without proof, with a section pointer on each statement. The
substantive inputs from elsewhere are: the properties ($\Omega$1)--($\Omega$6) of the universal
Hilbert cube, the induced $\sigma$-algebras, the quasitopos structure of $\QUS$ and the
push-forward monad with its product of measures, all from \cite[\S\S2, 3, 5, 7]{For21}; the quasitopos theorem for
concrete sheaves on a concrete site, from \cite{BH11}, see also \cite{Dub79,MMS22}, used in
\Cref{sss:thm:quasitopos} as a cross-check, every clause --- local cartesian closedness included
--- being also proved directly in \Cref{sss:sec:constructions,sss:sec:cc}; the Carath\'eodory joint
measurability lemma, proved in \Cref{qus:lem:caratheodory}; and, for
\Cref{prob:thm:markov}, the fact that the Kleisli category of a commutative affine monad on a
cartesian category is a Markov category, from \cite{Fri20}.

It is worth recording what is \emph{not} claimed, in one place.
\begin{itemize}
    \item \emph{Expectation is not a morphism.} $\mu\mapsto\int h\,d\mu$ fails to be
        sample-smooth already for $h=\cos$ on the real line, \Cref{prob:thm:expectation}; so the
        objective of a variational program is not interpreted by the structure that interprets the
        program, and \Cref{ml:thm:pathwise} bridges the gap by an analytic hypothesis rather than a
        categorical one, \Cref{ml:rem:not-a-morphism}.
    \item \emph{Conditioning is absent.} \Cref{prob:thm:markov} gives a Markov category and no
        more; conditionals, almost-sure equality and positivity are not addressed,
        \Cref{prob:rem:conditionals}, and the reason is stated there: the construction removes
        problems of \emph{selecting} a representative, not problems of \emph{inverting} a
        push-forward.
    \item \emph{The tangent space is only a cone.} Outside Cartesian spaces and manifolds
        $T_x\Xcal$ need carry no linear structure, \Cref{tan:prp:cotangent}, and what
        $T\PrPf(\R^n)$ looks like we do not know.
    \item \emph{The inclusion of \Cref{prob:prp:necessary} is not claimed to be an equality.} By
        \Cref{sss:prp:reflection} that is the question whether $\PrPf$ lands in
        $\SSS_{\mathrm{pr}}$, and \Cref{prob:rem:comparison} explains that it is precisely the
        question the construction was designed not to depend on.
    \item \emph{There is no locality.} Plot families are not required to be local in the parameter
        or patchable in the seed, \Cref{site:rem:no-locality}; \Cref{disc:rem:locality} explains
        why locality in the parameter would cost \Cref{prob:thm:representation}. Seed patching is
        not discussed.
\end{itemize}

\subsection{What the Mixed Site Buys, and What it Costs}
\label{disc:sec:tradeoff}

\begin{Rem}[The ledger]
    \label{disc:rem:ledger}
    Gained: an unconditional probability monad, \Cref{prob:thm:monad}, with an unconditional Markov
    Kleisli category, \Cref{prob:thm:markov}, an unconditional dependent product of kernels,
    \Cref{prob:thm:kernel-dep}, and an unconditional representation theorem,
    \Cref{prob:thm:representation}; the compatibility axiom relating smoothness and measurability
    becomes a consequence, \Cref{sss:lem:Q4}; and Euclidean and manifold correctness become
    immediate, \Cref{sss:thm:euclidean,sss:thm:manifolds}, so that Boman's theorem
    disappears from the foundations.

    Lost: an object is no longer \emph{presented} by its curves and its random variables,
    \Cref{sss:rem:curves-do-not-suffice}, so the economy of a two-family definition is gone and
    every construction is indexed by an arity; and, wherever the inclusion of
    \Cref{prob:prp:necessary} is strict, there are fewer Kleisli morphisms than a two-family
    presentation would give, \Cref{prob:rem:comparison}. Unchanged: the measurable half, \Cref{prob:prp:underlying}, the
    smooth families, \Cref{prob:prp:curves}, and hence every concrete computation that either
    theory can express.
\end{Rem}

\begin{Rem}[On locality]
    \label{disc:rem:locality}
    We do not impose a sheaf condition, and \Cref{site:rem:no-locality} explains that this is a
    choice rather than an oversight. It is worth saying what each side of the choice is for.

    Locality --- test objects $V\times\Omega$ for open $V \ins \R^n$, and plots required to be local
    in the smooth direction --- buys \emph{gluing}: constructing a morphism from local pieces, and
    constructing an object from local data. That is what differential-geometric arguments on objects
    which are not manifolds need: de Rham theory, integration of forms, connections, bundles
    presented by transition functions. It is \emph{not} needed for tangent and cotangent spaces, for
    Jacobians, or for classical differential geometry on manifolds, which by
    \Cref{sss:thm:manifolds} may be quoted rather than redone.

    Against that, locality interacts badly with \Cref{prob:def:PrPf}. A family of measures that is
    locally in $t$ a push-forward need not be one globally, since local representatives agree only
    in law on overlaps; so imposing a sheaf condition would force
    $\PrPf(\Xcal)^{\Omega_n}$ to be replaced by its sheafification, a possibly larger class whose new
    members are exactly the families with no global representative. The multiplication survives
    that change --- run the argument of \Cref{prob:thm:mult} on each member of the cover --- but
    \Cref{prob:thm:representation}, which is the point of the construction, does not. Locality is
    therefore not a neutral addition here.
\end{Rem}

\begin{Rem}[A single smooth test object]
    \label{disc:rem:one-object}
    The family $\lC\Omega_n\rC_{n\ge0}$ is the product closure of $\lC\Omega_0,\Omega_1\rC$, and one may ask
    whether it can be replaced by a single self-absorbing object $E\times\Omega$ with
    $E\times E\cong E$, so that the site becomes a one-object category and
    \Cref{sec:sss} becomes a rerun of \Cref{sec:qus} with $\Omega$ replaced by
    $E\times\Omega$. Both natural candidates work formally --- the Fr\'echet space $\R^\N$
    absorbs products by interleaving, exactly as $\Omega$ does, and the colimit $\R^{(\N)}$ is
    conservative over the present site --- but each carries a cost. For $\R^\N$ the plot family is
    strictly richer than the tower, since smoothness on a Fr\'echet space is not detected on
    finite-dimensional slices, and one inherits the apparatus of convenient calculus \cite{KM97}.
    For $\R^{(\N)}$ the reduction to the tower requires plots to factor locally through finite
    stages, hence a locality condition, which \Cref{disc:rem:locality} argues against. We keep the
    tower; every proof above is uniform in $n$, so it costs nothing.
\end{Rem}

\subsection{Relation to Other Work}
\label{disc:sec:related}

The measurable half is \cite{For21}, itself a development of the quasi-Borel spaces of
\cite{HKSY17}; see \cite{VKS19,Ste21} for the domain-theoretic and semantic development, and
\cite{KPS25} for work in progress on infinite products in quasi-Borel spaces. The smooth half is in the
tradition of diffeological and Fr\"olicher spaces, \cite{Sou80,IZ13,BKW25}, and the general
machinery of concrete sheaves on a concrete site is \cite{BH11,Dub79}, in the form we use it
\cite{MMS22}.

Both halves are visible inside $\SSS$ as comparison functors, and it is worth saying which way they
run. Towards the smooth side, \Cref{sss:prp:diffeology} gives $\mathrm{Dfg}:\,\SSS\to\Diff$,
faithful, correct on manifolds, and not full: a diffeological map is tested only on smooth plots,
whereas a sample-smooth map must in addition carry random variables to random variables. There is no
canonical functor back, \Cref{sss:rem:diffeology}. Towards the measurable side, $\Und$ of
\Cref{modal:lem:functors} is the corresponding functor $\SSS\to\QUS$, and it is better behaved:
it is a genuine adjoint, sitting in the middle of the string of \Cref{modal:thm:string}, with
$\flat$ and $\sharp$ both splitting it. Since $\QUS$ is itself the quasi-Borel construction with
the universal Hilbert cube as sample space, \cite{For21,HKSY17}, a reader arriving from the
quasi-Borel tradition reaches $\SSS$ through $\flat$ or $\sharp$ and gets back what they started
with. Neither comparison is used below; they are recorded because a reader arriving from either
tradition will want to know that the objects they already have are objects here.

The design decision of this paper is visible against that background. \cite{MMS22} combines two
concrete sites by a \emph{sum}, so that objects carry two independent plot families related only by
constants; that is the shape a presentation by curves and random variables would have, and
\Cref{site:rem:mixed} identifies it as the source of the obstruction. We take a \emph{product} of
test objects instead. We are not aware
of this instance elsewhere, although nothing in \cite{BH11,Dub79,MMS22} forbids it.

On the semantics side, \cite{HSV20} uses diffeological spaces for the correctness of automatic
differentiation and does not treat probability; \cite{HLMS23} handles non-smoothness by
piecewise-analytic plot domains and obtains total monads of samplers and of integrators, while
\cite{LRY23} attacks the same problem by a static smoothness analysis; \cite{LHSM23} is the
corresponding account of gradient estimators. Relative monads in the sense of
\cite{ACU15}, in the Kleisli-triple formulation of \cite{Man76}, are what one is left with in place
of \Cref{prob:thm:monad} if the mixed families are computed rather than defined. Markov categories
are \cite{Fri20}. Cohesion is \cite{Law07}.

\subsection{Open Questions}
\label{disc:sec:open}

\begin{enumerate}
    \item \textbf{Is the inclusion of \Cref{prob:prp:necessary} an equality?} Equivalently: does
        every family of measures that is smooth in the parameter and admissible in the seed admit a
        joint representative? By \Cref{prob:rem:comparison} this is a measurable selection problem
        with an additional smoothness requirement in the parameter, which the classical selection
        theorems \cite{Kec95} do not address; we know of no example where the inclusion is strict,
        \Cref{sss:rem:curves-do-not-suffice}. A positive answer would say that $\PrPf(\Xcal)$ is
        presented by its $n$-parameter families and its random variables after all.
    \item \textbf{Conditionals.} \Cref{prob:thm:markov} stops at a Markov category. When does
        $\Kleisli(\PrPf)$ have conditionals in the sense of \cite{Fri20} --- that is, when can a
        joint morphism $\Xcal\to\PrPf(\Ycal\times\Zcal)$ be factored through a morphism
        $\Xcal\times\Ycal\to\PrPf(\Zcal)$? \Cref{prob:rem:conditionals} says why this is a
        different kind of problem from the ones solved here: it asks one to \emph{invert} a
        push-forward, and no choice of plot families can make that free. Disintegration in $\QUS$
        is available, \cite[\S7]{For21}, so the question is what has to be assumed of $\Xcal$ and
        $\Ycal$ for the disintegrating kernel to be sample-smooth in the parameter, which is where
        the smoothness of the site re-enters.
    \item \textbf{Algebras.} $\PrPf$ is now an honest endofunctor, so the Eilenberg--Moore category
        is defined. What are the $\PrPf$-algebras? In the measurable half they should be a form of
        convex space; the smooth directions ought to add a differentiable structure to the convex
        combination, and we do not know what that amounts to.
    \item \textbf{Tangent structure.} \Cref{sss:sec:tangent} gives correct derivatives on
        Cartesian spaces and manifolds, and \cite{CW16} builds the corresponding theory for
        diffeological spaces. Does $\SSS$ carry a tangent structure in the axiomatic sense of
        \cite{Ros84,CC14}, at least on a subcategory on which $T$ is not already the identity?
    \item \textbf{The tangent space of a space of measures.} What is $T\PrPf(\R^n)$? By
        \Cref{prob:prp:curves} a tangent vector at $\mu$ is a reparametrisable curve through $\mu$
        up to first order, so the question is which ``directions'' at $\mu$ are witnessed by a
        smooth family of random variables. A comparison with the Wasserstein tangent space
        suggests itself and we have not made it.
    \item \textbf{Is $\SSS_{\mathrm{pr}}$ a proper subcategory?} \Cref{sss:prp:reflection}
        exhibits the presented objects as a reflective subcategory, and we know of no object outside
        it. A negative answer to question~1 would produce one; conversely
        $\SSS_{\mathrm{pr}} = \SSS$ would settle question~1 positively. An arbitrary object
        outside $\SSS_{\mathrm{pr}}$ would already be the first concrete separation between the
        present category and a presentation by $n$-parameter families and random variables, without
        deciding question~1 either way.
    \item \textbf{Comparison with $\omega$PAP.} \cite{HLMS23} obtains total monads on a category of
        plots with piecewise-analytic domains, by way of a sampler monad and an integrator monad.
        Their sampler monad is weighted, partial and threads the residual seed; its unweighted
        total fragment plays the role of the present $\PrPf$ \emph{before} quotienting by equality
        in law. Is there a comparison functor relating the two, and does that quotient match
        \Cref{prob:def:PrPf} under it?
    \item \textbf{Expectation-based structures.} Replacing \Cref{eq:PrPf-plots} by a condition on
        $t\mapsto\int h\,d\mu_t$ for test functions $h$ gives a different structure, for which
        the multiplication is easy and the product of measures is not. By
        \Cref{prob:thm:expectation} it does not contain the present one; whether it is contained in
        it we do not know. What does it look like over the present site, and is the failure of
        monoidality genuine?
\end{enumerate}

\medskip

A last word on what the paper is for. The obstruction it removes is small and completely specific:
one plot family, at one kind of test object, is a definition rather than a derivation. Everything
downstream of it --- the monad laws, the Markov structure, the representation of every Kleisli
morphism by a sampler --- follows from that one choice by seed splitting, without a hypothesis
anywhere, while the quasitopos structure and the modalities come from the site and do not depend on
it at all. Whether the choice is a real one, or whether the two
families would have forced it all along, is \Cref{sss:prp:reflection} and question~1 above, and we
do not know. What can be said is that the construction does not depend on the answer, which is the
property one wants of a foundation: a semantics for differentiable probabilistic programs should not
have to wait on a selection theorem.

\phantomsection
\addcontentsline{toc}{section}{Acknowledgements}
\section*{Acknowledgements}

This note was prepared with the assistance of Claude Opus 5, a large language model developed by Anthropic. Based on the author's notes, ideas and input, the model was used to re-draft and restructure the exposition, to prepare the typescript, to search for and cross-check references, to check and complete arguments, and to provide feedback, which the author used to refine further inputs. All references, statements and proofs have been verified by the author, who takes full
responsibility for the content, including any remaining errors.


\begin{thebibliography}{HLMS23}

\bibitem[ACU15]{ACU15}
Thorsten Altenkirch, James Chapman, and Tarmo Uustalu, \emph{{Monads need not
  be endofunctors}}, Logical Methods in Computer Science \textbf{11} (2015),
  no.~1, 3:1--3:40. \doi{10.2168/LMCS-11(1:3)2015} \eprintlink{1412.7148}

\bibitem[BH11]{BH11}
John~C. Baez and Alexander~E. Hoffnung, \emph{{Convenient categories of smooth
  spaces}}, Transactions of the American Mathematical Society \textbf{363}
  (2011), no.~11, 5789--5825. \doi{10.1090/S0002-9947-2011-05107-X}
  \eprintlink{0807.1704}

\bibitem[BKW25]{BKW25}
Augustin Batubenge, Yael Karshon, and Jordan Watts, \emph{{Diffeological,
  Fr\"olicher, and differential spaces}}, Rocky Mountain Journal of Mathematics
  \textbf{55} (2025), no.~3, 637--672. \doi{10.1216/rmj.2025.55.637}
  \eprintlink{1712.04576}

\bibitem[Bom67]{Bom67}
Jan Boman, \emph{{Differentiability of a function and of its compositions with
  functions of one variable}}, Mathematica Scandinavica \textbf{20} (1967),
  249--268. \doi{10.7146/math.scand.a-10835}

\bibitem[CC14]{CC14}
J.~Robin~B. Cockett and Geoff~S.~H. Cruttwell, \emph{{Differential structure,
  tangent structure, and SDG}}, Applied Categorical Structures \textbf{22}
  (2014), no.~2, 331--417. \doi{10.1007/s10485-013-9312-0}

\bibitem[CW16]{CW16}
J.~Daniel Christensen and Enxin Wu, \emph{{Tangent spaces and tangent bundles
  for diffeological spaces}}, Cahiers de Topologie et G\'eom\'etrie
  Diff\'erentielle Cat\'egoriques \textbf{57} (2016), no.~1, 3--50.
  \eprintlink{1411.5425}

\bibitem[Dub79]{Dub79}
Eduardo~J. Dubuc, \emph{{Concrete quasitopoi}}, Applications of Sheaves,
  Lecture Notes in Mathematics, vol.~753, Springer, Berlin, 1979, pp.~239--254.
  \doi{10.1007/BFb0061821}

\bibitem[FMM18]{FMM18}
Mikhail Figurnov, Shakir Mohamed, and Andriy Mnih, \emph{{Implicit
  reparameterization gradients}}, Advances in Neural Information Processing
  Systems 31 (NeurIPS), 2018. \eprintlink{1805.08498}

\bibitem[For21]{For21}
Patrick Forr\'e, \emph{{Quasi-measurable spaces: a convenient foundation of
  probability theory}}, preprint, 2021; revised version, 2026. All section
  pointers in the present paper refer to \texttt{v2}.
  \eprintlink{2109.11631}

\bibitem[Fri20]{Fri20}
Tobias Fritz, \emph{{A synthetic approach to Markov kernels, conditional
  independence and theorems on sufficient statistics}}, Advances in Mathematics
  \textbf{370} (2020), 107239. \doi{10.1016/j.aim.2020.107239}
  \eprintlink{1908.07021}

\bibitem[HKSY17]{HKSY17}
Chris Heunen, Ohad Kammar, Sam Staton, and Hongseok Yang, \emph{{A convenient
  category for higher-order probability theory}}, 32nd Annual ACM/IEEE
  Symposium on Logic in Computer Science (LICS 2017), IEEE, 2017, pp.~1--12.
  \doi{10.1109/LICS.2017.8005137} \eprintlink{1701.02547}

\bibitem[HLMS23]{HLMS23}
Mathieu Huot, Alexander~K. Lew, Vikash~K. Mansinghka, and Sam Staton,
  \emph{{$\omega$PAP spaces: reasoning denotationally about higher-order,
  recursive probabilistic and differentiable programs}}, 38th Annual ACM/IEEE
  Symposium on Logic in Computer Science (LICS 2023), IEEE, 2023, pp.~1--14.
  \doi{10.1109/LICS56636.2023.10175739} \eprintlink{2302.10636}

\bibitem[HSV20]{HSV20}
Mathieu Huot, Sam Staton, and Matthijs V\'ak\'ar, \emph{{Correctness of
  automatic differentiation via diffeologies and categorical gluing}},
  Foundations of Software Science and Computation Structures (FoSSaCS 2020),
  Lecture Notes in Computer Science, vol.~12077, Springer, 2020, pp.~319--338.
  \doi{10.1007/978-3-030-45231-5\_17} \eprintlink{2001.02209}

\bibitem[IZ13]{IZ13}
Patrick Iglesias-Zemmour, \emph{{Diffeology}}, Mathematical Surveys and
  Monographs, vol.~185, American Mathematical Society, Providence, RI, 2013.
  \doi{10.1090/surv/185}

\bibitem[JGP17]{JGP17}
Eric Jang, Shixiang Gu, and Ben Poole, \emph{{Categorical reparameterization
  with Gumbel-softmax}}, 5th International Conference on Learning
  Representations (ICLR), 2017. \eprintlink{1611.01144}

\bibitem[Kec95]{Kec95}
Alexander~S. Kechris, \emph{{Classical Descriptive Set Theory}}, Graduate Texts
  in Mathematics, vol.~156, Springer, New York, 1995.
  \doi{10.1007/978-1-4612-4190-4}

\bibitem[KM97]{KM97}
Andreas Kriegl and Peter~W. Michor, \emph{{The Convenient Setting of Global
  Analysis}}, Mathematical Surveys and Monographs, vol.~53, American
  Mathematical Society, Providence, RI, 1997. \doi{10.1090/surv/053}

\bibitem[KPS25]{KPS25}
Ohad Kammar, Seo~Jin Park, and Sam Staton, \emph{{Semantics for the coinductive
  structures in stochastic processes: infinite product measures in quasi-Borel
  spaces}}, Early Ideas extended abstract, CALCO 2025, Glasgow, 2025;
  \href{https://www.coalg.org/calco-mfps-2025/calco/accepted/}{\textsf{coalg.org}}.

\bibitem[KW14]{KW14}
Diederik~P. Kingma and Max Welling, \emph{{Auto-encoding variational Bayes}},
  2nd International Conference on Learning Representations (ICLR), 2014.
  \eprintlink{1312.6114}

\bibitem[Law07]{Law07}
F.~William Lawvere, \emph{{Axiomatic cohesion}}, Theory and Applications of
  Categories \textbf{19} (2007), no.~3, 41--49.
  \href{http://www.tac.mta.ca/tac/volumes/19/3/19-03abs.html}{\textsf{tac:19-03}}

\bibitem[LHSM23]{LHSM23}
Alexander~K. Lew, Mathieu Huot, Sam Staton, and Vikash~K. Mansinghka,
  \emph{{ADEV: sound automatic differentiation of expected values of
  probabilistic programs}}, Proceedings of the ACM on Programming Languages
  \textbf{7} (2023), no.~POPL, 121--153. \doi{10.1145/3571198}
  \eprintlink{2212.06386}

\bibitem[LRY23]{LRY23}
Wonyeol Lee, Xavier Rival, and Hongseok Yang, \emph{{Smoothness analysis for
  probabilistic programs with application to optimised variational inference}},
  Proceedings of the ACM on Programming Languages \textbf{7} (2023), no.~POPL,
  335--366. \doi{10.1145/3571205} \eprintlink{2208.10530}

\bibitem[Man76]{Man76}
Ernest~G. Manes, \emph{{Algebraic Theories}}, Graduate Texts in Mathematics,
  vol.~26, Springer, New York, 1976. \doi{10.1007/978-1-4612-9860-1}

\bibitem[MMS22]{MMS22}
Cristina Matache, Sean Moss, and Sam Staton, \emph{{Concrete categories and
  higher-order recursion: with applications including probability,
  differentiability, and full abstraction}}, 37th Annual ACM/IEEE Symposium on
  Logic in Computer Science (LICS 2022), ACM, 2022, pp.~57:1--57:14.
  \doi{10.1145/3531130.3533370} \eprintlink{2205.15917}

\bibitem[MMT17]{MMT17}
Chris~J. Maddison, Andriy Mnih, and Yee~Whye Teh, \emph{{The Concrete
  distribution: a continuous relaxation of discrete random variables}}, 5th
  International Conference on Learning Representations (ICLR), 2017.
  \eprintlink{1611.00712}

\bibitem[MRFM20]{MRFM20}
Shakir Mohamed, Mihaela Rosca, Michael Figurnov, and Andriy Mnih, \emph{{Monte
  Carlo gradient estimation in machine learning}}, Journal of Machine Learning
  Research \textbf{21} (2020), no.~132, 1--62. \eprintlink{1906.10652}

\bibitem[RMW14]{RMW14}
Danilo~J. Rezende, Shakir Mohamed, and Daan Wierstra, \emph{{Stochastic
  backpropagation and approximate inference in deep generative models}},
  Proceedings of the 31st International Conference on Machine Learning (ICML),
  Proceedings of Machine Learning Research, vol.~32, PMLR, 2014,
  pp.~1278--1286. \eprintlink{1401.4082}

\bibitem[Ros84]{Ros84}
Ji\v{r}\'i Rosick\'y, \emph{{Abstract tangent functors}}, Diagrammes
  \textbf{12} (1984), JR1--JR11.

\bibitem[Sou80]{Sou80}
Jean-Marie Souriau, \emph{{Groupes diff\'erentiels}}, Differential Geometrical
  Methods in Mathematical Physics, Lecture Notes in Mathematics, vol.~836,
  Springer, Berlin, 1980, pp.~91--128. \doi{10.1007/BFb0089728}

\bibitem[Ste21]{Ste21}
Dario Stein, \emph{{Structural foundations for probabilistic programming
  languages}}, D.Phil.\ thesis, University of Oxford, 2021.

\bibitem[Ste67]{Ste67}
Norman~E. Steenrod, \emph{{A convenient category of topological spaces}}, Michigan
  Mathematical Journal \textbf{14} (1967), no.~2, 133--152.
  \doi{10.1307/mmj/1028999711}

\bibitem[VKS19]{VKS19}
Matthijs V\'ak\'ar, Ohad Kammar, and Sam Staton, \emph{{A domain theory for
  statistical probabilistic programming}}, Proceedings of the ACM on
  Programming Languages \textbf{3} (2019), no.~POPL, 36:1--36:29.
  \doi{10.1145/3290349} \eprintlink{1811.04196}

\end{thebibliography}
\end{document}